\documentclass{article}

\usepackage{amsfonts, latexsym, amssymb, amsthm, amsmath, mathrsfs, esint, verbatim}
\usepackage{graphicx, pstricks, pst-plot, pst-math}
\usepackage{color}
\usepackage{authblk}

\newcommand{\R}{\mathbb{R}}

\newcommand{\Z}{\mathbb{Z}}
\newcommand{\C}{\mathbb{C}}
\newcommand{\Ss}{\mathbb{S}}

\newcommand{\be}{\begin{equation}}
\newcommand{\ee}{\end{equation}}

\newcommand{\eps}{\epsilon}

\renewcommand{\div}{\mathop{\mathrm{div}}\nolimits}
\DeclareMathOperator{\supp}{supp}
\renewcommand{\Re}{\mathop{\mathrm{Re}}}
\renewcommand{\Im}{\mathop{\mathrm{Im}}}
\DeclareMathOperator{\PV}{PV}
\DeclareMathOperator{\arcosh}{arcosh}
\DeclareMathOperator{\curl}{curl}
\newcommand{\loc}{\mathrm{loc}}

\newcommand\blank{{\mkern 2mu\cdot\mkern 2mu}}
\newcommand{\dd}[2]{\frac{\partial #1}{\partial #2}}
\newcommand{\set}[2]{\left\{ #1 \colon #2 \right\}}

\newtheorem{theorem}{Theorem}
\newtheorem{lemma}[theorem]{Lemma}
\newtheorem{proposition}[theorem]{Proposition}
\newtheorem{corollary}[theorem]{Corollary}
\newtheorem*{question}{Question}
\newtheorem*{op}{Open problem}

\newtheorem{remark}[theorem]{Remark}
\newtheorem{definition}[theorem]{Definition}

\title{Interaction energy of domain walls in a nonlocal Ginzburg-Landau type model from micromagnetics}

\author[1]{Radu Ignat}
\affil[1]{\small Institut de Math\'ematiques de Toulouse,
Universit\'e Paul Sabatier,
31062 Toulouse,
France. \authorcr
E-mail: Radu.Ignat@math.univ-toulouse.fr}
\author[2]{Roger Moser}
\affil[2]{\small Department of Mathematical Sciences,
University of Bath,
Bath BA2 7AY,
UK. \authorcr
E-mail: r.moser@bath.ac.uk}

\begin{document}

%\tableofcontents

\maketitle

\begin{abstract}
We study a variational model from
micromagnetics involving a nonlocal Ginzburg-Landau type energy for
$\Ss^1$-valued vector fields. These vector fields form
domain walls, called N\'eel walls, that correspond to one-dimensional
transitions between two directions within the unit circle $\Ss^1$.
Due to the nonlocality of the energy, a N\'eel wall is a two length scale object,
comprising a core and two logarithmically decaying tails. Our aim is to determine the
energy differences leading to repulsion or attraction between N\'eel walls.
In contrast to the usual Ginzburg-Landau vortices, we obtain a renormalised
energy for N\'eel
walls that shows both a tail-tail interaction and a core-tail interaction.
This is a novel feature for Ginzburg-Landau type energies that entails
attraction between N\'eel walls of the same sign and repulsion between
N\'eel walls of opposite signs.

\end{abstract}

\section{Introduction}

In this article, we analyse a variational model describing the formation of domain walls in ferromagnetic thin films. These domain walls are called N\'eel walls
and represent one-dimensional transition layers connecting two directions of the magnetisation within the unit circle $\Ss^1$.
Due to dipolar effects, the variational problem is strongly nonlocal and generates N\'eel walls with an interesting core-and-tail structure.
Our aim is to study the repulsive or attractive interaction between the domain walls
in terms of their energy.
This interaction energy governs the location of the domain walls and is analogous to
the renormalised energy in Ginzburg-Landau type problems (see the seminal book \cite{BBH}). Although 
our analysis builds to some extent on the theory of Ginzburg-Landau vortices, our model has novel features that have not been studied
before. In contrast to the usual Ginzburg-Landau vortices, we obtain a renormalised energy for N\'eel
walls incorporating two types of interaction: a tail-tail interaction and a core-tail interaction. This is
due to the nonlocal character of the model and the two distinct length scales of the core and the tails (of logarithmic decay). Moreover, N\'eel walls of opposite signs
repel each other and N\'eel walls of the same sign attract each other,
whereas Ginzburg-Landau vortices show the opposite behaviour. This
observation is consistent with the physical prediction (see \cite[Section 3.6.(C)]{HS98}).
Furthermore, in typical Ginzburg-Landau systems, most of the energy is contained in the
highest order term, whereas in our model, it is the lowest order term that contains most of the energy.
From a technical point of view, the lack of a quantised Jacobian is an additional difficulty in the analysis of our model.

\subsection{The model}
\label{sec:model}
\paragraph{The magnetisation} We consider a one-dimensional model for transition
layers (incorporating several N\'eel walls) in the magnetisation of
a thin ferromagnetic film. The magnetisation is represented
by a continuous map $$m : (-1, 1) \to \Ss^1.$$
More precisely, we can think of
a ferromagnetic thin film of the shape $(-1, 1) \times (0, h)  \times \R$
(with very small thickness $h > 0$ in the $x_2$-direction) and a magnetisation
vector field $M : (-1, 1)  \times (0, h) \times \R \to \Ss^2$
of the form $M(x_1, x_2, x_3) = (m(x_1), 0)$.
Here, the non-dependence of $M$ on $x_2$ is a natural assumption for a thin
film, whereas the non-dependence on $x_3$ represents a simplification
of the problem.
(It implies that the walls appear in planes parallel to the $x_2x_3$-plane
and we assume that
the magnetisation depends only on the normal direction $x_1$.)
The strip over $x_1\in (-1,1)$ does not necessarily represent the whole ferromagnetic
sample, but merely a region that contains the N\'eel walls in question.
The assumption that the third component $M_3$ vanishes is consistent
with the fact that N\'eel walls correspond to an in-plane
magnetisation. Another characteristic feature of N\'eel
walls is that the magnetisations on either side (represented
by $m(-1)$ and $m(1)$ in our model) differ by a vector parallel to the
wall plane (in this case the $x_2$-direction).
Thus there exists a number $\alpha \in (0, \pi)$ such
that 
\be
\label{BC}
m_1(-1) = m_1(1) = \cos \alpha.
\ee 
Moreover, we will sometimes assume that
\be
\label{BCper}
m(-1) = m(1) = (\cos \alpha, \sin \alpha),
\ee
so that a winding number can be defined.

%\bigskip

\begin{figure}[ht] %
  \vspace{5mm}
  \begin{minipage}{0.48\linewidth}
  \centering
  \psscalebox{0.9}{
  \begin{pspicture}(-2.3, -1.35)(2.7, 1.35) \psset{xunit=15mm, yunit=15mm}
    \psline[arrows=->, linewidth=0.01](-2.3, 0)(2.7, 0)
    \psline[linewidth=0.01](-2, -0.05)(-2, 0.05) \rput[t](-2, -0.15){$-1$}
    \psline[linewidth=0.01](2, -0.05)(2, 0.05) \rput[t](2, -0.15){$1$}
    
    \psline[arrows=->](-2, 0)(-1.567, 0.25)
    \psline[arrows=->](-1.5, 0)(-1.067, 0.129)
    \psline[arrows=->](-1, 0)(-0.5, 0)
    \psline[arrows=->](-0.4, 0)(0.033, -0.129)
    \psline[arrows=->](0.2, 0)(0.633, -0.25)
    \psline[arrows=->](0.8, 0)(0.671, -0.483)
    \psline[arrows=->](1.4, 0)(0.9, 0)
    \psline[arrows=->](1.7, 0)(1.571, 0.483)
    \psline[arrows=->](2, 0)(2.433, 0.25)
    
    \rput(-1.7, 0.35){$m$}
    \psarc[linewidth=0.01](-2, 0){0.4}{0}{30}
    \rput(-1.6, 0.09){$\alpha$}
    \psarc[linewidth=0.01](2, 0){0.4}{0}{30}
    \rput(2.4, 0.09){$\alpha$}
  \end{pspicture}
  }
  \end{minipage}
  \hfill
  \begin{minipage}{0.48\linewidth}
  \centering
  \psscalebox{0.9}{
    \begin{pspicture}(-1,-1.5)(4,1.2) \psset{xunit=1cm, yunit=1cm}
      \psaxes[labels=none, ticks=none]{->}(1.7, 0)(-2.2, -1.2)(5.2,1.6)
      \psline(1.6,1)(1.8,1) \rput[r](1.5,1){$1$} %

 \rput[r](1.6,1.6){$m_1$}
 \rput[r](5.5,-0.3){$x_1$}
 
        \psline(-1.7,0.34)(-1.5,0.34) \rput[r](1.65,0.54){$\cos \alpha$} %

        \psline(1.6,-1)(1.8,-1) \rput[r](1.5,-1){$-1$} %

       \psline(-1.6,-0.1)(-1.6,0.1) \rput[c](-1.6,-0.3){$-1$} %
       \psline(4.6,-0.1)(4.6,0.1) \rput[c](4.6,-0.3){$1$} %

        \pscustom{ %
          \psline(-1.6,0.34)(-1,0.34) \psplot[plotpoints=100]{-0.5}{-0.05}{x dup mul
            0.001 add -0.5 exp ln 0.2 mul 0.3 add} %
          \pscurve(-0.02,0.98)(0,1)(0.02,0.98) %
          \psplot[plotpoints=100]{0.05}{0.5}{x dup mul 0.001 add -0.5 exp ln 0.2 mul
            0.3 add} %
          \psline(1,0.34)(1.2,0.34)
          \psline(1.8,0.34)(2,0.34) \psplot[plotpoints=100]{2.5}{2.95}{x -3 add dup
            mul 0.001 add -0.5 exp ln -0.4 mul 0.3 add} %
          \pscurve(2.98,-0.98)(3,-1)(3.02,-0.98) %
          \psplot[plotpoints=100]{3.05}{3.5}{x -3 add dup mul 0.001 add -0.5 exp ln
            -0.4 mul 0.3 add} %
          \psline(4,0.34)(4.6,0.34) }
      \end{pspicture}
    }
  \end{minipage} 
  \vspace{-5mm}
    \caption{A magnetisation $m=(m_1, m_2)$ of winding number $-1$ consisting of a positive N\'eel wall of angle $2\alpha$ and a negative N\'eel wall of angle $2(\pi-\alpha)$ (right).}
  \label{fig:neel}
\end{figure}
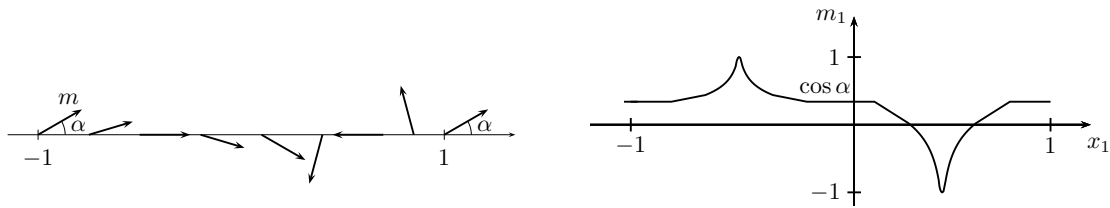

More precisely, since $m$ is continuous, there exists a
continuous function $\varphi : (-1, 1) \to \R$, called a lifting of $m$,
such that $$m = (\cos \varphi, \sin \varphi)\quad \textrm{in}\quad (-1,1)$$ and $\varphi(-1) = \alpha$. If \eqref{BCper} holds, then the winding number (or topological degree) of $m$ is defined as
$$\deg(m)=\frac{\varphi(1)-\varphi(-1)}{2\pi}\in \Z.$$
The angle $\alpha\in (0, \pi)$
will stay fixed throughout this paper. (The case $\alpha\in \{0, \pi\}$ is geometrically different and is not studied here.)
However, our arguments do not require that $m_2(-1) = m_2(1) = \sin \alpha$
in principle and we will present our results in a wider generality, i.e.,
with $m_2(\pm 1)\in \{\pm \sin \theta\}$. 

\paragraph{The energy} The energy for our model comprises
two terms, called the exchange energy and the magnetostatic energy (or stray-field energy),
respectively. The exchange energy is modelled by the
following expression involving the $L^2$-norm of the derivative $m'$:
\[
\frac{\epsilon}{2} \int_{-1}^1 |m'|^2 \, dx_1 = \frac{\epsilon}{2} \int_{-1}^1 (\varphi')^2 \, dx_1 = \frac{\epsilon}{2} \int_{-1}^1 \frac{(m_1')^2}{1 - m_1^2} \, dx_1.
\]
Here $\epsilon > 0 $ is a ratio between a material constant called the exchange length
and the length scale of the thin film. (This is a model
obtained after rescaling, i.e., the length scale of the ferromagnetic sample
has been set to unit size.) The
number $\epsilon$ is assumed to be small, and we will eventually study the
limit $\epsilon \searrow 0$.

We write $x = (x_1, x_2)$ for a generic point in the upper
half-plane $\R_+^2 = \R \times (0, \infty)$.
In order to compute the magnetostatic energy,
we need to solve
the boundary value problem\footnote{Here, $\nabla u$ represents the stray-field associated to $M$, which is also invariant in the $x_3$-direction.}
\begin{alignat}{2}
\Delta u & = 0 &\quad& \text{in } \R_+^2, \label{eqn:u_harmonic} \\
\dd{u}{x_2} & = -m_1' && \text{on } (-1, 1) \times \{0\}, \label{eqn:u_boundary1} \\
\dd{u}{x_2} & = 0 && \text{on } (-\infty, -1) \times \{0\} \text{ and on } (1, \infty) \times \{0\}. \label{eqn:u_boundary2}
\end{alignat}
Equivalently, if we extend $m_1$ by the constant $\cos \alpha$ on $\R\setminus (-1,1)$, then \be
\label{weak_stray}
\int_{\R^2_+}\nabla u\cdot \nabla \zeta\, dx=\int_{-\infty}^\infty m_1'\zeta(\blank, 0)\, dx_1 \quad \textrm{for every } \zeta\in C^\infty_0(\R^2).
\ee
Let $\dot{W}^{1, 2}(\R^2)$ be the completion of $C_0^\infty(\R^2)$
with respect to the norm
\[
\|\zeta\|_{\dot{W}^{1, 2}(\R^2)} = \|\nabla \zeta\|_{L^2(\R^2)}.
\]
(We sometimes abuse notation and
treat elements of $\dot{W}^{1, 2}(\R^2)$ as functions, even though
the completion process identifies any two functions that differ by a constant.)
For an open set $\Omega \subset \R^2$, we write
$\dot{W}^{1, 2}(\Omega)$ for the set of all restrictions
of functions in $\dot{W}^{1, 2}(\R^2)$ to $\Omega$ and
\[
\|\zeta\|_{\dot{W}^{1, 2}(\Omega)} = \|\nabla \zeta\|_{L^2(\Omega)}.
\]
By the Lax-Milgram theorem, solutions of \eqref{weak_stray} are
unique in $\dot{W}^{1, 2}(\R_+^2)$ (i.e., up to a constant). Thus the
quantity
\[
\frac{1}{2} \int_{\R_+^2} |\nabla u|^2 \, dx
\]
depends only on $m_1$. This is the term representing the magnetostatic energy.
It is worth remarking that the solutions $u$ of
\eqref{eqn:u_harmonic}--\eqref{eqn:u_boundary2}
in $\dot{W}^{1, 2}(\R_+^2)$ have a limit for $|x| \to \infty$.
Indeed, if we extend $u$ to $\R^2$ by even reflection, then we obtain
a harmonic function near $\infty$ with finite Dirichlet
energy, and it is well-known that the limit exists at $\infty$.
Then we normalise this constant and define $U(m)$
(sometimes also denoted $U(m_1)$) to be the unique solution
of \eqref{weak_stray} in $\dot{W}^{1, 2}(\R_+^2)$
with\label{def:U}
\[
U(m) \to 0 \, \, \textrm{ as } \, \, |x| \to \infty.
\]
Moreover, in view of
\eqref{weak_stray}, using the extension of 
$m_1$ by the constant $\cos \alpha$ on $\R\setminus (-1,1)$, we may
express the magnetostatic energy in terms of the homogeneous $\|\cdot\|_{\dot{H}^{1/2}}$-seminorm of $m_1$ (see e.g. \cite{DKMOreduced, Ig}):
\be
\label{stray_uniq}
\frac{1}{2} \int_{\R_+^2} |\nabla U(m)|^2 \, dx=\frac 1 2 \int_{\R} \left|\left|\frac{d}{dx_1}\right|^{1/2}m_1\right|^2\, dx_1.
\ee

To summarise, we study the energy functional
\[
E_\epsilon(m) = \frac{\epsilon}{2} \int_{-1}^1 |m'|^2 \, dx_1 + \frac{1}{2} \int_{\R_+^2} |\nabla U(m)|^2 \, dx
\]
for $m \in W^{1, 2}\big((-1, 1), \Ss^1\big)$ satisfying \eqref{BC}. We are interested in the behaviour of $m$ and of its energy $E_\eps(m)$ as $\epsilon \searrow 0$,
especially under conditions that force the nucleation of several
N\'eel walls.

\paragraph{N\'eel walls} If we trace $m$ from $-1$ to $1$, we may well find that
$m$ winds around the circle $\Ss^1$ one or several times. If that happens,
then there necessarily exist two points $a_+, a_- \in (-1, 1)$
such that $m_1(a_+) = 1$ and $m_1(a_-) = -1$. But even if the
topology of $m$ is trivial (i.e., if $\deg(m) = 0$), a transition from
$(\cos \alpha, \sin \alpha)$ to $(\cos \alpha, - \sin \alpha)$
may occur, giving rise to
a point in between where $m_1$ reaches one of the values $\pm 1$. We think of any such transition as a N\'eel wall and
we use these points in order to track them. Obviously, it
is possible for $m_1$ to attain $\pm 1$ when no proper transition
occurs, but from the energetics point of view, this makes
no difference and we call this a N\'eel wall anyway. 
We speak of
a positive or negative N\'eel wall depending on the sign of $m_1$ (see Figure \ref{fig:neel}). 
We will see that a N\'eel wall has a two-length scale structure comprising a core of size $\delta=\eps \log \frac 1 \eps$ around the transition point and two tails of size $O(1)$,
where $m_1$ decays logarithmically to $\cos \alpha$ (see Theorem~\ref{thm:core_convergence} below).
The total change of the phase during the transition is called the rotation
angle of the N\'eel wall (which may be $0$ by the above convention).\footnote{When
studying the interaction between a pair of walls, the physics literature (see \cite{HS98}) distinguishes between winding walls, which refers to a pair with the same rotation
sense, and unwinding walls, which refers to a pair with opposite rotation sense.
Except for degenerate cases, a pair of N\'eel walls with opposite signs
according to our terminology corresponds to winding walls and a pair with
the same sign corresponds to unwinding walls.} 
For more physical background, we refer to \cite{HS98, DKMO05}.

We will assume in the following that there are certain points $a_1, \ldots, a_N \in (-1, 1)$
such that
\begin{equation} \label{eqn:ordering}
-1 < a_1 < \cdots < a_N < 1
\end{equation}
and certain numbers $d_1, \ldots, d_N \in \{-1, 1\}$ such that
\be
\label{points}
m_1(a_n) = d_n \quad \textrm{for} \quad n = 1, \ldots, N.\ee
These points $(a_n)_{1\leq n \leq N}$ represent the positions of the N\'eel walls
that we study, while $(d_n)_{1\leq n \leq N}$ indicate
whether a N\'eel wall is positive or negative. 
We keep the number $N$ of walls fixed
throughout the paper. Let 
\[
A_N  =\left\{a = (a_1, \ldots, a_N) \in (-1, 1)^N \text{ with } \eqref{eqn:ordering}\right\}.
\]
For $a \in A_N$ and $d \in \{\pm 1\}^N$,
we consider the set $$M(a, d)=\left\{m \in W^{1,2}((-1, 1); \Ss^1) \textrm{ with } \eqref{BC}  \textrm{ and } \eqref{points} \right\}.$$ 

Our aim is to answer the following question.

\begin{question} \label{question1}
For a given $a \in A_N$ and $d \in \{\pm 1\}^N$, what is the behaviour of
$$\inf_{M(a, d)} E_\epsilon \quad \textrm{ as } \quad \epsilon \searrow 0?$$
\end{question}

That is, if we prescribe N\'eel walls at the positions
$a_1, \ldots, a_N$ with signs $d_1, \ldots, d_N$,
what energy does it take to achieve such a configuration?
We first note that a minimal configuration $m$ always
exists and that its first component $m_1$ is unique.
(Obviously, $|m_2|$ is also unique, but the sign
of the $m_2$ component can change between $a_n$ and $a_{n+1}$
for two different minimisers.)

\begin{proposition}
\label{pro:exist_unique}
There exists a minimiser of 
$\inf_{M(a, d)} E_\epsilon$ for any $\eps>0$. Moreover, any minimiser $m$ is smooth on $(-1,1)\setminus \{a_1, \dots, a_N\}$ and has a unique $m_1$-component.
\end{proposition}

\begin{proof}
The direct method in the calculus of variations yields a minimiser $m$ of $E_\eps$ in $M(a,d)$. The regularity of $m$ is standard (see, e.g., \cite{Ignat_Knupfer}). The uniqueness of $m_1$ follows from the strict convexity of \eqref{stray_uniq} and of the function $(v,w) \mapsto \frac{v^2}{1-w^2}$ for $(v, w) \in \R\times (-1,1)$.
\end{proof}

We look for an expansion of $\inf_{M(a, d)} E_\epsilon$ similar to \cite{DKMO_rep},
where it is shown that
\be
\label{DKMO_expansion}
\inf_{M(a, d)} E_\epsilon = \sum_{n = 1}^N \frac{\pi (d_n - \cos \alpha)^2}{2 \log \frac{1}{\epsilon}} + O\left(\frac{\log \log \frac{1}{\epsilon}}{\left(\log \frac{1}{\epsilon}\right)^2}\right)
\ee
for $\epsilon > 0$ small. Since this is not good enough
to understand the interaction between domain walls, we need to determine the second term in such an expansion completely and identify the third term as well.
This problem is analogous to finding the ``renormalised energy''
in Ginzburg-Landau type problems, but in the context of N\'eel walls, it has
remained open until now.

We give the answer to this question in Theorem \ref{thm1}.
The key is to identify the contributions to the renormalised energy coming from
the interaction between two tails and between a core and a tail of two
different walls. It turns out that the above expansion is easier
to understand when we replace $\epsilon$ by $\delta = \epsilon \log \frac{1}{\epsilon}$
(recall that this is the typical length scale of the core
of a N\'eel wall).
The first two terms of the expansion \eqref{DKMO_expansion} are
then united in a single leading order term in the expansion in
$1/|\log \delta|$ (see \eqref{DKMO_improved} below). The next-to-leading order term corresponds
to the renormalised energy.

\subsection{Motivation} There are several reasons for asking the above question.
First, we may want to study the positions of N\'eel walls in
equilibrium. Once we have determined the renormalised
energy, we can find the likely positions by minimising it. Second,
we may want to study the dynamics of N\'eel walls (see, e.g.,
\cite{CMO07, Cote-Ignat-Miot}).
The dynamics of the magnetisation is described by the
Landau-Lifshitz-Gilbert equation, which is derived from
the micromagnetic energy through a variational principle.
For reasons that are explained below, understanding the
asymptotic behaviour of the energy is expected to be a
key step towards deriving an effective motion law for
the walls in the limit $\epsilon \searrow 0$.
A third reason for studying these long range interactions is
that we want to understand some phenomena in thin ferromagnetic films
where they matter,
such as cross-tie walls. A cross-tie wall is a typical domain
wall that consists in an ensemble of N\'eel walls and micromagnetic
vortices (similar to Ginzburg-Landau vortices), see
\cite{DKMO_rep, ARS02, RS01, RS03}. It has an internal length scale,
the size of which is not predicted by any existing theory,
and our analysis on the interaction energy of N\'eel walls
could represent an significant step forward here. 

A related question concerns the analysis of general transition layers $m$
carrying a winding number when the location of the N\'eel walls
is no longer prescribed. More precisely, suppose that the lifting
$\varphi : (-1, 1) \to \R$ of $m = (\cos \varphi, \sin \varphi)$
satisfies the boundary conditions
$\varphi(-1) = \alpha$ and $\varphi(1) = 2\ell \pi + \alpha$,
so that we have winding number $\ell$, i.e. $\ell=\deg(m)$. Hence
the magnetisation performs $\ell$ full rotations, so that \eqref{BCper} is satisfied.
Then by continuity, we necessarily have a certain number
of transitions between $(\cos \alpha, \sin \alpha)$ and
$(\cos \alpha, - \sin \alpha)$.

\begin{op} 
For a prescribed winding number and given suitable control of $E_\epsilon(m)$,
what can we say about the profile of $m$ and of the stray field potential
$U(m)$? 
\end{op}

As mentioned before, a prescribed degree $\ell$ will automatically give
rise to certain N\'eel walls. But it is not obvious, for example,
that these N\'eel walls stay separate from one another (uniformly as $\eps\to 0$) and that one can rule out other transitions. In fact, it is an open 
question whether the lifting of $m$ is monotone even for minimisers (which would exclude unexpected transitions). 
However, assuming good control of the energy, we expect to have exactly
$2\ell$ transitions (corresponding to the expected N\'eel walls) and no
extraordinary behaviour of the magnetisation in between.
For the stray field energy, it is expected
that the energy density concentrates at the walls. Such information would be
useful in the study of compactness properties
in the appropriate function spaces, for example with a view
to $\Gamma$-convergence.

\subsection{Main results} \label{sect:result}

For any $\epsilon \in (0, \frac{1}{2}]$, let $$\delta = \epsilon \log \frac{1}{\epsilon}$$
and define the metric $\varrho$ on $(-1, 1)$ by\footnote{We will not
use the fact that $\varrho$ is a metric, but if we want to
verify it, we can use that $\varrho(\Phi_d(b), \Phi_d(c)) = \varrho(b, c)$
for the M\"obius transforms $\Phi_d$ defined in \eqref{Mobius} below for every $d\in (-1,1)$. 
For the triangle inequality, it then suffices to show that
$\varrho(c, 0) \le \varrho(b, 0) + \varrho(b, c)$ for $b, c \in (-1, 1)$,
which is not difficult.}
\[
\varrho(b, c) = \frac{|b - c|}{1 - bc}\in [0,1)\quad \textrm{for} \quad b, c\in (-1,1).
\]

We have the following result, answering the question on page \pageref{question1}.

\begin{theorem} \label{thm1}
There exists a function $e : \{\pm 1\} \to \R$ such that for any
$a \in A_N$ and $d \in \{\pm 1\}^N$, the following holds true.
Let $\gamma_n = d_n - \cos \alpha$ for $n = 1, \ldots, N$ and let
\[
\mathbb{W}(a,d) = \sum_{n = 1}^N e(d_n) - \frac{\pi}{2} \sum_{n = 1}^N \gamma_n^2 \log(2 - 2a_n^2) 
- \frac{\pi}{2} \sum_{n = 1}^N \sum_{k \not= n} \gamma_k \gamma_n \log \left(\frac{1 + \sqrt{1 - \varrho(a_k, a_n)^2}}{\varrho(a_k, a_n)}\right).
\]
Then
\[
\mathbb{W}(a,d) = \lim_{\epsilon \searrow 0} \left(\left(\log \frac{1}{\delta}\right)^2 \inf_{M(a, d)} E_\epsilon  - \frac{\pi}{2} \log \frac{1}{\delta}  \sum_{n = 1}^N \gamma_n^2\right).
\]
\end{theorem}

In analogy to the theory of Ginzburg-Landau vortices, we
call $\mathbb{W}(a,d)$ the renormalised energy for the $N$ walls placed
at $a=(a_1, \dots a_N)$ with signs $d=(d_1, \dots, d_N)$.
As the theorem shows, $\mathbb{W}(a,d)$ represents the next-to-leading order term
in the expansion of $\inf_{M(a, d)} E_\epsilon$ in $1/|\log \delta|$. If we express these asymptotics in terms of $\eps$, our result improves \eqref{DKMO_expansion} by determining the precise second and third coefficients:
\be
\label{DKMO_improved}
\inf_{M(a, d)} E_\epsilon =\frac{1}{\left(\log \frac{1}{\epsilon}\right)^2} \left(\frac 12\sum_{n = 1}^N \pi (d_n - \cos \alpha)^2 \left(\log \frac{1}{\epsilon} + \log \log \frac{1}{\epsilon}\right)+\mathbb{W}(a,d)\right) 
+o\left( \frac{1}{\left(\log \frac{1}{\epsilon}\right)^2} \right).\ee

We now briefly discuss how the above expression comes
about. Suppose that for a given $a \in A_N$, we study minimisers $m$ of $E_\epsilon$
in $M(a, d)$. When $\epsilon$ is small, we expect to have a
typical N\'eel wall profile near each of the points
$a_1, \ldots, a_N$ with the prescribed signs $d_1, \ldots, d_N$, and the full transition layer $m$ is essentially a superposition
of all of these. As discussed previously, we can think of a N\'eel wall as consisting
of two parts: a small core around $a_n$ and two logarithmically
decaying tails. In our situation, the walls are confined in
the relatively short interval $(-1, 1)$ and each tail will
interact with the other walls and with the boundary as well.
We can then account for the full energy $\inf_{M(a, d)} E_\epsilon$ (at leading
and next-to-leading order) as follows.

\begin{description}
\item[Core energy.]
The core of each wall requires a certain amount of energy,
namely
\[
\frac{e(1)}{\left(\log \frac{1}{\delta}\right)^2} \quad \text{and} \quad \frac{e(-1)}{\left(\log \frac{1}{\delta}\right)^2}
\]
for a positive and a negative wall, respectively. The constants $e(\pm1)$ are given in Definition~\ref{def:e} below
as limits of a rescaled energy of the core profile as $\eps\to 0$. Then the
sum accounts for the term
\[
\frac{\sum_{n = 1}^N e(d_n)}{\left(\log \frac{1}{\delta}\right)^2}.
\]
This is the only term where we have a contribution from
the exchange energy and it appears only at next-to-leading
order in the full energy. All the remaining terms below come
from the magnetostatic energy alone.

\item[Tail energy.]
The two tails of the wall at $a_n$ give rise to the energy
\[
\frac{\pi \gamma_n^2}{2\log \frac{1}{\delta}},
\]
leading to a total of
\[
\frac{\pi \sum_{n = 1}^N \gamma_n^2 }{2\log \frac{1}{\delta}}.
\]
This is the leading order term of the full energy.

\item[Tail-boundary interaction.]
Moving a wall relative to the boundary points $\pm 1$ will
deform the tail profile, resulting in a change of the energy.
This phenomenon gives rise to the energy
\[
\frac{\pi \gamma_n^2 \log(2 - 2a_n^2)}{2\left(\log \frac{1}{\delta}\right)^2}
\]
for the wall at $a_n$. Summing up these contributions, we obtain
\[
\frac{\pi}{2\left(\log \frac{1}{\delta}\right)^2} \sum_{n = 1}^N \gamma_n^2 \log(2 - 2a_n^2).
\]
(The sign here is not a mistake; it is the opposite of the sign of
the corresponding expression in Theorem \ref{thm1}.) This means
that the tails are attracted by the boundary, in the sense
that the energy decreases if $a_n$ approaches $\pm 1$.

\item[Tail-tail interaction.]
There is an energy contribution coming from reinforcement or cancellation
between the stray fields generated by different walls.
For the walls at $a_k$ and $a_n$ with $k \not= n$, this
amounts to
\[
\frac{\pi \gamma_k \gamma_n}{2\left(\log \frac{1}{\delta}\right)^2} \log \left(\frac{1 + \sqrt{1 - \varrho(a_k, a_n)^2}}{\varrho(a_k, a_n)}\right).
\]
The total contribution is
\[
\frac{\pi}{2\left(\log \frac{1}{\delta}\right)^2} \sum_{n = 1}^N \sum_{k \not= n} \gamma_k \gamma_n \log \left(\frac{1 + \sqrt{1 - \varrho(a_k, a_n)^2}}{\varrho(a_k, a_n)}\right).
\]
(Again we have the opposite sign relative to the above theorem.)
A conclusion is that the tails of two walls attract each other
if they have opposite signs and repel each other if they have the same sign.\footnote{This is because the function $t\mapsto \frac{1 + 
\sqrt{1 -t^2}}{t}$ is decreasing on $(0,1)$.}

\item[Tail-core interaction.]
Since the profile of a N\'eel wall decays only logarithmically,
it will change the turning angle of the neighbouring walls
slightly. This has an effect on the energy as well (at the next-to-leading order).
Indeed, the
tail of the wall at $a_k$ and the core of the wall at $a_n$ with $k \not= n$ lead to a contribution of
\[
- \frac{\pi \gamma_k \gamma_n}{\left(\log \frac{1}{\delta}\right)^2} \log \left(\frac{1 + \sqrt{1 - \varrho(a_k, a_n)^2}}{\varrho(a_k, a_n)}\right).
\]
We also have an interaction between the two tails of a wall
and its own core: if $k = n$, then we obtain the energy
\[
{-}\frac{\pi \gamma_n^2  \log(2 - 2a_n^2)}{\left(\log \frac{1}{\delta}\right)^2}.
\]
This gives a total of
$$
-\frac{\pi}{\left(\log \frac{1}{\delta}\right)^2} \sum_{n = 1}^N \Biggl(\gamma_n^2 \log(2 - 2a_n^2) 
+ \sum_{k \not= n} \gamma_k \gamma_n \log \left(\frac{1 + \sqrt{1 - \varrho(a_k, a_n)^2}}{\varrho(a_k, a_n)}\right)\Biggr).
$$
This is twice the size of the terms from the tail-boundary interaction and
tail-tail interaction, but with the
opposite signs, resulting in a net repulsion between walls of
opposite signs and a net attraction between walls of the same sign.
Furthermore, we have a net repulsion of the walls by the boundary.
\end{description}
Notwithstanding the term `energy' used in this description,
strictly speaking, these are energy differences and
therefore some of them may be negative. All except one of these contributions occur similarly in the theory of Ginzburg-Landau vortices. The core-tail interaction, 
on the other hand, is new and more delicate to handle.

\subsection{Physical relevance} Our result represents a rigorous proof of the physical prediction on the interaction energy between N\'eel walls. Indeed, Hubert and Sch\"afer (\cite[Section 3.6. (C)]{HS98}) predict the following behaviour in the case of a pair
of N\'eel walls: {\it ``The extended tails of N\'eel walls lead to strong interactions
between them [...] The interactions become important as soon as the tail regions overlap.
The sign of the interaction depends on the wall rotation sense. N\'eel walls of opposite rotation sense (so-called unwinding walls) attract each other because they generate
opposite charges in their overlapping tails. If they are not pinned, they can annihilate.
N\'eel walls of equal rotation sense (winding walls) repel each other.''}
(We recall that unwinding walls
correspond---according to our definition in Section \ref{sec:model}---to a pair
of N\'eel walls with the same sign, while winding walls correspond to a pair of
walls with the opposite signs as in Figure \ref{fig:neel}.)

\subsection{Comparison with a linear model}

If we replace the exchange energy by the simpler expression
\[
\frac{\epsilon}{2} \int_{-1}^1 (m_1')^2 \, dx_1,
\]
then the energy functional, now given by
\[
\tilde{E}_\epsilon(m_1) = \frac{\epsilon}{2} \int_{-1}^1 (m_1')^2 \, dx_1 + \frac{1}{2} \int_{\R_+^2} |\nabla U(m_1)|^2 \, dx, \quad m_1:(-1,1)\to \R,
\]
becomes a quadratic form and the Euler-Lagrange equation
for its critical points becomes linear. This functional
has been used as a tool for studying the energy of N\'eel
walls \cite{DKMO_rep, Ignat_Knupfer}.
Since the exchange energy in $E_\eps$ does not enter the expansion \eqref{DKMO_expansion}
at the leading order, we may expect good approximations from the linear model involving the energy
functional $\tilde{E}_\eps$. The exchange energy does have an effect on the
next-to-leading order term however, even though it is not
through a direct contribution but rather by changing the
core width of a domain wall. For the linear model, the core width of a domain wall is of order
$\epsilon$. Accordingly, for the functional $\tilde{E}_\epsilon$,
the expansion that corresponds to \eqref{DKMO_expansion} is of the form
\[
\inf_{m \in M(a, d)} \tilde{E}_\epsilon(m_1) = \sum_{n = 1}^N \frac{\pi (d_n - \cos \alpha)^2}{2\log \frac{1}{\epsilon}} + \frac{\tilde{\mathbb{W}}(a, d)}{\left(\log \frac{1}{\epsilon}\right)^2} + o\left(\frac{1}{\left(\log \frac{1}{\epsilon}\right)^2}\right)
\]
as $\epsilon \searrow 0$. Here, $\tilde{\mathbb{W}}$ is nearly the
same as the function $\mathbb{W}$ from Theorem \ref{thm1},
except that it may differ by a number depending only on $N$ and $d$.
That is, there exists a function $\tilde{e} : \{\pm 1\} \to \R$ such
that
\begin{multline*}
\tilde{\mathbb{W}}(a, d) = \sum_{n = 1}^N \tilde{e}(d_n) - \frac{\pi}{2} \sum_{n = 1}^N \gamma_n^2 \log (2 - 2a_n^2)
- \frac{\pi}{2} \sum_{n = 1}^N \sum_{k \not= n} \gamma_k \gamma_n \log \left(\frac{1 + \sqrt{1 - \varrho(a_k, a_n)^2}}{\varrho(a_k, a_n)}\right).
\end{multline*}
As for the full model, we may regard $\tilde{e}(\pm 1)$ as the core
energy of a transition of sign $\pm 1$. Our analysis does not give
an explicit expression, but for variational principles where
the number and signs of the N\'eel walls does not change, this
part of the limiting energy is irrelevant.

The formula can be proved with the same arguments as in the proof of
Theorem \ref{thm1} below, although the linearity allows a few shortcuts.
Therefore, we do not give a separate proof but leave it to the reader
to make the necessary changes.

As a consequence, the linear model does not describe the interaction
between N\'eel walls accurately, but the discrepancy is easily
corrected by adjusting the core width (i.e., replacing $\epsilon$
by $\delta$). Although we study only the energy of interacting
N\'eel walls in this paper, the analogy to the theory of
Ginzburg-Landau vortices (see Sect. \ref{sect:Ginzburg-Landau})
suggests that the same may be true for the dynamics of N\'eel
walls. The simplified model may therefore be useful as a test
case for future analysis, or, with the necessary care and the
appropriate corrections, even be used for quantitative predictions.

\subsection{Comparison with Ginzburg-Landau vortices} \label{sect:Ginzburg-Landau}

The interaction between topological singularities has been intensively
studied in the last two decades in the context of Ginzburg-Landau
problems. The work was pioneered by Bethuel, Brezis, and H\'elein
\cite{Bethuel-Brezis-Helein:93, BBH}, and an overview of later
developments can be found in a book by Sandier and Serfaty \cite{SSbook}.
These problems are designed to describe
phenomena in superconductors and Bose-Einstein condensates, and
a simple model that captures some of the main features is based
on the functionals
\begin{equation} \label{eqn:Ginzburg-Landau}
G_\epsilon(f) = \int_\Omega \left(\frac{1}{2} |\nabla f|^2 + \frac{1}{4\epsilon^2} (1 - |f|^2)^2\right) \, dx
\end{equation}
for a domain $\Omega \subset \R^2$ and a function $f : \Omega \to \R^2$.
We identify $\R^2$ with $\C$. Then in the limit $\epsilon \searrow 0$,
the analysis in the aforementioned papers leads to a limiting
function $f : \Omega \to \C$ of the form
\[
f(z) = e^{i \theta(z)} \prod_{n = 1}^N \left(\frac{z - a_n}{|z - a_n|}\right)^{d_n}
\]
for certain points $a_1, \ldots, a_N \in \Omega$, integers
$d_1, \ldots, d_N \in \Z \backslash \{0\}$, and a function
$\theta : \Omega \to \R$. The renormalised energy
\[
\liminf_{r \searrow 0} \left(\frac{1}{2} \int_{\Omega \backslash \bigcup_{n = 1}^N B_r(a_n)} |\nabla f|^2\, dx - {\pi} \log \frac{1}{r} \sum_{n = 1}^N d_n^2\right)
\]
appears in a result similar to Theorem \ref{thm1} (together with an
additional term describing the core energy).
Here $B_r(a)$ stands for the open ball of radius $r$ and centre $a$.
We have topological singularities at $a_1, \ldots, a_n$ with vortex
structures and with topological degrees $d_1, \ldots, d_n$.
These data are encoded in the distributional Jacobian
\[
J(f) = \frac{1}{2} \curl (f^\perp \cdot \nabla f),
\]
where $f^\perp=(-f_2, f_1)$.

The renormalised energy gives information about the vortex positions
in equilibrium, but it is also important for their dynamics. Typically,
if $f$ evolves by a variational equation derived from $G_\epsilon$, then
on an appropriate time scale, the limiting motion law for the vortices
(as $\epsilon \searrow 0$) is described by an analogous equation derived
from the renormalised energy. This is true for gradient flows
\cite{MR1376654, MR1376655, MR1629646}, Schr\"odinger type equations
\cite{MR1623410, MR1699965}, as well as nonlinear wave equations
\cite{MR1676761, MR1710937}.

It has been observed before that certain phenomena from micromagnetics
give rise to similar models
\cite{Hang-Lin:01, Kurzke-Melcher-Moser-Spirn, Kurzke-2006, Moser, Moser_ARMA}.
The connection to our model is less obvious, but can be seen once
we show that under assumptions such as in Theorem \ref{thm1},
we obtain a limiting function from the rescaled stray field potential $(\log \frac 1 \delta) U(m)$
of the form
\[
u_{a, d}^*(x) = u_*(x) + \sum_{n = 1}^N \gamma_n \left(\arctan\left(\frac{x_2}{x_1 - a_n}\right) - \frac{\pi (x_1 - a_n)}{2|x_1 - a_n|}\right)
\]
for some $a \in A_N$ and $d\in \{\pm 1\}^N$ and a harmonic function $u_* : \R_+^2 \to \R$
that is smooth near $(-1, 1) \times \{0\}$ (see Sect. \ref{sect:reduced}
for details). Examining $u_{a, d}^*$ near the point $(a_n, 0)\in \R^2$,
we see that it behaves like the phase of a vortex in the upper
half-plane, up to the constant $\gamma_n$. The expression
\[
\liminf_{r \searrow 0} \left(\frac{1}{2} \int_{\R_+^2 \backslash \bigcup_{n = 1}^N B_r(a_n, 0)} |\nabla u_{a, d}^*|^2 \, dx - \frac{\pi}{2} \log \frac{1}{r} \sum_{n = 1}^N \gamma_n^2\right)
\]
also plays a role, although for our problem, it only
accounts for a part of renormalised energy in Theorem \ref{thm1}
(even after adding the core energy).
In fact, a N\'eel wall
at $a_n$ behaves like a vortex of ``degree'' $\gamma_n$ in many
respects, which is why the toolbox from the theory of
Ginzburg-Landau vortices is very useful for the analysis.

There are, however, significant differences to Ginzburg-Landau vortices
as well. A N\'eel wall is a two-length scale object, comprising
a core and two tails, each with its own characteristic length.
In contrast, in the standard Ginzburg-Landau problem, a vortex
has a single length scale characterising its core and the renormalised
energy between the vortices comes essentially from the interaction
of its out-of-core structure. For N\'eel walls, we have a
renormalised energy consisting of two parts. The interaction between
the tails of two walls is similar to the interaction between
Ginzburg-Landau vortices and gives rise to the above expression.
But in addition, we have an interaction between the core of one
wall and the tail of another, which is a novel feature for
Ginzburg-Landau type systems. This interaction is responsible for
the fact that we have attraction for walls of the same sign and
repulsion in the case of opposite signs, whereas for Ginzburg-Landau
vortices, we have attraction for degrees of opposite signs and
repulsion for degrees of the same sign.
Finally, our ``degree'' $\gamma_n$ is not quantised in the same
way as the degree of Ginzburg-Landau vortices. It does take
only two values ($\pm 1 - \cos \alpha$), but these depend on the
choice of the angle $\alpha$ and are not topological invariants. As a
consequence, the Jacobian becomes a much less powerful tool.
To overcome these difficulties, we use duality arguments and
``logarithmically failing" interpolation inequalities (see \cite{DKO, IO}),
$\Gamma$-convergence methods, and refined elliptic estimates.

In our model, the magnetostatic energy, being of order $O(1/|\log \delta|)$,
dominates the higher order exchange energy (of order $O(1/(\log \delta)^2)$
for small values of $\epsilon$. This is in contrast to most
Ginzburg-Landau systems, where the highest order term is dominant.
This is the case for the functionals $G_\epsilon$ in \eqref{eqn:Ginzburg-Landau},
but also for similar models coming from micromagnetics. For example,
the model for boundary vortices studied by Kurzke \cite{Kurzke-2006}
contains a term coming from the exchange energy and one coming from the
magnetostatic energy as well, but in the analysis, the roles of the
two are reversed relative to the model for N\'eel walls.

\subsection{Notation}

We now introduce some notation that we will use frequently
throughout this paper. As mentioned previously, we define
$\delta = \epsilon \log \frac{1}{\epsilon}$.
This is the scale of the N\'eel walls' typical core width.

If $x_0 \in \R^2$ and $r > 0$, then
we write $B_r(x_0)$ for the open ball in $\R^2$ of radius
$r$ centred at $x_0$. Furthermore, we write $B_r^+(x_0) = B_r(x_0) \cap \R_+^2$.
For a set $S \subset \R^2$, we also use the notation $\partial^+S = \partial S \cap \R_+^2$.

For a vector $\xi = (\xi_1, \xi_2) \in \R^2$, we write $\xi^\perp = (-\xi_2, \xi_1)$.
If $f : \Omega \to \R$ is a function on a domain $\Omega \subset \R^2$
with gradient $\nabla f = (\dd{f}{x_1}, \dd{f}{x_2})$, then
$\nabla^\perp f = (-\dd{f}{x_2}, \dd{f}{x_1})$.

For $a \in A_N$, we define
\[
\rho(a) = \frac{1}{2} \min\{2a_1 + 2, a_2 - a_1, \ldots, a_N - a_{N - 1}, 2 - 2a_N\}.
\]
Thus this is a quantity that controls the distance between two
points of $a$ and the distance to the boundary.
For $r > 0$, also define
\[
B_r^*(a) = \bigcup_{n = 1}^N B_r^+(a_n, 0)
\]
and
\[
\Omega_r(a) = \R_+^2 \backslash B_r^*(a).
\]

Given a function $f : \R_+^2 \to \R$ with a well-defined trace on $(-1, 1) \times \{0\}$,
we often use the
shorthand notation $f'$ for $\dd{f}{x_1}(\blank, 0)$ and
\[
\int_{-1}^1 f \, dx_1 = \int_{-1}^1 f(x_1, 0) \, dx_1.
\]

In addition to the space $\dot{W}^{1, 2}(\R_+^2)$ introduced in
Sect. \ref{sec:model}, we also define $\dot{W}_*^{1, 2}(\R_+^2; a)$
for $a \in A_N$, which is the space of all $u \in \bigcap_{r > 0} \dot{W}^{1, 2}(\Omega_r(a))$
such that
\[
\sup_{r > 0} \frac{\|\nabla u\|_{L^2(\Omega_r(a))}^2}{\log \left(\frac{1}{r} + 2\right)} < \infty.
\]
If $b \in (-1, 1)$ is a single point, then $\dot{W}_*^{1, 2}(\R_+^2; b)$
is defined similarly.

\section{The renormalised energy} \label{sect:reduced}

When we study minimisers of $E_\epsilon$ in $M(a, d)$ and
let $\epsilon$ tend to $0$, then we expect the suitably rescaled
stray field potential to converge to a harmonic function with
specific boundary data depending on $a$ and $d$. We compute this
function here, in order to obtain the limiting magnetostatic
energy. This corresponds to the sum of the tail self-energies and
the contributions of the core-tail, tail-tail, and tail-boundary interactions. As a side
product, we will also obtain information about the expected logarithmic profile of the N\'eel walls.

\subsection{The limiting rescaled stray-field potential}

Fix $a \in A_N$ and $d \in \{\pm 1\}^N$. We recall that $$\textrm{$\gamma_n = d_n - \cos \alpha$ \quad
for $n = 1, \ldots, N$}$$ and we denote
\[
\sigma_n = \frac{\pi}{2} \left(\sum_{k = n + 1}^N \gamma_k - \sum_{k = 1}^n \gamma_k\right), \quad n = 0, \ldots, N.
\]
We look for a solution of the
following boundary value problem:
\begin{alignat}{2}
\Delta u_{a, d}^* & = 0 &\quad& \text{in $\R_+^2$}, \label{eqn:reduced_harmonic} \\
u_{a, d}^* & = \sigma_0 && \text{on $(-1, a_1) \times \{0\}$}, \label{eqn:reduced_boundary_left} \\
u_{a, d}^* & = \sigma_n && \text{on $(a_n, a_{n + 1})$ for $n = 1, \ldots, N - 1$}, \\
u_{a, d}^* & = \sigma_N && \text{on $(a_N, 1) \times \{0\}$}, \label{eqn:reduced_boundary_right} \\
\dd{u_{a, d}^*}{x_2} & = 0 && \text{on $(-\infty, -1) \times \{0\}$ and on $(1, \infty) \times \{0\}$}. \label{eqn:reduced_Neumann}
\end{alignat}
The boundary data are chosen so that we have a jump of
size $-\pi \gamma_n$ at $a_n$ for $n = 1, \ldots, N$.
We also require that $u_{a, d}^*\in \dot{W}_*^{1, 2}(\R_+^2; a)$ so that the boundary conditions make sense.
However, the Dirichlet energy cannot be finite on all of $\R_+^2$, since
any solution $u_{a, d}^*$ will behave similarly to the function\footnote{This function satisfies \eqref{eqn:reduced_harmonic}--\eqref{eqn:reduced_boundary_right}, but not
\eqref{eqn:reduced_Neumann}. It does not belong to $\dot{W}_*^{1, 2}(\R_+^2; a)$ due
to its behaviour at $\infty$.}
\[
\sum_{n = 1}^N \gamma_n \left(\arctan \left(\frac{x_2}{x_1 - a_n}\right) - \frac{\pi(x_1 - a_n)}{2|x_1 - a_n|}\right)
\]
near the boundary $(-1, 1) \times \{0\}$ (see Proposition \ref{prop:minimiser} below).
But we can compensate by considering the expression
\be
\label{energy_rescaled_stray}
E_{a, d}^*(u_{a, d}^*) = \frac{1}{2} \liminf_{r \searrow 0} \left(\int_{\Omega_r(a)} |\nabla u_{a, d}^*|^2 \, dx - \pi \log \frac{1}{r}\sum_{n = 1}^N \gamma_n^2\right),
\ee
and this will be part of the limiting energy.
We will see that the unique solution $u_{a, d}^*$ of
\eqref{eqn:reduced_harmonic}--\eqref{eqn:reduced_Neumann} is
the minimizer of $E_{a, d}^*$ and corresponds to the limit as
$\eps\searrow 0$ of the rescaled stray field potential
$(\log \frac 1 \delta) U(m_\eps)$ associated
to the magnetisation $m_\eps$ with walls at $a_1, \ldots, a_N$
with prescribed signs $d_1, \ldots, d_N$ (see Proposition \ref{prop:minimiser} and Theorem \ref{thm:core_convergence}).

In order to solve \eqref{eqn:reduced_harmonic}--\eqref{eqn:reduced_Neumann},
we first study the simpler problem
\begin{alignat}{2}
\Delta u  & = 0 &\quad & \text{in $\R_+^2$}, \label{P1}\\
u & = \frac{\pi}{2} && \text{on $(-1, 0) \times \{0\}$}, \label{P2}\\
u & = -\frac{\pi}{2} && \text{on $(0, 1) \times \{0\}$}, \label{P3}\\
\dd{u}{x_2} & = 0 && \text{on $(-\infty, -1) \times \{0\}$ and $(1, \infty) \times \{0\}$}. \label{P4}
\end{alignat}
We can obtain a solution $u_{a, d}^*$ of \eqref{eqn:reduced_harmonic}--\eqref{eqn:reduced_Neumann} by a linear
combination of solutions to problems of the type \eqref{P1}--\eqref{P4} (see Proposition \ref{prop:minimiser}).

We first construct an explicit solution, 
using the fact that harmonic functions remain
harmonic upon precomposition with a conformal
map. We identify $\R^2$ with $\C$. Consider the following domain in the complex plane $\C$:
\[ \label{def:S}
S = \set{w_1 + iw_2 \in \C}{w_1 > 0 \text{ and } 0 < w_2 < \pi}.
\]
Also consider the conformal map $F : S \to \R_+^2$ with
\[
F(w) = - \frac{1}{\cosh w}, \quad w \in S.
\]
Extend $F$ continuously to the boundary $\partial S\setminus \{\frac{\pi i} 2\}$.
Assuming that $u$ solves \eqref{P1}--\eqref{P4}, set
$$\hat{u} = u \circ F.$$ Then $\hat{u}$ solves the
boundary value problem
\begin{alignat*}{2}
\Delta \hat{u}  & = 0 &\quad& \text{in $S$}, \\
\hat{u} & = \frac{\pi}{2} &&  \text{on $(0, \infty)$}, \\
\hat{u} & = - \frac{\pi}{2} && \text{on $\pi i + (0, \infty)$}, \\
\dd{\hat{u}}{x_1} & = 0 && \text{on $i(0,\pi)$}.
\end{alignat*}
This problem has an obvious solution,
\[
\hat{u}(w) = \frac{\pi}{2} - \Im w,
\]
with $\Im w=w_2$ for $w=w_1+i w_2$. Thus we obtain a solution of \eqref{P1}--\eqref{P4} by 
\[
u(z) = \hat{u}(F^{-1}(z)) = \frac{\pi}{2} - \Im F^{-1}(z), \quad z\in \R^2_+,
\]
which can be written as
\be
\label{uniq_sol}
u = \Im f \quad \text{for} \quad f(z) = \frac{\pi i}{2} - F^{-1}(z), \quad z\in \R^2_+,
\ee
where $f$ is an holomorphic function. Since
$\lim_{|z| \to \infty} F^{-1}(z)=\frac{\pi i}{2}$, we conclude that
$\lim_{|x| \to \infty} u(x) = 0$.

\begin{proposition} \label{pro:uniq}
The function $u$ from \eqref{uniq_sol} is the unique solution
of \eqref{P1}--\eqref{P4} in $\dot{W}_*^{1, 2}(\R_+^2; 0)$.
It satisfies $|u| \le \frac{\pi}{2}$ in $\R_+^2$ and
$\lim_{|x|\to \infty} u(x) = 0$. Moreover, it is odd in
$x_1$, that is, $u(x_1, x_2) = - u(-x_1, x_2)$ for every $x \in \R_+^2$.
Furthermore, there exists a constant $C$ such that
\begin{equation} \label{eqn:gradient_core}
\left|\nabla u(x) - \frac{x^\perp}{|x|^2}\right| \le C|x|
\end{equation}
for all $x \in B_{1/2}^+(0)$.
\end{proposition}

\begin{proof}
It is clear from the construction that $|u| \le \frac{\pi}{2}$ and
$\lim_{|x| \to \infty} u(x) = 0$.
Since the Dirichlet energy is conformally invariant, we have
\[
\int_{R_+^2 \backslash B_r(0)} |\nabla u|^2 \, dz = \int_{F^{-1}(R_+^2 \backslash B_r(0))} |\nabla \hat{u}|^2 \, dw.
\]
Note that
\begin{equation} \label{eqn:cosh}
\begin{split}
|\cosh w|^2 & = \frac{1}{4} \left(e^w + e^{-w}\right) \left(e^{\bar{w}} + e^{-\bar{w}}\right) \\
& = \frac{1}{4} \left(e^{2\Re w} + e^{-2\Re w} + e^{2i\Im w} + e^{-2i\Im w}\right) \\
& = \frac{1}{2} \cosh (2\Re w) + \frac{1}{2} \cos (2 \Im w) \quad \textrm{ for $w \in \C$.}
\end{split}
\end{equation}
Thus
\[
F^{-1}(R_+^2 \backslash B_r(0)) \subset \set{w \in S}{\Re w < \frac{1}{2} \arcosh \left(\frac{2}{r^2} + 1\right)}.
\]
It follows that
\[
\int_{\R_+^2 \backslash B_r(0)} |\nabla u|^2 \, dz \le \frac{\pi}{2} \arcosh \left(\frac{2}{r^2} + 1\right).
\]
In particular, we have $u \in \dot{W}_*^{1, 2}(\R_+^2; 0)$.

If $\tilde{u}\in \dot{W}_*^{1, 2}(\R_+^2; 0)$ is another solution of \eqref{P1}--\eqref{P4},
then the difference
$v = u - \tilde{u}$ belongs to $\dot{W}_*^{1, 2}(\R_+^2; 0)$
and thus to $W^{1, p}(B_1^+(0)) \cap \dot{W}^{1, 2}(\R_+^2 \backslash B_{1/2}(0))$
for any $p \in [1, 2)$.
Furthermore, it is harmonic in $\R_+^2$  and satisfies
$v=0$ on $(-1,1)\times\{0\}$ and $\dd{v}{x_2} = 0$ on
$(-\infty, -1) \times \{0\}$ and on $(1, \infty) \times \{0\}$.
Standard regularity theory then implies that $v \in\dot{W}^{1, 2}(\R_+^2)$,
and with an integration by parts, we obtain
\[
\int_{\R_+^2} |\nabla v|^2 \, dx = 0.
\]
Hence $u$ is the unique solution of \eqref{P1}--\eqref{P4} in
$\dot{W}_*^{1, 2}(\R_+^2; 0)$.

The odd symmetry of $u$ is a consequence of the uniqueness, as
$x\mapsto -u(-x_1, x_2)$ is another solution of the boundary value problem.

Finally, we consider the function
\[
\tilde{v}(x) = u(x) - \arctan \left(\frac{x_2}{x_1}\right) + \frac{\pi x_1}{2|x_1|},
\]
which is harmonic in $\R_+^2$ as well with $\tilde{v} = 0$ on
$(-1, 1) \times\{0\}$. Invoking standard regularity theory again,
we conclude that $\tilde{v} \in C^\infty(\overline{B_{1/2}^+(0)})$.
By the odd symmetry, we have $\nabla \tilde{v}(0) = 0$.
Hence we obtain inequality \eqref{eqn:gradient_core}.
\end{proof}

\begin{definition}
For every $b \in (-1, 1)$, we introduce the M\"obius transform 
\be
\label{Mobius}
\Phi_b(z) = \frac{z + b}{1 + bz}, \quad z \in \C,
\ee
and its inverse $\Phi_b^{-1} = \Phi_{-b}$. We also define $u_b:\R_+^2\to \R$ by
\be
\label{u_b}
u_b = u \circ \Phi_{-b}\quad \textrm{in}\quad \R_+^2,
\ee
which, by Proposition \ref{pro:uniq}, is the unique solution of the boundary value problem
\begin{alignat*}{2}
\Delta u_b & = 0 & \quad & \text{in $\R_+^2$}, \\
u_b & = \frac{\pi}{2} && \text{on $(-1, b) \times \{0\}$}, \\
u_b & = - \frac{\pi}{2} && \text{on $(b, 1) \times \{0\}$}, \\
\dd{u_b}{x_2} & = 0 && \text{on $(-\infty, -1) \times \{0\}$ and on $(1, \infty) \times \{0\}$},
\end{alignat*}
in the space $\dot{W}_*^{1, 2}(\R_+^2; b)$.
\end{definition}

Now we can also construct a solution of
\eqref{eqn:reduced_harmonic}--\eqref{eqn:reduced_Neumann} by superposition.

\begin{proposition} \label{prop:minimiser}
The function $u_{a, d}^*$, defined by
\be
\label{stray_eq}
u_{a, d}^* = \sum_{n = 1}^N \gamma_n u_{a_n},
\ee
is the unique minimiser of $E_{a, d}^*$ among all $u \in \dot{W}_*^{1, 2}(\R_+^2; a)$ 
satisfying \eqref{eqn:reduced_boundary_left}--\eqref{eqn:reduced_boundary_right}. Moreover, $u_{a, d}^*$ is the unique solution of
\eqref{eqn:reduced_harmonic}--\eqref{eqn:reduced_Neumann} in $\dot{W}_*^{1, 2}(\R_+^2; a)$.
\end{proposition}

\begin{proof}
By Proposition \ref{pro:uniq}, we know that $u_{a, d}^*\in \dot{W}_*^{1, 2}(\R_+^2; a)$, and it
satisfies \eqref{eqn:reduced_boundary_left}--\eqref{eqn:reduced_boundary_right}. Let $u$ be another function with these properties. 
Since $u_{a, d}^*$ is harmonic in $\R^2_+$, integration by parts leads to
\begin{align*}
\int_{\Omega_r(a)} \left|\nabla u-\nabla u_{a, d}^*\right|^2\, dx&=\int_{\Omega_r(a)} |\nabla u|^2\, dx- \int_{\Omega_r(a)} |\nabla u_{a, d}^*|^2\, dx 
-2\int_{\Omega_r(a)} \left(\nabla u-\nabla u_{a, d}^*\right)\cdot \nabla u_{a, d}^*\, dx\\
&=\int_{\Omega_r(a)} |\nabla u|^2\, dx- \int_{\Omega_r(a)} |\nabla u_{a, d}^*|^2\, dx 
-2\int_{\partial \Omega_r(a)} \left(u-u_{a, d}^*\right)\nu \cdot \nabla u_{a, d}^*\, d\sigma.
\end{align*}
Using \eqref{eqn:gradient_core}, we find that there exists a constant $C$ such that for
any $n = 1, \ldots, N$ and for $x \in B_{\rho(a)}^+(a_n, 0)$,
\[
\left|\dd{u_{a, d}^*}{x_1}(x) + \frac{\gamma_n x_2}{(x_1 - a_n)^2 + x_2^2}\right| + \left|\dd{u_{a, d}^*}{x_2}(x) - \frac{\gamma_n (x_1 - a_n)}{(x_1 - a_n)^2 + x_2^2}\right| \le C.
\]
Hence the boundary integral will tend to $0$ when we let $r \searrow 0$.
Therefore,
\[
E_{a, d}^*(u_{a, d}^*)=E_{a, d}^*(u)-\frac 1 2 \lim_{r\searrow 0}\int_{\Omega_r(a)} \left|\nabla u-\nabla u_{a, d}^*\right|^2\, dx.
\]
The limit on the right hand side exists, because the quantity is monotone in $r$.

This implies that $u_{a, d}^*$ is the unique minimizer of $E_{a, d}^*$.
The uniqueness of solutions of \eqref{eqn:reduced_harmonic}--\eqref{eqn:reduced_Neumann} follows
with the same arguments as for Proposition~\ref{pro:uniq}.
\end{proof}

\begin{remark}
Since we have an explicit representation of $u_{a, d}^*$, an easy computation shows that the liminf in the definition of $E_{a, d}^*(u_{a, d}^*)$
is in fact a limit. By the preceding computation, the same holds for $E_{a, d}^*(u)$ for any $u\in \dot{W}_*^{1, 2}(\R_+^2; a)$ satisfying  \eqref{eqn:reduced_boundary_left}--\eqref{eqn:reduced_boundary_right}.
\end{remark}

\subsection{The energy of the limiting rescaled stray field}
\label{sec:comput_ene}

Next we want to compute the energy $E_{a, d}^*(u_{a, d}^*)$ defined in \eqref{energy_rescaled_stray}. 
Since this depends only on $a$ and $d$, we use the abbreviation
\[
W_1(a, d) = E_{a, d}^*(u_{a, d}^*).
\]
This quantity corresponds to the tail-boundary and tail-tail interaction energy.

\begin{proposition}
\label{pro:W_1}
If $\varrho$ is the metric on $(-1,1)$ defined in Section \ref{sect:result}, then
\be
\label{W_1}
W_1(a, d) = \frac{\pi}{2} \sum_{n = 1}^N \gamma_n^2 \log (2 - 2a_n^2) + \frac{\pi}{2} \sum_{k \not= n} \gamma_k \gamma_n \log \left(\frac{1 + \sqrt{1 - \varrho(a_k, a_n)^2}}{\varrho(a_k, a_n)}\right).
\ee
\end{proposition}

For the proof, we need the following two lemmas. First we compute the rescaled tail-tail interaction energy for two N\'eel walls located at two points $b\neq c$.

\begin{lemma}
\label{cross_product}
For all $b, c \in (-1, 1)$ with $b \not= c$, if
$u_b$ and $u_c$ are defined by \eqref{u_b}, then
\begin{equation} \label{eqn:reduced_energy1}
\int_{\R_+^2} \nabla u_b \cdot \nabla u_c \, dx =  \pi \log \left(\frac{1 + \sqrt{1 - \varrho(b, c)^2}}{\varrho(b, c)}\right).
\end{equation}
\end{lemma}

Then we compute the rescaled tail self-energy together with the tail-boundary interaction of a N\'eel wall located at a point $b\in (-1,1)$.

\begin{lemma}
\label{H^1_norm}
There exists a constant $C>0$ such that for every $b \in (-1, 1)$ and $r \in (0, 1 - |b|)$,
\begin{equation} \label{eqn:reduced_energy2}
\left|\int_{\R_+^2 \backslash B_r(b)} |\nabla u_b|^2 \, dx - \pi \log \left(\frac{2 - 2b^2}{r}\right)\right| \le C \left(\frac{|b|r}{1 - b^2} + \frac{r^2}{(1 - b^2 - |b|r)^2}\right).
\end{equation}
In particular, $$\lim_{r\searrow 0}\left( \int_{\R_+^2 \backslash B_r(b)} |\nabla u_b|^2 \, dx - \pi \log \frac 1 r\right)=\pi \log(2 - 2b^2).$$
\end{lemma}

\begin{proof}[Proof of Lemma \ref{cross_product}]
Using the definition of the M\"obius transform $\Phi_b$,
it is easy to check that
\[
\Phi_b \circ \Phi_c = \Phi_{\frac{b + c}{1 + bc}}.
\]
Set $q = \frac{b - c}{1 - bc}$. Then $|q|=\varrho(b, c)$ and
\[
\begin{split}
\int_{\R_+^2} \nabla u_b \cdot \nabla u_c \, dx & = \int_{\R_+^2} \nabla u \cdot \nabla (u_c \circ \Phi_b) \, dx  \\
& = \int_{\R_+^2} \nabla u \cdot \nabla u_q \, dx \\
& = -\PV\int_{-1}^1 u_q \dd{u}{x_2} \, dx_1 \\
& = -\frac{\pi}{2} \PV\int_{-1}^q \dd{u}{x_2} \, dx_1 + \frac{\pi}{2}\PV\int_{q}^1 \dd{u}{x_2} \, dx_1.
\end{split}
\]
In order to determine $\dd{u}{x_2}$, we recall that
the function $f$ defined in \eqref{uniq_sol} is holomorphic in $\R_+^2$. Hence
we have
\[
f'(z) = \dd{u}{x_2} + i \dd{u}{x_1}, \quad z\in \R_+^2.
\]
In particular,
\[
\dd{u}{x_2} = \Re f'(z),  \quad z\in \R_+^2.
\]
Differentiating both sides of the equation
\[
\frac{\pi i}{2} - w = f(F(w)),
\]
we calculate
\[
f'(F(w)) = -\frac{1}{F'(w)} = -\frac{\cosh^2 w}{\sinh w}.
\]

Let $t \in (0, 1)$ and $s = \arcosh \frac{1}{t}$. Set
$w = s + i \pi$, so that $F(w) = t$. Then
$\cosh w = - \frac{1}{t}$ and $\sinh w = - \sinh s = -\sqrt{t^{-2} - 1}$.
Hence
\[
f'(t) = \frac{1}{t \sqrt{1 - t^2}} = \frac{d}{dt} \left(\log |t| - \log \left(1 + \sqrt{1 - t^2}\right)\right).
\]
That is, 
\begin{equation} \label{form_du_x2}
\dd{u}{x_2}(x_1, 0) = \frac{d}{dx_1} \left(\log |x_1| - \log \left(1 + \sqrt{1 - x_1^2}\right)\right), \quad x_1\in (0,1).
\end{equation}
By the odd symmetry of $u$ in $x_1$ (see Proposition \ref{pro:uniq}), we have the same equality on $(-1, 0) \times \{0\}$.
Thus
\[
\PV\int_{-1}^q \dd{u}{x_2} \, dx_1 = \log |q| - \log \left(1 + \sqrt{1 - q^2}\right)
\]
and
\[
\PV\int_q^1 \dd{u}{x_2} \, dx_1 = - \log |q| + \log \left(1 + \sqrt{1 - q^2}\right),
\]
where the principal value can be ignored for
exactly one of these integrals because there
is no singularity.
It follows that
\eqref{eqn:reduced_energy1} holds.
\end{proof}

\begin{proof}[Proof of Lemma \ref{H^1_norm}]
Let $r \in (0, 1)$ and consider the integral
\[
I_r = \int_{\R^2_+ \backslash B_r(0)} |\nabla u|^2 \, dx.
\]
In order to estimate $I_r$, we first study the image of
$\R_+^2 \backslash B_r(0)$ under $F^{-1}$ and then
perform the change of variables $x=F(w)=-1/\cosh(w)$. Recall identity
\eqref{eqn:cosh}, which implies that
\[
\begin{split}
\set{w \in S}{\Re w < \frac{1}{2} \arcosh\left(\frac{2}{r^2} - 1\right)} & \subset F^{-1}(\R_+^2 \backslash B_r(0)) \\
& \subset \set{w \in S}{\Re w < \frac{1}{2} \arcosh \left(\frac{2}{r^2} + 1\right)},
\end{split}
\]
where $S$ is the domain defined on page \pageref{def:S}. Recalling the function $\hat{u}(w)=u(x)$ for $x=F(w)$,
we see that
\[
I_r = \int_{F^{-1}(\R_+^2 \backslash B_r(0))} |\nabla \hat{u}|^2 \, dw \in \left[\frac{\pi}{2} \arcosh\left(\frac{2}{r^2} - 1\right), \frac{\pi}{2} \arcosh \left(\frac{2}{r^2} + 1\right)\right].
\]
Thus
\[
\frac{4}{r^2} - 2 \le e^{2I_r/\pi} + e^{-2I_r/\pi} \le \frac{4}{r^2} + 2.
\]
As $0 < e^{-2I_r/\pi} < 1$, this means that
\[
\frac{4}{r^2} - 3 \le e^{2I_r/\pi} \le \frac{4}{r^2} + 2
\]
and
\[
\frac{\pi}{2} \log \left(\frac{4}{r^2} - 3\right) \le I_r \le \frac{\pi}{2} \log \left(\frac{4}{r^2} + 2\right).
\]
In particular, there exists a universal constant $C_1>0$
such that
\be
\label{intermid}
\left|\int_{\R_+^2 \backslash B_r(0)} |\nabla u|^2 \, dx - \pi \log \frac{2}{r}\right| \le C_1 r^2, \quad r\in (0,1).
\ee

For $b \in (-1, 1)$ and $r \in (0, 1 - |b|)$, we have
\[
\int_{\R_+^2 \backslash B_r(b)} |\nabla u_b|^2 \, dx = \int_{\R_+^2 \backslash \Phi_{-b}(B_r^+(b))} |\nabla u|^2 \, dx.
\]
Thus we now examine the set $\Phi_{-b}(B_r^+(b))$.
Since $\Phi_{-b}$ is a M\"obius transform, it maps the
semicircle $\partial^+ B_r(b)$ to a semicircle, which
contains the points
\[
\Phi_{-b}(b - r) = -\frac{r}{1 - b^2 + br}<0 \quad \text{and} \quad \Phi_{-b}(b + r) = \frac{r}{1 - b^2 - br}>0.
\]
Since $\Phi_{-b}(b)=0$, we obtain
\[
B_{r/(1 - b^2 +|b|r)}^+(0) \subset \Phi_{-b}(B_r^+(b)) \subset B_{r/(1 - b^2 - |b|r)}^+(0).
\]
It then follows from \eqref{intermid} that 
\[
\begin{split}
\lefteqn{\pi \log \frac{2}{r} + \pi \log (1 - b^2 - |b|r) - \frac{C_1 r^2}{(1 - b^2 - |b|r)^2}} \qquad\qquad \\
& \le \int_{\R_+^2 \backslash B_r(b)} |\nabla u_b|^2 \, dx \\
& \le \pi \log \frac{2}{r} + \pi \log (1 - b^2 + |b|r) + \frac{C_1 r^2}{(1 - b^2 - |b|r)^2}.
\end{split}
\]
Hence \eqref{eqn:reduced_energy2} holds.
\end{proof}

\begin{proof}[Proof of Proposition \ref{pro:W_1}]
The formula follows directly from the definition of
$W_1$, the definition \eqref{energy_rescaled_stray} of
$E_{a, d}^*$, Proposition \ref{prop:minimiser},
and the last two lemmas.
\end{proof}

\subsection{The rescaled tail profile}
\label{sec:rescale_tail}

Since $u_{a, d}^*$ is the expected limit of the rescaled stray field potential,
assuming that condition
\eqref{eqn:u_boundary1} will be preserved in the limit,
we can derive an expected profile $\mu_{a, d}^*$ for the rescaled first component of
the tails of the N\'eel walls.

Using the unique solution $u$ of the problem \eqref{P1}--\eqref{P4}, we first define the logarithmically rescaled tail profiles
\be
\label{def_mu}
\mu(x_1) := -\PV\int_{-1}^{x_1} \dd{u}{x_2}(t, 0) \, dt \stackrel{\eqref{form_du_x2}}{=} \log \left(1 + \sqrt{1 - x_1^2}\right) - \log |x_1|, \quad x_1\in (-1,1)\setminus\{0\},
\ee
and
\be
\label{def_mub}
\mu_b(x_1) :=\mu\circ \Phi_{-b}(x_1)=\mu\left(\frac{x_1 - b}{1 - bx_1}\right), \quad x_1\in (-1,1)\setminus\{b\}.
\ee
Then for $u_b=u\circ \Phi_{-b}$, we find $\dd{u_b}{x_2} = \dd{u}{x_2} \Phi'_{-b}=-\mu_b'$ on $(-1, b) \times \{0\}$
and $(b, 1) \times \{0\}$, using the conformality
of $\Phi_{-b}$ and the fact that $\Phi'_{-b}$ is real on
$(-1, b) \times \{0\}$ and on $(b, 1) \times \{0\}$.
If we define
\[
\mu_{a, d}^* = \sum_{n = 1}^N \gamma_n \mu_{a_n}, \quad x_1\in (-1,1)\setminus\{a_1, \dots, a_N\},
\]
then we also have
\[
\dd{u_{a, d}^*}{x_2} = -(\mu_{a, d}^*)' \text{ on
$(-1, 1) \times \{0\}$ except at the singularities $a_k$, $k=1, \dots, N$.}
\]
Note that
$$\mu_{a_n}(a_k)=\mu_{a_k}(a_n)=\log \left(\frac{1 + \sqrt{1 - \varrho(a_n, a_k)^2}}{\varrho(a_n, a_k)}\right), \quad n\neq k.$$

We need to examine the behaviour of $\mu_{a, d}^*$ near the points
$a_n$, as this will be important for determining the energy of
the tail-core interaction.

\begin{proposition}
Set
\[
\lambda_n = \gamma_n \log(2 - 2a_n^2) + \sum_{k \not= n} \gamma_k \mu_{a_k}(a_n)
\]
for $n = 1, \ldots, N$. Then there exists a constant $C=C(a)>0$
such that
\begin{equation} \label{eqn:mu_estimate}
\left|\mu_{a, d}^*(a_n \pm r) - \lambda_n - \gamma_n \log \frac{1}{r}\right| \le C r
\end{equation}
for any $r \in (0, \rho(a)]$ and any $n = 1, \ldots, N$.
\end{proposition}

\begin{proof}
First we study $\mu_b$ again for
a fixed $b$. Obviously, for $c \in (-1, 1)$ with $b \not= c$,
this function is smooth at $c$, and therefore there exists
a constant $C_1>0$, depending only on $b$ and $c$, such that
\[
|\mu_b(c \pm r) - \mu_b(c)| \le C_1 r
\]
for any sufficiently small $r > 0$. We also have, by \eqref{def_mu} and \eqref{def_mub}, the formula
\[
\mu_b(b \pm r) = \log\left(1 + \sqrt{1 - \frac{r^2}{(1 - b^2 \mp br)^2}}\right) + \log(1 - b^2 \mp br) + \log \frac{1}{r}.
\]
Thus there exists a constant $C_2 = C_2(b)$ such that
\[
\left|\mu_b(b \pm r) - \log(2 - 2b^2) - \log \frac{1}{r}\right| \le C_2 r
\]
for $r \in (0, 1-|b|)$. If we sum up for $b=a_k$, $k=1, \dots, N$,
then the conclusion follows.
\end{proof}

We will prove in Section \ref{sect:lower} that the function
\[
\cos \alpha + \frac{\mu_{a, d}^*}{\log \frac{1}{\delta}}
\]
gives an approximation for the profile of the first component $m_{1, \eps}$ for a minimiser
$m_\eps$ of $E_\epsilon$ over $M(a, d)$. In a standard N\'eel wall with
rotation angle $2\alpha$ (or $2\pi - 2\alpha$), the function $m_{1, \eps}$ would be required
to make a transition from $\cos \alpha$ to $1$ (or $-1$) and back.
If we superimpose several N\'eel walls, then this is no longer
true. Instead, in an interval $(a_n - r, a_n + r)$, up to a small error,
$m_{1, \eps}$ is now required, by \eqref{eqn:mu_estimate}, to make a transition from
\[
\cos \alpha + \frac{ \lambda_n+\gamma_n  \log \frac{1}{r}}{\log \frac{1}{\delta}}
\]
to $\pm 1$ and back. We discount the term involving $\log \frac{1}{r}$,
as it will be cancelled by a similar term elsewhere. 
The contribution of the remaining term to the energy
(to leading order, rescaled by $(\log \delta)^2$) is then
$- \pi \gamma_n \lambda_n$ near the point $a_n$ (see Remark \ref{rem:core_estimate} below). Thus in total,
we have a correction term of the form
\[
W_2(a, d) = -\pi \sum_{n = 1}^N \gamma_n \lambda_n = -\pi \sum_{n = 1}^N \left(\gamma_n^2 \log(2 - 2a_n^2) + \sum_{k \not= n} { \gamma_n}\gamma_k \mu_{a_k}(a_n)\right).
\]
Using the above explicit expression for $\mu_{a_k}(a_n)$,
we find
\[
W_2(a, d) = -\pi \sum_{n = 1}^N \gamma_n^2 \log(2 - 2a_n^2) - \pi \sum_{n = 1}^N \sum_{k \not= n} \gamma_k \gamma_n \log \left(\frac{1 + \sqrt{1 - \varrho(a_k, a_n)^2}}{\varrho(a_k, a_n)}\right).
\]
This corresponds to the sum of the tail-core interactions described in
Section \ref{sect:result}. Note that
$$W_2(a, d)=-2W_1(a, d).$$
We finally define $W(a, d) := W_1(a, d) + W_2(a, d).$ That is,
\[
W(a, d) = -\frac{\pi}{2} \sum_{n = 1}^N \gamma_n^2 \log(2 - 2a_n^2) - \frac{\pi}{2} \sum_{n = 1}^N \sum_{k \not= n} \gamma_k \gamma_n \log \left(\frac{1 + \sqrt{1 - \varrho(a_k, a_n)^2}}{\varrho(a_k, a_n)}\right).
\]

\section{The Euler-Lagrange equation} \label{sect:Euler-Lagrange}

Since we study the number $\inf_{M(a, d)} E_\epsilon$
for a given $a \in A_N$ and $d \in \{\pm 1\}^N$, it is useful
to consider minimisers of $E_\epsilon$ in $M(a, d)$ (see Proposition \ref{pro:exist_unique}). The problem gives
rise to an Euler-Lagrange equation, which is most easily
expressed in terms of the continuous function $\varphi : (-1, 1) \to \R$
with $\varphi(-1) = \alpha$ and $m = (\cos \varphi, \sin \varphi)\in W^{1,2}((-1,1), \Ss^1)$.
Let $u = U(m)$ be the function defined on page \pageref{def:U}. Then the equation is
\begin{equation} \label{eqn:Euler-Lagrange}
\epsilon \varphi''(x_1) = \dd{u}{x_1}(x_1, 0) \sin \varphi(x_1) \quad \text{for } x_1 \in (-1, 1) \backslash \{a_1, \ldots, a_N\}.
\end{equation}
For the derivation of \eqref{eqn:Euler-Lagrange}, let $\zeta\in C^\infty_0(-1, 1)$ with $\zeta(a_k)=0$ for $k=1, \dots, N$ be an infinitesimal variation of $m_1$. Then
\begin{align*}
\left.\frac{d}{dt}\right|_{t=0} \left(\frac 1 2  \int_{\R^2_+} |\nabla U(m_1+t\zeta)|^2\, dx\right) &=
\int_{\R^2_+} \nabla U(m_1)\cdot \nabla U(\zeta)\, dx\\
&\stackrel{\eqref{weak_stray}}{=}\int_{-1}^1\zeta' U(m_1) \, dx_1=-\int_{-1}^1\zeta \dd{u}{x_1} \, dx_1.
\end{align*}
Equation \eqref{eqn:Euler-Lagrange} is now obtained with the usual computations.

We use the abbreviation $u'$ for the derivative
of the trace of $u$ with respect to $(-1, 1) \times \{0\}$.
Then we have the following shorthand form of \eqref{eqn:Euler-Lagrange}:
\[
\epsilon \varphi'' = u' \sin \varphi \quad \text{on } (-1, 1) \backslash \{a_1, \ldots, a_N\}.
\]
We now analyse this equation. More precisely, we prove an interior
$W^{2,2}$-estimate for solutions of \eqref{eqn:Euler-Lagrange} and we prove a
Pohozaev identity. As a consequence, we eventually find that the exchange
energy in the core of a N\'eel wall is of order
$O(1/(\log \delta)^2)$.

\subsection{An interior $W^{2,2}$-estimate}

We have the following interior $W^{2,2}$-estimate for a solution of the
Euler-Lagrange equation and for the corresponding stray field potential.
This estimate will be used in Theorem 
\ref{thm:core_convergence} below in order to find the specific profile of the N\'eel wall.

\begin{lemma} \label{lemma:higher_estimates}
Let $0\leq r<r'<R'<R$. There exists a constant $C>0$ (depending only on $r'-r$ and $R-R'$) such that
the following holds true. Let $\epsilon > 0$ and set $I=(-R,-r)\cup (r, R)$. Suppose that the functions 
$u \in W^{1,2}(B_R^+(0)\setminus B_r(0))$ and $\varphi \in W^{1,2}(I)$
solve the system 
\begin{alignat*}{2}
\Delta u & = 0 &\quad& \text{in $B_R^+(0)\setminus B_r(0)$}, \\
\dd{u}{x_2} & = \varphi' \sin \varphi && \text{on $I \times \{0\}$}, \\
\epsilon \varphi'' & = u' \sin \varphi && \text{in $I$}.
\end{alignat*}
Further suppose that $\sin \varphi\neq 0$ in $I$. Then
\begin{multline*}
\int_{B_{R'}^+(0)\setminus B_{r'}(0)} |\nabla^2 u|^2 \, dx + \epsilon \int_{(-R',-r')\cup (r', R')} \bigg((\varphi'')^2 + (\varphi')^4(1 + \cot^2 \varphi)\bigg) \, dx_1 \\
\le C \bigg(\epsilon \int_I (\varphi')^2 \, dx_1+\int_{B_R^+(0)\setminus B_r(0)} |\nabla u|^2 \, dx\bigg).
\end{multline*}
\end{lemma}

\begin{proof}
We first note that $\varphi$ is smooth on $I$ (see \cite{Ignat_Knupfer}) and therefore
$u$ is smooth in $B_R^+(0)\setminus B_r(0)$ up to the boundary $I\times\{0\}$. 
Consider the function $v = \dd{u}{x_1}$ on $B_R^+(0)\setminus B_r(0)$. We have
\[
v = \frac{\epsilon \varphi''}{\sin \varphi} \quad \text{and} \quad \dd{v}{x_2} = - \frac{\partial^2}{\partial x_1^2} \big( \cos \varphi \big)\, 
\textrm{ on $I \times \{0\}$}.
\]
Let $\eta \in C_0^\infty(B_R^+(0)\setminus B_r(0))$.
Since $v$ is harmonic in $B_R^+(0)\setminus B_r(0)$, integration by parts yields
\[
\begin{split}
\int_{B_R^+(0)} \eta^2 |\nabla v|^2 \, dx & = \int_I \eta^2 v\frac{d^2}{dx_1^2} (\cos \varphi) \, dx_1 - 2 \int_{B_R^+(0)} \eta v \nabla \eta \cdot \nabla v \, dx \\
& = -\epsilon \int_I \eta^2 \frac{\varphi''}{\sin \varphi} \left(\varphi'' \sin \varphi + (\varphi')^2 \cos \varphi\right) \, dx_1- 2 \int_{B_R^+(0)} \eta v \nabla \eta \cdot \nabla v \, dx \\
& = - \epsilon \int_I \eta^2 (\varphi'')^2 \, dx_1 - \epsilon \int_I \eta^2  \frac{\left((\varphi')^3\right)'}{3} \cot \varphi \, dx_1 
- 2 \int_{B_R^+(0)} \eta v \nabla \eta \cdot \nabla v \, dx \\
& = - \epsilon \int_I \eta^2 (\varphi'')^2 \, dx_1 - \frac{\epsilon}{3} \int_I \eta^2 (\varphi')^4 (1 + \cot^2 \varphi) \, dx_1
+ \frac{2\epsilon}{3} \int_I \eta \eta' (\varphi')^3 \cot \varphi \, dx_1\\ & \quad  - 2 \int_{B_R^+(0)} \eta v \nabla \eta \cdot \nabla v \, dx.
\end{split}
\]
By Young's inequality,
\[
\int_I \eta \eta' (\varphi')^3 \cot \varphi \, dx_1 \le \frac{1}{4} \int_I \eta^2 (\varphi')^4 \cot^2 \varphi \, dx_1 + \int_I (\eta')^2 (\varphi')^2 \, dx_1
\]
and
\[
-\int_{B_R^+(0)} \eta v \nabla \eta \cdot \nabla v \, dx \le \frac{1}{4} \int_{B_R^+(0)} \eta^2 |\nabla v|^2 \, dx + \int_{B_R^+(0)} |\nabla \eta|^2 v^2 \, dx.
\]
Therefore,
\begin{align}
\label{H2estim_local}
\int_{B_R^+(0)} \eta^2 |\nabla v|^2 \, dx + &\epsilon \int_I \eta^2 \left((\varphi'')^2  + (\varphi')^4(1 + \cot^2 \varphi)\right) \, dx_1 \\
\nonumber
&\le 12 \int_{B_R^+(0)} v^2 |\nabla \eta|^2 \, dx + 4\epsilon \int_I (\eta')^2 (\varphi')^2 \, dx_1.
\end{align}
As $u$ is harmonic, we observe that
that
\[
\frac{\partial^2 u}{\partial x_2^2} = - \dd{v}{x_1},
\]
so that
\[
\frac 1 2 |\nabla^2 u|^2= \left|\dd{v}{x_1}\right|^2+ \left|\dd{v}{x_2}\right|^2 \quad \textrm{in }\, B_R^+(0)\setminus B_r(0).
\]
Choosing suitable cut-off functions $\eta$, we can now easily derive the desired inequality.
\end{proof}

\subsection{A Pohozaev identity}

As in the theory of Ginzburg-Landau vortices, a
variant of the identity due to Pohozaev \cite{Pohozaev:65}
is useful for our problem. For a function $u : B_1^+(0) \to \R$,
we use the notation
\[
\partial_\rho u = \frac{x}{|x|} \cdot \nabla u.
\]

\begin{lemma} \label{lemma:Pohozaev}
Let $\epsilon > 0$.
Suppose that the functions $u \in W^{1,2}(B_1^+(0))$ and $\varphi \in W^{1,2}(-1,1)$
solve the system 
\begin{alignat}{2}
\Delta u & = 0 &\quad& \text{in $B_1^+(0)$}, \label{eqn:Pohozaev_harmonic} \\
\dd{u}{x_2} & = \varphi' \sin \varphi && \text{on $(-1,1) \times \{0\}$}, \label{eqn:Pohozaev_boundary} \\
\epsilon \varphi'' & = u' \sin \varphi && \text{in $(-1, 0)$ and $(0, 1)$}. \label{eqn:Pohozaev_Euler-Lagrange}
\end{alignat}
Then for any $r \in (0, 1)$,
\[
\epsilon \int_{-r}^{r} (\varphi')^2 \, dx_1 = r\epsilon (\varphi'(r))^2 + r\epsilon (\varphi'(-r))^2
+ r\int_{\partial^+ B_r(0)} \left(|\nabla u|^2 - 2\left(\partial_\rho u\right)^2\right) \, d\sigma.
\]
\end{lemma}

\begin{proof}
As $u$ is harmonic, we calculate
\[
\div \left(\frac{1}{2} |\nabla u|^2 x - (x \cdot \nabla u) \nabla u\right) = 0 \quad \textrm{in $B_1^+(0)$.}
\]
For any $r, s \in (0, 1)$ with $s < r$, it follows that
\[
\begin{split}
0 & = \int_{\partial^+ B_r(0)} \left(\frac{r}{2} |\nabla u|^2 - r \left(\partial_\rho u\right)^2\right) \, d\sigma - \int_{\partial^+ B_s(0)} \left(\frac{s}{2} |\nabla u|^2 - s \left(\partial_\rho u\right)^2\right) \, d\sigma \\
& \quad + \int_{-r}^{-s} x_1 \dd{u}{x_1} \dd{u}{x_2} \, dx_1 + \int_s^r x_1 \dd{u}{x_1} \dd{u}{x_2} \, dx_1.
\end{split}
\]
Using \eqref{eqn:Pohozaev_boundary} and \eqref{eqn:Pohozaev_Euler-Lagrange},
we compute
\[
\begin{split}
\int_{-r}^{-s} x_1 \dd{u}{x_1} \dd{u}{x_2} \, dx_1 & = \epsilon\int_{-r}^{-s} x_1 \varphi'' \varphi' \, dx_1 = \frac{\epsilon}{2} \int_{-r}^{-s} x_1 \frac{d}{dx_1} (\varphi')^2 \, dx_1 \\
& = \frac{\epsilon r}{2} (\varphi'(-r))^2 - \frac{\epsilon s}{2} (\varphi'(-s))^2 - \frac{\epsilon}{2} \int_{-r}^{-s} (\varphi')^2 \, dx_1.
\end{split}
\]
Similarly,
\[
\int_s^r x_1 \dd{u}{x_1} \dd{u}{x_2} \, dx_1 = \frac{\epsilon r}{2} (\varphi'(r))^2 - \frac{\epsilon s}{2} (\varphi'(s))^2 - \frac{\epsilon}{2} \int_s^r (\varphi')^2 \, dx_1.
\]
Hence
\begin{multline*}
\frac{\epsilon}{2} \int_{-r}^{-s} (\varphi')^2 \, dx_1 + \frac{\epsilon}{2} \int_s^r (\varphi')^2 \, dx_1 
= \frac{r}{2} \left(\epsilon(\varphi'(r))^2 + \epsilon(\varphi'(-r))^2 + \int_{\partial^+ B_r(0)} \left(|\nabla u|^2 - 2 \left(\partial_\rho u\right)^2\right) \, d\sigma\right) \\
- \frac{s}{2} \left(\epsilon(\varphi'(s))^2 + \epsilon(\varphi'(-s))^2 + \int_{\partial^+ B_s(0)} \left(|\nabla u|^2 - 2 \left(\partial_\rho u\right)^2\right) \, d\sigma\right).
\end{multline*}
Since $\varphi' \in L^2(-1, 1)$ and $|\nabla u| \in L^2(B_1^+(0))$,
there exists a sequence $s_k \searrow 0$ such that
\[
s_k \left(\epsilon(\varphi'(s_k))^2 + \epsilon(\varphi'(-s_k))^2 + \int_{\partial^+ B_{s_k}(0)} |\nabla u|^2 \, d\sigma\right) \to 0.
\]
Therefore, we obtain the desired identity.
\end{proof}

As a consequence, we obtain the following estimate, which implies that the exchange
energy in the core of a N\'eel wall is of order $O(1/(\log \delta)^2)$.

\begin{lemma} \label{lemma:core_estimate}
Suppose that $\epsilon \in (0, \frac{1}{2}]$ and
the functions $u \in W^{1,2}(B_1^+(0))$ and $\varphi \in W^{1,2}(-1,1)$
solve the system \eqref{eqn:Pohozaev_harmonic}--\eqref{eqn:Pohozaev_Euler-Lagrange}.
Let
\[
F = \frac{\epsilon}{2} \int_{-1}^1 (\varphi')^2 \, dx_1 + \frac{1}{2} \int_{B_1^+(0)} |\nabla u|^2 \, dx.
\]
Then
\begin{equation} \label{eqn:core_estimate1}
\frac{\epsilon}{2} \int_{-\delta}^\delta (\varphi')^2 \, dx_1 \le \frac{F}{\log \frac{1}{\delta}}.
\end{equation}
If $\varphi(0) \in \pi\Z$, then
\begin{equation} \label{eqn:core_estimate2}
\int_{-\delta}^\delta \sin^2 \varphi \, dx_1 \le 8\delta F.
\end{equation}
\end{lemma}

\begin{proof}
It follows from Lemma \ref{lemma:Pohozaev} that
\[
\frac{\epsilon}{2} \int_\delta^1 \frac{1}{r} \int_{-r}^r (\varphi')^2 \, dx_1 \, dr \le F.
\]
Thus there exists an $r \in (\delta, 1]$
such that
\[
\frac{\epsilon}{2} \int_{-r}^r (\varphi')^2 \, dx_1 \le \frac{F}{\log \frac{1}{\delta}}.
\]
Inequality \eqref{eqn:core_estimate1} then follows immediately.

For the second inequality, we note that for $|x_1| \le \delta$,
\[
|\sin \varphi(x_1)| = \left|\int_0^{x_1} \varphi'(t) \cos \varphi(t) \, dt\right| \le \left(|x_1| \int_0^{x_1} (\varphi')^2 \, dt\right)^{1/2} \le \sqrt{\frac{2\delta F}{\epsilon \log \frac{1}{\delta}}} \le 2\sqrt{F}.
\]
(Here we use that fact that $\log \frac{1}{\epsilon} \le 2\log \frac{1}{\delta}$
for $0 < \epsilon \le 1$.) Thus \eqref{eqn:core_estimate2}
follows immediately as well.
\end{proof}

Later we will prove estimates similar to \eqref{eqn:core_estimate1} and \eqref{eqn:core_estimate2} even
without making use of the Euler-Lagrange equation, but assuming suitable control of the energy instead
(see Theorem \ref{thm:stray_field} and Remark \ref{rem:non_min}).

\section{The core}

In this section, we study what happens near a single
N\'eel wall, rescaled to unit size. Several technical difficulties arise in the analysis of the local
behaviour of the magnetization, due to the nonlocal nature of the
magnetostatic energy. For this reason, we first introduce a modified
energy functional where the stray field is considered on a half ball
(instead of $\R^2_+$) with a Neumann boundary condition. We prove energy
estimates (upper and lower bounds) for this modified functional together
with statements about its behaviour under small perturbations of the boundary data.
The analysis of the minimisers of this functional is essential because it yields good approximation results for the N\'eel wall profile
(see Theorem \ref{thm:core_convergence}) and for the core energy
(see Theorem \ref{thm:core_energy_limit}).

\subsection{A modified functional}

For $\gamma \in (0, 2)$, we define the convex set 
$$M_\gamma=\{\mu \in W^{1, 2}(-1, 1) \, :\,  \mu(0) = 1, \, \, \mu(\pm 1) \le 1 - \gamma\}.$$
We think of $\mu$ as the first component
of the magnetisation near a N\'eel wall with transition angle
$2\arccos(1 - \gamma)$. For $\mu \in M_\gamma$, consider the convex set  
\[
W^{1, 2}_\mu(B_1^+(0))= \{w \in W^{1, 2}(B_1^+(0)) \colon w(x_1, 0) = \mu(x_1) \text{ for } x_1 \in (-1, 1) \text{ and } w \le 1 - \gamma \textrm{ on } \partial^+ B_1(0)\}. 
\]
Clearly $W^{1, 2}_\mu(B_1^+(0))\neq \emptyset$ and there exists a unique
function where the infimum
\[
\inf_{w \in W_\mu^{1, 2}(B_1^+(0))} \int_{B_1^+(0)} |\nabla w|^2 \, dx
\]
is attained, owing to the strict convexity of the Dirichlet functional.

Now we define the following modified energy functional for $\mu \in M_\gamma$:
\[
E_\epsilon^\gamma(\mu) = \frac{\epsilon}{2} \int_{-1}^1 \frac{(\mu')^2}{1 - \mu^2} \, dx_1 + 
\frac{1}{2} \inf_{w \in W_\mu^{1, 2}(B_1^+(0))} \int_{B_1^+(0)} |\nabla w|^2 \, dx.
\]

\begin{proposition}
\label{prop_min}
The functional $E_\epsilon^\gamma$ admits a unique minimizer 
$\mu \in M_\gamma$, which satisfies $\mu(\pm 1) = 1 - \gamma$
and $1\ge \mu \ge 1 - \gamma$ in $(-1,1)$. Moreover, if
$v \in W_\mu^{1, 2}(B_1^+(0))$ is the unique function with
\[
\int_{B_1^+(0)} |\nabla v|^2 \, dx = \inf_{w \in W_\mu^{1, 2}(B_1^+(0))} \int_{B_1^+(0)} |\nabla w|^2 \, dx,
\]
then $1-\gamma\leq v\leq 1$ in $B_1^+(0)$ and $v$ is the unique
solution in $W^{1, 2}(B_1^+(0))$ of the boundary value problem
\begin{alignat}{2}
\Delta v & = 0 &\quad& \text{in $B_1^+(0)$}, \label{eqn:core_harmonicv} \\
v(x_1, 0) & = \mu(x_1) && \text{for $x_1 \in (-1, 1)$}, \label{eqn:core_boundary1v} \\
v & = 1 - \gamma && \text{on $\partial^+ B_1(0)$}. \label{eqn:core_boundary2v}
\end{alignat}
\end{proposition}

\begin{proof}
The direct method of the calculus of variations leads
to existence of a minimizer $\mu \in M_\gamma$ of $E_\epsilon^\gamma$.
Moreover, if $\mu \in M_\gamma$ is a minimiser of $E_\epsilon^\gamma$,
then using a simple cut-off argument at $1-\gamma$ and $1$,
we see that $\mu(\pm 1) = 1 - \gamma$
and $1\ge \mu \ge 1 - \gamma$ in $(-1,1)$. This implies that the corresponding minimizer $v$ of the Dirichlet energy
satisfies $1-\gamma\leq v\leq 1$ in $B_1^+(0)$ and solves
\eqref{eqn:core_harmonicv}--\eqref{eqn:core_boundary2v}. Obviously, this boundary value problem has a unique solution in $W^{1, 2}(B_1^+(0))$.
Since the function $(x_1, x_2) \mapsto \frac{x_1^2}{1-x_2^2}$, for $(x_1, x_2) \in \R \times (-1, 1)$,
is strictly convex, we deduce that $E_\epsilon^\gamma$ admits in fact a {\it unique} minimizer 
$\mu \in M_\gamma$. 
\end{proof}

Let $\mu \in M_\gamma$ be the minimizer of $E_\epsilon^\gamma$ and $v$ be the
solution of \eqref{eqn:core_harmonicv}--\eqref{eqn:core_boundary2v}. 
Since $\curl \nabla^\perp v = 0$, there exists a function
$u \in W^{1, 2}(B_1^+(0))$ with $\nabla^\perp v = - \nabla u$.
This is then a solution of
\begin{align}
\Delta u & = 0 \quad \text{in $B_1^+(0)$}, \label{eqn:core_harmonic} \\
\dd{u}{x_2} & = - \mu' \quad \text{on $(-1, 1) \times \{0\}$}, \label{eqn:core_bd1}\\
x \cdot \nabla u & = 0 \quad \text{on $\partial B_1^+(0)$} \label{eqn:core_boundary2}.
\end{align}
Note that $u$ is determined uniquely up to a constant
by these conditions. On the other hand, given a function
$u \in W^{1, 2}(B_1^+(0))$ satisfying \eqref{eqn:core_harmonic}--\eqref{eqn:core_boundary2},
we can reconstruct a corresponding function $v \in W_\mu^{1, 2}(B_1^+(0))$
with $\nabla^\perp v = -\nabla u$. It follows that the minimiser $\mu$ of
$E_\epsilon^\gamma$ automatically minimises the quantity
\[
\frac{\epsilon}{2} \int_{-1}^1 \frac{(\mu')^2}{1 - \mu^2} \, dx_1 + \frac{1}{2} \int_{B_1^+(0)} |\nabla u|^2 \, dx,
\]
where $u$ is determined (up to a constant) by
\eqref{eqn:core_harmonic}--\eqref{eqn:core_boundary2}. Since $\mu$ plays the role of the first component
of the magnetisation near a N\'eel wall, then $u$ roughly
corresponds to a stray-field potential associated to $\mu$ in $B_1^+(0)$.

The Euler-Lagrange equation for the minimiser of $E_\epsilon^\gamma$ is
therefore
\begin{equation} \label{eqn:core_Euler-Lagrange}
\epsilon \varphi'' = u' \sin \varphi \quad \text{in $(-1, 0)$ and $(0, 1)$}
\end{equation}
for any continuous function $\varphi : (-1, 1) \to \R$ with $\mu = \cos \varphi$,
similarly to Section \ref{sect:Euler-Lagrange}.

\subsection{Energy estimates}

We first prove the following preliminary estimate with
a direct construction. This result is similar to \cite{DKO} (see also \cite{Ig}).

\begin{proposition} \label{prop:core_construction}
Let $\beta \in (0,1)$. Then there exists a constant $C_0>0$ (depending on $\beta$) such that for
every $\gamma \in (\beta, 2 - \beta)$
and every $\epsilon \in (0, \frac{1}{2}]$,
\[
\inf_{M_\gamma} E_\epsilon^\gamma \le \frac{\pi \gamma^2}{2\log \frac{1}{\delta}} + \frac{C_0}{\left(\log \frac{1}{\delta}\right)^2}.
\]
\end{proposition}

\begin{proof}
Consider the function $\mu:(-1, 1)\to [1-\gamma, 1]$ given by
\[
\mu(x_1) = 1 - \gamma \frac{\log (x_1^2 + \delta^2) - \log \delta^2}{\log(1 + \delta^2) - \log \delta^2}\in [1-\gamma, 1] \quad \textrm{for } x_1\in (-1, 1).
\]
Then $1 + \mu \ge 2 - \gamma$ and thus
\[
1 - (\mu(x_1))^2 \ge \gamma(2 - \gamma) \frac{\log (x_1^2 + \delta^2) - \log \delta^2}{\log(1 + \delta^2) - \log \delta^2}.
\]
Therefore,
\[
\begin{split}
\int_{-1}^1 \frac{(\mu')^2}{1 - \mu^2} \, dx_1 & \le \frac{2\gamma^2}{\gamma (2 - \gamma) \log \sqrt{\frac{1}{\delta^2} + 1}} \int_{-1}^1 \frac{x_1^2}{(x_1^2 + \delta^2)^2 \log \left(\frac{x_1^2}{\delta^2} + 1\right)} \, dx_1 \\
& \le \frac{2\gamma^2}{\delta \gamma (2 - \gamma) \log \sqrt{\frac{1}{\delta^2} + 1}} \int_{-\infty}^\infty  \frac{t^2}{(t^2 + 1)^2 \log (t^2 + 1)} \, dt.
\end{split}
\]

Define $w : B_1^+(0) \to \R$ by
\[
w(r \cos \theta, r \sin \theta) = \mu(r)
\]
for $0 < r \le 1$ and $0 < \theta < \pi$. Then
\[
\begin{split}
\int_{B_1^+(0)} |\nabla w|^2 \, dx & = \frac{\pi \gamma^2}{\left(\log \sqrt{\frac{1}{\delta^2} + 1}\right)^2} \int_0^1 \frac{r^3}{(r^2 + \delta^2)^2} \, dr \\
& = \frac{\pi \gamma^2}{2\left(\log \sqrt{\frac{1}{\delta^2} + 1}\right)^2} \int_{\delta^2}^{1 + \delta^2} \frac{t - \delta^2}{t^2} \, dt \le \frac{\pi \gamma^2}{\log \sqrt{\frac{1}{\delta^2} + 1}}.
\end{split}
\]
Hence
\[
\begin{split}
E_\epsilon^\gamma(\mu) & \le \frac{\pi \gamma^2}{2\log \sqrt{\frac{1}{\delta^2} + 1}} \\
& \quad + \frac{\gamma^2}{\gamma (2 - \gamma) \log \frac{1}{\epsilon} \log \sqrt{\frac{1}{\delta^2} + 1}} \int_{-\infty}^\infty  \frac{t^2}{(t^2 + 1)^2 \log (t^2 + 1)} \, dt,
\end{split}
\]
which implies the statement of the proposition.
\end{proof}

We can match the leading order term in this inequality with an
estimate from below. Moreover, we obtain more information about the
behaviour of the minimiser $\mu$ of $E_\epsilon^\gamma$: in particular, we have a uniform $\dot{W}^{1,2}$-estimate for the difference
between the rescaled stray field potential $u \log \frac 1 \delta$ (associated to $\mu$ via \eqref{eqn:core_harmonic}--\eqref{eqn:core_boundary2})
and the map
\[
(x_1, x_2) \mapsto \gamma \left(\arctan\left(\frac{x_2}{x_1}\right) - \frac{\pi x_1}{2|x_1|}\right), \quad (x_1, x_2)\in B_1^+(0).
\]
First, however, we need some information on the regularity of solutions of the Euler-Lagrange equation, especially at the boundary.

\begin{lemma} \label{lemma:core_regularity}
Let $\gamma \in (0, 2)$. Suppose that $\varphi \in W^{1, 2}(-1, 1)$ with $\cos \varphi(-1) = \cos \varphi(1) = 1 - \gamma$.
Further suppose that $u \in W^{1, 2}(B_1^+(0))$ is a function such that \eqref{eqn:core_harmonic}--\eqref{eqn:core_Euler-Lagrange}
are satisfied for $\mu = \cos \varphi$. Then $\nabla u$ is continuous in $B_1^+(0)$ and has a continuous extension to
$\partial B_1^+(0) \backslash \{0\}$.
\end{lemma}

\begin{proof}
Let $\tilde{\mu} = \cos \varphi - 1 + \gamma$ and consider the extension of $u$
to $\R_+^2$ and of $\tilde{\mu}$ to $\R$ given by
\[
u(x) = u\left(\frac{x}{|x|^2}\right) \quad \text{for } x \in \R_+^2 \text{ with } |x| > 1
\]
and
\[
\tilde{\mu}(x_1) = - \tilde{\mu}\left(\frac{1}{x_1}\right) \quad \text{for } x_1 \in (-\infty, -1) \cup (1, \infty).
\]
Then we have $\Delta u = 0$ in $\R_+^2$ and $\dd{u}{x_2} = - \tilde{\mu}'$ on $\R \times \{0\}$.
Since $\tilde{\mu}' \in L_\loc^2(\R)$, we conclude that
$u'(\blank, 0) \in L_\loc^p(\R)$ for any $p \in [1, 2)$
with standard regularity theory. Thus by \eqref{eqn:core_Euler-Lagrange}, we have
$\varphi'' \in L^p(-1, 1)$, and it follows that $\tilde{\mu}'$ is H\"older continuous on
$[-1, -r] \cup [r, 1]$ for every $r \in (0, 1)$. The extension is then locally H\"older continuous in $\R \backslash \{0\}$, and
using standard regularity theory once more, we conclude that $\nabla u$ is
continuous away from 0.
\end{proof}

\begin{theorem} \label{thm:core_energy_estimate}
For any $\beta \in (0, 1)$, there exists a constant $C>0$ (depending on $\beta$) such that the following holds true for every $\eps\in (0, \frac 1 2]$.
Suppose that $\gamma \in (\beta, 2 - \beta)$ and
$\mu \in M_\gamma$ is the minimiser of $E_\epsilon^\gamma$.
Then
\begin{equation} \label{eqn:core_energy_estimate}
\left|E_\epsilon^\gamma(\mu) - \frac{\pi \gamma^2}{2 \log \frac{1}{\delta}}\right| \le \frac{C}{\left(\log \frac{1}{\delta}\right)^2}. 
\end{equation}
Let $u \in W^{1, 2}(B_1^+(0))$ be the solution of
\eqref{eqn:core_harmonic}--\eqref{eqn:core_boundary2}.
Then
\begin{equation} \label{eqn:core_u_estimates}
\epsilon \int_{-1}^1 \frac{(\mu')^2}{1 - \mu^2} \, dx_1 + \int_{B_1^+(0) \backslash B_\delta(0)} \left|\nabla u(x) - \frac{\gamma x^\perp}{|x|^2 \log \frac{1}{\delta}}\right|^2 \, dx + \int_{B_\delta(0)} |\nabla u|^2\, dx \le \frac{C}{\left(\log \frac{1}{\delta}\right)^2}
\end{equation}
and
\begin{equation} \label{eqn:core_boundary_estimate}
\int_{\partial^+ B_1(0)} |\nabla u|^2 \, d\sigma \le \frac{C}{\left(\log \frac{1}{\delta}\right)^2}.
\end{equation}
\end{theorem}

\begin{proof}
We already know, by Proposition \ref{prop:core_construction},
that there exists a constant $C_0>0$, depending only on $\beta$, such
that
\begin{equation} \label{eqn:upper_energy_estimate}
E_\epsilon^\gamma(\mu) \le \frac{\pi \gamma^2}{2 \log \frac{1}{\delta}} + \frac{C_0}{\left(\log \frac{1}{\delta}\right)^2}.
\end{equation}
We want to prove an estimate from below of the same order. To this end,
we choose some $\varphi \in W^{1, 2}(-1, 1)$ with $\mu = \cos \varphi$.
Consider the function $\xi : B_1^+(0) \to \R$,
defined in polar coordinates by
\[
\xi(r \cos \theta, r \sin \theta) = \begin{cases}
\displaystyle \frac{\gamma (\theta - \pi/2)}{\log \frac{1}{\delta}} & \text{if $r \ge \delta$}, \\[4\jot]
\displaystyle \frac{\gamma r(\theta - \pi/2)}{\delta \log \frac{1}{\delta}} & \text{if $0 < r < \delta$}.
\end{cases}
\] 
According to Proposition \ref{prop_min}, we have
\[
\begin{split}
\frac{\gamma^2 \pi}{\log \frac{1}{\delta}} & = \frac{\gamma \pi}{2\log \frac{1}{\delta}} \left(\int_{-1}^0 \mu' \, dx_1 - \int_0^1 \mu' \, dx_1\right) \\
& = \int_{-1}^1 \xi(x_1, 0) \mu'(x_1) \, dx_1 - \int_{-\delta}^\delta \left(\frac{\gamma \pi x_1}{2|x_1| \log \frac{1}{\delta}} + \xi(x_1, 0)\right) \mu'(x_1) \, dx_1.
\end{split}
\]
Using Lemma \ref{lemma:core_estimate} and \eqref{eqn:upper_energy_estimate},
we find a constant $C_1 = C_1(\beta)$
such that
\[
\left|\int_{-\delta}^\delta \left(\frac{\gamma \pi x_1}{2|x_1| \log \frac{1}{\delta}} + \xi(x_1, 0)\right) \mu'(x_1) \, dx_1\right| \le \frac{\gamma \pi}{2\log \frac{1}{\delta}} \int_{-\delta}^\delta |\varphi'| |\sin \varphi| \, dx_1 \le \frac{C_1}{\left(\log \frac{1}{\delta}\right)^2}.
\]
Moreover, by \eqref{eqn:core_harmonic}--\eqref{eqn:core_boundary2}, we have
\[
\int_{-1}^1 \xi(x_1, 0) \mu'(x_1) \, dx_1 {=} - \int_{(-1, 1) \times \{0\}} \xi \dd{u}{x_2} \, dx_1 = \int_{B_1^+(0)} \nabla \xi \cdot \nabla u \, dx.
\]
Combining these estimates, we obtain
\begin{equation} \label{eqn:mixed_term}
\int_{B_1^+(0)} \nabla \xi \cdot \nabla u \, dx \ge \frac{\gamma^2 \pi}{\log \frac{1}{\delta}} - \frac{C_1}{\left(\log \frac{1}{\delta}\right)^2}.
\end{equation}

We compute
\[
|\nabla \xi|^2=\frac{\gamma^2}{\delta^2 (\log \delta)^2}\left(1+(\theta-\frac{\pi}{2})^2\right) \quad \text{if $r<\delta$}
\]
and
\[
|\nabla \xi|^2=\frac{\gamma^2}{r^2 \left(\log \frac{1}{\delta}\right)^2} \quad \text{if $r>\delta$.}
\]
This implies
\[
\int_{B_1^+(0)} |\nabla \xi|^2 \, dx = \frac{\gamma^2 \pi}{\log \frac{1}{\delta}} + \frac{\gamma^2 \pi}{\left(\log \frac{1}{\delta}\right)^2} \left(\frac{1}{2} + \frac{\pi^2}{24}\right) \le \frac{\gamma^2 \pi}{\log \frac{1}{\delta}} + \frac{\gamma^2 \pi}{\left(\log \frac{1}{\delta}\right)^2}.
\]
Hence, by \eqref{eqn:upper_energy_estimate} and \eqref{eqn:mixed_term},
\[
\int_{B_1^+(0)} |\nabla u - \nabla \xi|^2 \, dx \le 2E_\epsilon^\gamma(\mu) - \frac{\gamma^2 \pi}{\log \frac{1}{\delta}} + \frac{2C_1 + \pi{\gamma^2}}{\left(\log \frac{1}{\delta}\right)^2} \le \frac{C_2}{\left(\log \frac{1}{\delta}\right)^2},
\]
where $C_2 = 2C_0 +2C_1+ \pi{\gamma^2}$. It follows in particular that
\[
\int_{B_\delta^+(0)} |\nabla u|^2 \, dx \le 2\int_{B_\delta^+(0)} |\nabla u - \nabla \xi|^2 \, dx + 2\int_{B_\delta^+(0)} |\nabla \xi|^2 \, dx \le \frac{2C_2 + 2\pi{\gamma^2}}{\left(\log \frac{1}{\delta}\right)^2}
\]
and
\[
\int_{B_1^+(0) \backslash B_\delta(0)} \left|\nabla u(x) - \frac{\gamma x^\perp}{|x|^2 \log \frac{1}{\delta}}\right|^2 \, dx\leq \frac{C_2}{\left(\log \frac{1}{\delta}\right)^2},
\]
since $\nabla \xi=\frac{\gamma x^\perp}{|x|^2 \log \frac{1}{\delta}}$ if $r>\delta$. 
Furthermore, we have
\[
\begin{split}
\int_{B_1^+(0)} |\nabla u|^2 \, dx & = \int_{B_1^+(0)} |\nabla u - \nabla \xi|^2 \, dx - \int_{B_1^+(0)} |\nabla \xi|^2 \, dx 
+ 2\int_{B_1^+(0)} \nabla \xi \cdot \nabla u \, dx \\
& \ge \frac{\gamma^2 \pi}{\log \frac{1}{\delta}} - \frac{\gamma^2 \pi + 2C_1}{\left(\log \frac{1}{\delta}\right)^2}.
\end{split}
\]
If we combine this with \eqref{eqn:upper_energy_estimate}, then we obtain
\[
\epsilon \int_{-1}^1 \frac{(\mu')^2}{1 - \mu^2} \, dx_1 \le 2E_\epsilon^\gamma(\mu) - 
\int_{B_1^+(0)} |\nabla u|^2 \, dx \le \frac{2C_0 + 2C_1 + \gamma^2 \pi}{\left(\log \frac{1}{\delta}\right)^2}.
\]
Thus we have proved inequalities \eqref{eqn:core_energy_estimate} and \eqref{eqn:core_u_estimates}.

Finally, we apply Lemma \ref{lemma:Pohozaev} and let $r\nearrow 1$. Then by Lemma \ref{lemma:core_regularity},
we obtain \eqref{eqn:core_boundary_estimate}.
\end{proof}

\subsection{Behaviour under small perturbations}

Next we want to understand how the number $\inf_{M_\gamma} E_\epsilon^\gamma$
changes when we vary $\gamma$, and how the energy
changes when we perturb the boundary condition
\eqref{eqn:core_boundary2}. In particular, we prove the
following statements.

\begin{proposition} \label{prop:change_gamma}
Let $\beta \in (0,1)$. There exists a constant $C(\beta)>0$ such that for all
$\gamma_1, \gamma_2 \in (\beta, 2 - \beta)$ and every $\eps\in (0, \frac 12]$,
\[
\inf_{M_{\gamma_2}} E_\epsilon^{\gamma_2} - \frac{\pi \gamma_2^2}{2\log \frac{1}{\delta}}\le \inf_{M_{\gamma_1}} E_\epsilon^{\gamma_1} - \frac{\pi \gamma_1^2}{2\log \frac{1}{\delta}} + \frac{C|\gamma_2 - \gamma_1|}{\left(\log \frac{1}{\delta}\right)^2}.
\]
\end{proposition}

In other words, the function
\begin{equation} \label{eqn:def_g}
g: (0, 1) \to \R, \ \gamma \mapsto \inf_{M_{\gamma}} E_\epsilon^{\gamma} - \frac{\pi \gamma^2}{2\log \frac{1}{\delta}},
\end{equation}
is locally Lipschitz continuous with Lipschitz constant of order $O(1/(\log \delta)^2)$.

\begin{proof}
First note that it suffices to prove the inequality when $|\gamma_2 - \gamma_1|$ is small,
for otherwise it follows from Theorem \ref{thm:core_energy_estimate}.

Let $\mu_1 \in M_{\gamma_1}$ be the minimiser of $E_\epsilon^{\gamma_1}$
and let $v_1 \in W^{1, 2}(B_1^+(0))$
be the corresponding solution of \eqref{eqn:core_harmonicv}--\eqref{eqn:core_boundary2v} in Proposition \ref{prop_min}.
Define
\[
\mu_2 = \frac{\gamma_2}{\gamma_1}(\mu_1 - 1) + 1
\]
and
\[
v_2 = \frac{\gamma_2}{\gamma_1}(v_1 - 1) + 1.
\]
Then $\mu_2 \in M_{\gamma_2}$ and $v_2$ is the unique solution of
\eqref{eqn:core_harmonicv}--\eqref{eqn:core_boundary2v} associated to $\mu_2$, so that,
by Proposition~\ref{prop_min},
\[
E_\epsilon^{\gamma_2}(\mu_2) = \frac{\epsilon}{2} \int_{-1}^1 \frac{(\mu_2')^2}{1 - \mu_2^2} \, dx_1 + \frac{1}{2} \int_{B_1^+(0)} |\nabla v_2|^2 \, dx.
\]
We have
\[
1 + \mu_2 = \frac{\gamma_2}{\gamma_1}(1 + \mu_1) + \frac{2(\gamma_1 - \gamma_2)}{\gamma_1}
\]
and
\[
1 - \mu_2 = \frac{\gamma_2}{\gamma_1} (1 - \mu_1).
\]
Hence
\[
\frac{(\mu_2')^2}{1 - \mu_2^2} = \frac{\frac{(\mu_1')^2}{1 - \mu_1^2}}{1 + \frac{2(\gamma_1 - \gamma_2)}{\gamma_2(1 + \mu_1)}}.
\]
Under the assumptions of the proposition, we have $\beta < \gamma_1, \gamma_2 < 2 - \beta$ and
$\beta < 1 + \mu_1 \le 2$ throughout $(-1, 1)$. Therefore, we have a constant
$C_1 = C_1(\beta)$ such that
\[
\frac{1}{1 + \frac{2(\gamma_1 - \gamma_2)}{\gamma_2(1 + \mu_1)}} \le \frac{\gamma_2^2}{\gamma_1^2} + C_1 |\gamma_2 - \gamma_1|,
\]
provided that $|\gamma_2 - \gamma_1|$ is small enough.

We now have
\[
E_\epsilon^{\gamma_2}({\mu}_2) \le \left(\frac{\gamma_2^2}{\gamma_1^2} + C_1|\gamma_ 2 - \gamma_1|\right) \frac{\epsilon}{2} \int_{-1}^1 \frac{(\mu_1')^2}{1 - \mu_1^2} \, dx_1 + \frac{\gamma_2^2}{2\gamma_1^2} \int_{B_+^1(0)} |\nabla v_1|^2 \, dx.
\]
Combining this with \eqref{eqn:core_energy_estimate}
and \eqref{eqn:core_u_estimates},
we obtain the inequality
\[
g(\gamma_2)\leq g(\gamma_1) + \frac{\gamma_2^2-\gamma_1^2}{\gamma_1^2} g(\gamma_1) + \frac{C_2|\gamma_1-\gamma_2 |}{\left(\log \frac{1}{\delta}\right)^2}
\]
for a constant $C_2 = C_2(\beta)$ by Theorem \ref{thm:core_energy_estimate},
where $g$ is the function defined in \eqref{eqn:def_g}.
Finally, we know that
\[
g(\gamma_1) \le \frac{C_3}{\left(\log \frac{1}{\delta}\right)^2}
\]
for another constant $C_3$ depending only on $\beta$ (by Theorem \ref{thm:core_energy_estimate} again).
Hence we obtain the desired inequality.
\end{proof}

\begin{remark}
\label{rem:core_estimate}
Recall the discussion in Section \ref{sec:rescale_tail}. Consider a N\'eel wall at the point $a_n$ such that $m_1$
makes a transition from $1 - \gamma_n$ to $1$ and back. If we modify the wall, changing $\gamma_n$ to
$\gamma_n - \zeta_n$ with $\zeta_n = \frac{\lambda_n}{\log \frac 1 \delta}$, then the change of the energy (to leading order, rescaled by $(\log \delta)^2$) is
$$\left(\log \frac{1}{\delta}\right)^2 \left(\inf_{M_{\gamma_n-\zeta_n}} E_\epsilon^{\gamma_n-\zeta_n}-\inf_{M_{\gamma_n}} E_\epsilon^{\gamma_n}\right)= -\pi \gamma_n \lambda_n+o(1)
\quad \textrm{as}\quad \eps\searrow 0.$$
This is the phenomenon that leads to the core-tail interaction term in
the renormalised energy.
\end{remark}

\begin{lemma} \label{lemma:core_lower_estimate}
Let $C_0 > 0$, $\beta \in (0,1)$ and $q > 2$. Then there exists a constant $C>0$ (depending only on $C_0$, $\beta$ and $q$) such
that for any $\gamma \in (\beta, 2 - \beta)$, any $\eps\in (0, \frac 1 2]$, and
any $\eta \in (0, C_0)$, the following holds true.
Suppose that $\mu \in M_\gamma$ and $u \in W^{1, 2}(B_1^+(0))$ with
$\Delta u = 0$ in $B_1^+(0)$ and $\dd{u}{x_2} = - \mu'$ on $(-1, 1) \times \{0\}$.
Suppose further that
\[
\|x \cdot \nabla u\|_{L^q(\partial^+ B_1(0))} \le \frac{\eta}{\log \frac{1}{\delta}}
\]
and
\begin{equation} \label{eqn:difference_to_v_0}
\left\|\nabla u - \frac{\gamma x^\perp}{|x|^2 \log \frac{1}{\delta}}\right\|_{L^2(B_1^+(0) \backslash B_\delta(0))} + \|\nabla u\|_{L^2(B_\delta^+(0))} \le \frac{C_0}{\log \frac{1}{\delta}}.
\end{equation}
Then
\[
\frac{\epsilon}{2} \int_{-1}^1 \frac{(\mu')^2}{1 - \mu^2} \, dx_1 + \frac{1}{2} \int_{B_1^+(0)} |\nabla u|^2 \, dx \ge \inf_{M_\gamma} E_\epsilon^\gamma - \frac{C\eta}{\left(\log \frac{1}{\delta}\right)^2}.
\]
\end{lemma}

\begin{proof}
Consider the function $\tilde{w} \in W^{1, 2}(B_1^+(0))$ with
$\nabla \tilde{w} = \nabla^\perp u$ and $\tilde{w}(x_1, 0) = \mu(x_1)$
for $x_1 \in (-1, 1)$. Let $S_\theta = \set{(\cos t, \sin t)}{t \in (0, \theta)}$
for $0 < \theta < {\frac\pi 2}$. Note that
\[
\tilde{w}(\cos \theta, \sin \theta) - \mu(1) = \int_{S_\theta} x \cdot \nabla u \, d\sigma \le \frac{\theta^{1 - 1/q} \eta}{\log \frac{1}{\delta}}, \quad \theta\in (0, {\frac\pi 2}).
\]
Similarly, we find that
\[
\tilde{w}(- \cos \theta, \sin \theta) - \mu(-1) \le \frac{\theta^{1 - 1/q} \eta}{\log \frac{1}{\delta}}, \quad \theta\in (0, {\frac\pi 2}).
\]
Thus if we define
\[
g(x) = \frac{2\eta x_2^{1 - 1/q}}{\log \frac{1}{\delta}}
\]
and
\[
w = \tilde{w} - g,
\]
then we have that $w\leq \max\{\mu(\pm 1)\}\leq 1-\gamma$ on $\partial B_1^+(0)$ (because $2\sin \theta\geq \theta$ for $\theta\in (0, \frac \pi 2)$). Moreover, $w \in W_\mu^{1, 2}(B_1^+(0))$. Indeed, we have
\[
\int_{B_1^+(0)} |\nabla w|^2 \, dx = \int_{B_1^+(0)} |\nabla u|^2 \, dx + 2\int_{B_1^+(0)} \nabla u \cdot \nabla^\perp g \, dx + \int_{B_1^+(0)} |\nabla g|^2 \, dx.
\]
For any $p \in [1, q)$, we have a constant $C_1 = C_1(p)$ such that
\[
\|\nabla g\|_{L^p(B_1^+(0))} \le \frac{C_1 \eta}{\log \frac{1}{\delta}}.
\]
For $p>2$, inequality \eqref{eqn:difference_to_v_0} gives another constant $C_2 = C_2(p, C_0)$ such that
\[
\|\nabla u\|_{L^{p/(p - 1)}(B_1^+(0))} \le \frac{C_2}{\log \frac{1}{\delta}}.
\]
We conclude, using H\"older's inequality, that
\[
\int_{B_1^+(0)} \nabla u \cdot \nabla^\perp g \, dx\leq \frac{C_1 C_2 \eta}{\left(\log \frac{1}{\delta}\right)^2},
\]
using some fixed $p\in (2, q)$. Therefore,
\[
\int_{B_1^+(0)} |\nabla w|^2 \, dx \le \int_{B_1^+(0)} |\nabla u|^2 \, dx + \frac{C_3 \eta}{\left(\log \frac{1}{\delta}\right)^2}
\]
for a number $C_3>0$ that depends only on $q$ and $C_0$. Now the statement of the lemma follows.
\end{proof}

\begin{corollary} \label{cor:core_estimate}
Let $\beta \in (0, \frac 2 3)$, $C_0> 0$, and $q > 2$. Then there exists a constant $C>0$ such
that for any $\gamma \in (2\beta, 2 - 2\beta)$, any $\eps\in (0, \frac 1 2]$, any $\eta \in (0, C_0)$, and any $\zeta \in (-C_0, C_0)$,
the following holds true.
Suppose that $\mu \in W^{1, 2}(-1, 1)$ with
$\mu(\pm 1) \le 1 - \gamma + \frac{\zeta}{\log \frac{1}{\delta}}$ and $\mu(0) = 1$.
Let $u \in W^{1, 2}(B_1^+(0))$ be a function with
$\Delta u = 0$ in $B_1^+(0)$ and $\dd{u}{x_2} = - \mu'$ on $(-1, 1) \times \{0\}$.
Suppose that
\[
\|x \cdot \nabla u\|_{L^q(\partial^+ B_1(0))} \le \frac{\eta}{\log \frac{1}{\delta}}
\]
and
\[
\left\|\nabla u - \frac{\gamma x^\perp}{|x|^2 \log \frac{1}{\delta}}\right\|_{L^2(B_1^+(0) \backslash B_\delta(0))} + \|\nabla u\|_{L^2(B_\delta^+(0))} \le \frac{C_0}{\log \frac{1}{\delta}}.
\]
Then
\[
\frac{\epsilon}{2} \int_{-1}^1 \frac{(\mu')^2}{1 - \mu^2} \, dx_1 + \frac{1}{2} \int_{B_1^+(0)} |\nabla u|^2 \, dx \ge \inf_{M_\gamma} E_\epsilon^\gamma - \frac{\pi \gamma \zeta+C\eta}{\left(\log \frac{1}{\delta}\right)^2} - \frac{C}{\left(\log \frac{1}{\delta}\right)^3}.
\]
\end{corollary}

\begin{proof} Set $\tilde{\gamma}=\zeta/\log \frac{1}{\delta}$. 
We have $\mu \in M_{\gamma - \tilde{\gamma}}$, thus for $\eps$ small enough, we can first apply
Lemma \ref{lemma:core_lower_estimate}  for $\gamma - \tilde{\gamma}\in (\beta, 2-\beta)$ instead of $\gamma$. We obtain
\[
\frac{\epsilon}{2} \int_{-1}^1 \frac{(\mu')^2}{1 - \mu^2} \, dx_1 + \frac{1}{2} \int_{B_1^+(0)} |\nabla u|^2 \, dx \ge \inf_{M_{\gamma - \tilde{\gamma}}} E_\epsilon^{\gamma -\tilde{\gamma}} - \frac{C_1 \eta}{\left(\log \frac{1}{\delta}\right)^2}
\]
for a constant $C_1 = C_1(\beta, C_0, q)$.
From Proposition \ref{prop:change_gamma}, it follows that
\[
\inf_{M_{\gamma - \tilde{\gamma}}} E_\epsilon^{\gamma - \tilde{\gamma}} - \inf_{M_\gamma} E_\epsilon^{\gamma} \ge -\frac{\pi \gamma\zeta}{\left(\log \frac{1}{\delta}\right)^2} -
\frac{C_2}{\left(\log \frac{1}{\delta}\right)^3}
\]
for another constant $C_2 = C_2(\beta, C_0)$.
Now it suffices to combine these estimates.
\end{proof}

\subsection{The profile near the core}

Minimisers of $E_\epsilon^\gamma$ give a good approximation
of the behaviour of N\'eel walls near the core. When we let $\epsilon \searrow 0$,
then, after renormalisation, we have convergence of those minimisers to a
specific profile. More precisely, we have the following.

\begin{theorem} \label{thm:core_convergence}
For $\epsilon \in (0, \frac{1}{2}]$ and $\gamma\in (0, 2)$, let $\mu_\epsilon \in M_\gamma$
be the minimiser of $E_\epsilon^\gamma$. Let $u_\epsilon$ be the corresponding
solution of \eqref{eqn:core_harmonic}--\eqref{eqn:core_boundary2}
with
\be
\label{eqn:average}
\int_{B_1^+(0)\setminus B_{1/2}(0)} u_\epsilon \, dx = 0.
\ee
Suppose that $0<r< R < 1$. Then, as $\eps\searrow 0$,
\[
u_\epsilon(x) \log \frac{1}{\delta} \rightharpoonup \gamma \left(\arctan\left(\frac{x_2}{x_1}\right) - \frac{\pi x_1}{2|x_1|}\right) \quad \textrm{weakly in $W^{2, 2}(B_R^+(0) \backslash B_r(0))$}
\]
 and
\[
(\mu_\epsilon(x_1) - 1 + \gamma) \log \frac{1}{\delta} \to \gamma \log \frac{1}{|x_1|} \quad \textrm{strongly in $W^{1,2}((-R, -r) \cup (r, R))$}.
\]
\end{theorem}

\begin{remark}
For a related problem (concerning N\'eel walls in the
presence of an anisotropy but without the confinement to an interval),
Melcher \cite{Me1, Me2} proved similar results on the
profile of N\'eel walls (with methods different from ours).
\end{remark}

The proof of Theorem \ref{thm:core_convergence} relies on Lemma \ref{lemma:higher_estimates}. Therefore, we first need
to show that the assumption $\sin \varphi_\epsilon \neq 0$ in that lemma is satisfied
for a function $\varphi_\epsilon$ with $\mu_\epsilon = \cos \varphi_\epsilon$,
at least away from $x_1=0$, for sufficiently small values of
$\epsilon$. Because we will use this result in a slightly
different context as well, we formulate it more generally.

\begin{lemma} \label{lemma:m1_estimate}
There exists a constant $C>0$ such that the following holds true.
Suppose that $\Omega \subset \R_+^2$ is a
Lipschitz domain with $(-1, 1) \times \{0\} \subset \partial \Omega$
and let $\nu$ be the outer normal vector on $\partial \Omega$.
Let $\epsilon \in (0, \frac{1}{2}]$. Suppose that $\mu \in W^{1, 2}(-1, 1)$
and let $u \in W^{1, 2}(\Omega)$ be a solution of
\begin{alignat*}{2}
\Delta u & = 0 & \quad & \text{in $\Omega$}, \\
\dd{u}{x_2} & = - \mu' && \text{on $(-1, 1) \times \{0\}$}, \\
\nu \cdot \nabla u & = 0 && \text{on $\partial \Omega \backslash ((-1, 1) \times \{0\})$}.
\end{alignat*}
Let $x_1 \in (\delta - 1, 1 - \delta)$ and set
\[
\Sigma = \left(B_\delta^+(x_1,0)\cap \Omega\right) \cup \set{y \in \Omega \cap B_2^+(0)}{y_2 \ge |y_1 - x_1| \text{ or $|y|\in (\textstyle \frac 32,2)$}}.
\]
Then
\[
|\mu(x_1) - \mu(-1)| \le C \left(\log \frac{1}{\delta} \int_\Sigma |\nabla u|^2 \, dx\right)^{1/2} + 
C \left(\log \frac 1 \delta \int_{x_1 - \delta}^{x_1 + \delta} \eps (\mu')^2 \, dt\right)^{1/2}.
\]
\end{lemma}

\begin{proof}
Define
\[
\psi(t) = \begin{cases}
1 & \text{if $t \le x_1 - \delta$}, \\
\frac{x_1 - t}{2\delta} + \frac{1}{2} & \text{if $x_1 - \delta < t < x_1 + \delta$}, \\
0 & \text{if $t \ge x_1 + \delta$}.
\end{cases}
\]
Furthermore, let
\[
\eta(\theta) = \begin{cases}
0 & \text{if $\theta \le \frac{\pi}{4}$}, \\
\frac{2\theta}{\pi} - \frac{1}{2} & \text{if $\frac{\pi}{4} < \theta < \frac{3\pi}{4}$}, \\
1 & \text{if $\theta \ge \frac{3\pi}{4}$},
\end{cases}
\]
and define $\tilde{\chi} : \R_+^2 \to \R$ by
\[
\tilde{\chi}(x_1 + r \cos \theta, r \sin \theta) = (1 - \eta(\theta)) \psi(x_1 + r) + \eta(\theta) \psi(x_1 - r).
\]
Finally, let
\[
\chi(x) = \begin{cases}
\tilde{\chi}(x) & \text{if $|x| \le \frac 32$}, \\
2(2 - |x|) \tilde{\chi}(x) & \text{if $\frac 32 < |x| < 2$}, \\
0 & \text{if $|x| \ge 2$}.
\end{cases}
\]
Then $\Omega \cap \supp \nabla \chi \subset \overline{\Sigma}$ and
\[
\int_{\R_+^2} |\nabla \chi|^2 \, dx \le C_1 \log \frac{1}{\delta}
\]
for some universal constant $C_1$.

We have
\[
\mu(x_1) - \mu(-1) = \int_{-1}^{x_1} \mu'(t) \, dt.
\]
Moreover,
\[
\begin{split}
\left|\int_{-1}^{x_1} \mu'(t) \, dt - \int_{-1}^1 \chi(t, 0) \mu'(t) \, dt\right| & \le \int_{x_1 - \delta}^{x_1 + \delta} |\mu'| \, dt \\
& \le \left(2\delta \int_{x_1 - \delta}^{x_1 + \delta} (\mu')^2 \, dx\right)^{1/2}.
\end{split}
\]
Using the boundary value problem for $u$, we find that
\[
\left|\int_{-1}^1 \chi(t, 0) \mu'(t) \, dt\right| = \left|\int_{\Omega} \nabla \chi \cdot \nabla u \, dx\right| \le \left(C_1 \log \frac{1}{\delta} \int_{\Sigma} |\nabla u|^2 \, dx\right)^{1/2}.
\]
Combining these estimates, we obtain the
desired inequality.
\end{proof}

\begin{proof}[Proof of Theorem \ref{thm:core_convergence}]
We choose $\varphi_\epsilon : (-1, 1) \to [0, \pi)$ such
that $\varphi_\epsilon(0) = 0$ and $\mu_\epsilon = \cos \varphi_\epsilon$.
Using Theorem \ref{thm:core_energy_estimate} and Lemma \ref{lemma:m1_estimate} (applied for $\Omega=B_1^+(0)$),
we see that for any $r, R \in (0, 1)$ and any sufficiently small $\beta > 0$, we have
$|\mu_\epsilon -1 + \gamma| < \beta$ and $|\sin \varphi_\epsilon|\geq \beta$ in $(-R, -r) \cup (r, R)$ if $\epsilon > 0$ is small enough.
Therefore, we can apply Lemma \ref{lemma:higher_estimates}
in $(-R, -r)$ and in $(r, R)$.

By Theorem \ref{thm:core_energy_estimate}, the Poincar\'e inequality, and \eqref{eqn:average}, the functions
$u_\epsilon \log \frac{1}{\delta}$ are uniformly bounded in the space
$W^{1, 2}(B_1^+(0) \backslash B_r(0))$ for all $r > 0$.
Hence, by Lemma \ref{lemma:higher_estimates}, they are
uniformly bounded in $W^{2, 2}(B_R^+(0) \backslash B_r(0))$
whenever $0 < r < R < 1$. Hence there exists a sequence
$\epsilon_k \to 0$ as $k\to \infty$ such that we have weak convergence
\[
u_{\epsilon_k} \log \frac{1}{\delta_k} \rightharpoonup  w
\]
in $W^{2, 2}(B_R^+(0) \backslash B_r(0))$ for all $r, R \in (0, 1)$,
where $\delta_k = \epsilon_k \log \frac{1}{\epsilon_k}$.
The limit
\[
w \in \bigcap_{0<r<R< 1} W^{2, 2}(B_R^+(0) \backslash B_r(0))
\]
is harmonic in $B_1^+(0)$ and satisfies the boundary
condition $x \cdot \nabla w = 0$ on $\partial^+ B_1(0)$.
According to Lemma \ref{lemma:higher_estimates},
we also have
\[
\limsup_{k \to \infty} \left(\left(\log \frac{1}{\delta_k}\right)^2 \epsilon_k \int_r^R (\varphi_{\epsilon_k}'')^2 \, dx_1\right) < \infty.
\]
It follows from \eqref{eqn:core_Euler-Lagrange} that
\[
\limsup_{k \to \infty} \left(\frac{\left(\log \frac{1}{\delta_k}\right)^2}{\epsilon_k} \int_r^R (u_{\epsilon_k}')^2 \, dx_1\right) < \infty.
\]
Thus $w$ is constant on $(0, 1) \times \{0\}$, and we can
prove the same on $(-1, 0) \times \{0\}$.

Define
\[
\tilde{w}(x) = w(x) - \gamma \left(\arctan\left(\frac{x_2}{x_1}\right) - \frac{\pi x_1}{2|x_1|}\right), \quad x \in B_1^+(0).
\]
Since
\[
\int_{B_1^+(0)} \left|\nabla w - \frac{\gamma x^\perp}{|x|^2}\right|^2 \, dx < \infty
\]
by Theorem \ref{thm:core_energy_estimate}, it follows that
$\tilde{w} \in W^{1, 2}(B_1^+(0))$. We conclude that $\tilde{w}(\blank, 0) \in
H^{1/2}(-1, 1)$. Moreover, the trace $\tilde{w}(\blank, 0)$ is
constant on $(-1, 0)$ and on $(0, 1)$.
But $H^{1/2}(-1, 1)$ does not allow any jumps;
hence $\tilde{w}(\blank, 0)$ is in fact constant on $(-1, 1)$.
We also have
$\Delta \tilde{w} = 0$ in $B_1^+(0)$ and $x \cdot \nabla \tilde{w} = 0$
on $\partial^+ B_1(0)$. Thus it follows that $\tilde{w}$ is constant
in $B_1^+(0)$. Because of \eqref{eqn:average}, we have $\tilde{w} = 0$.
That is,
\[
w(x) = \gamma \left(\arctan\left(\frac{x_2}{x_1}\right) - \frac{\pi x_1}{2|x_1|}\right) \quad \textrm{in} \quad B_1^+(0).
\]
As the limit is thus independent of the sequence $\epsilon_k$,
this implies the first claim.

We have
\[
\log \frac{1}{\delta} \, \dd{u_\epsilon}{x_2} (x_1, 0) \to \frac{\gamma}{x_1}
\]
strongly in $L^2(-R, -r)$ and in $L^2(r, R)$.
But since $\mu_\epsilon' = - \dd{u_\epsilon}{x_2}$, it follows
that $(\mu_\epsilon - 1 + \gamma) \log \frac{1}{\delta}$
converges strongly in $W^{1,2}(-R, -r)$ to 
$\lambda_- - \gamma \log |x_1|$
and in $W^{1,2}(r, R)$
and to $\lambda_+ -\gamma \log |x_1|$ for two constants
$\lambda_-, \lambda_+ \in [-\infty, \infty]$. It remains to determine these constants.

Choose a function $\chi \in C_0^\infty(B_1(-1, 0))$ with $(-1, 0) \not\in \supp \nabla \chi$.
Then
\[
\lim_{\epsilon \searrow 0} \left(\log \frac{1}{\delta} \int_{B_1^+(0)} \nabla \chi \cdot \nabla u_\epsilon \, dx\right) = \gamma \int_{B_1^+(0)} \nabla \chi \cdot \frac{x^\perp}{|x|^2} \, dx = - \gamma \int_{-1}^1 \frac{\chi(x_1)}{x_1} \, dx_1.
\]
On the other hand,
\[
\int_{B_1^+(0)} \nabla \chi \cdot \nabla u_\epsilon \, dx = \int_{-1}^1 \mu_\epsilon' \chi \, dx_1 = - \int_{-1}^1 (\mu_\epsilon - 1 + \gamma) \chi' \, dx_1
\]
by an integration by parts. Hence
\[
\begin{split}
\lim_{\epsilon \searrow 0} \left(\log \frac{1}{\delta} \int_{B_1^+(0)} \nabla \chi \cdot \nabla u_\epsilon \, dx\right) & = \int_{-1}^1 (\gamma\log |x_1| - \lambda_-) \chi' \, dx_1 \\
& = \chi(-1, 0)\lambda_- - \gamma \int_{-1}^1 \frac{\chi(x_1)}{x_1} \, dx_1,
\end{split}
\]
and we obtain $\lambda_- = 0$. Similarly
we show that $\lambda_+ = 0$.
\end{proof}

\subsection{The core energy}

We can now determine the values of the function $e$ in
Theorem \ref{thm1}, albeit not explicitly. They arise as the limits in
the following result for $\gamma_\pm = 1 \mp \cos \alpha$.

\begin{theorem} \label{thm:core_energy_limit}
For any $\gamma \in (0, 2)$, the limit
\[
e_\gamma = \lim_{\epsilon \searrow 0} \left(\left(\log \frac{1}{\delta}\right)^2 \inf_{M_\gamma} E_\epsilon^\gamma - \frac{\pi \gamma^2}{2} \log \frac{1}{\delta}\right)
\]
exists.
\end{theorem}

\begin{definition}
\label{def:e}
The function $e : \{\pm 1\} \to \R$ is defined by
\[
e(-1)=e_{1 + \cos \alpha} \quad \text{and} \quad e(1) = e_{1 - \cos \alpha}.
\]
\end{definition}

\begin{proof}[Proof of Theorem \ref{thm:core_energy_limit}]
Define
\[
f(\epsilon) = \left(\log \frac{1}{\delta}\right)^2 \inf_{M_\gamma} E_\epsilon^\gamma - \frac{\pi \gamma^2}{2} \log \frac{1}{\delta}.
\]
Fix $\epsilon > 0$ small enough and choose a number $R \in (1, \frac{1}{\delta})$.
Let $\mu \in M_\gamma$ be the minimiser of $E_\epsilon^\gamma$
and let $v \in W^{1, 2}(B_1^+(0))$ be the solution of
\eqref{eqn:core_harmonicv}--\eqref{eqn:core_boundary2v}.
Define
\[
\tilde{\mu}(x_1) = \begin{cases}
\displaystyle \frac{\gamma \log \frac{1}{|x_1|}}{\log \frac{1}{\delta}} + 1 - \gamma & \text{if $\frac{1}{R} \le |x_1| < 1$}, \\[4\jot]
\displaystyle \left(1 - \frac{\log R}{\log \frac{1}{\delta}}\right) \mu(R x_1) + \frac{\log R}{\log \frac{1}{\delta}} & \text{if $ |x_1| \le \frac{1}{R}$},
\end{cases}
\]
and
\[
\tilde{v}(x) = \begin{cases}
\displaystyle \frac{\gamma \log \frac{1}{|x|}}{\log \frac{1}{\delta}} + 1 - \gamma & \text{if $\frac{1}{R} \le |x| < 1$}, \\[4\jot]
\displaystyle \left(1 - \frac{\log R}{\log \frac{1}{\delta}}\right) v(Rx) + \frac{\log R}{\log \frac{1}{\delta}} & \text{if $ |x| \le \frac{1}{R}$}.
\end{cases}
\]
Then we have $\tilde{\mu} \in M_\gamma$ and
$\tilde{v} \in W_{\tilde{\mu}}^{1, 2}(B_1^+(0))$. Moreover, by Proposition \ref{prop_min},
\[
1 + \tilde{\mu}(x_1) \ge 1 + \mu(R x_1)\geq 2-\gamma>0
\]
and
\[
1 - \tilde{\mu}(x_1) = \left(1 - \frac{\log R}{\log \frac{1}{\delta}}\right) (1 - \mu(R x_1))>0
\]
for $|x_1| \le \frac{1}{R}$. Therefore, we compute
\[
\frac{\epsilon}{R} \int_{-1/R}^{1/R} \frac{(\tilde{\mu}')^2}{1 - \tilde{\mu}^2} \, dx_1 \le \left(1 - \frac{\log R}{\log \frac{1}{\delta}}\right) \epsilon \int_{-1}^1 \frac{(\mu')^2}{1 - \mu^2} \, dx_1.
\]
Furthermore,
\[
\int_{B_{1/R}^+(0)} |\nabla \tilde{v}|^2 \, dx = \left(1 - \frac{\log R}{\log \frac{1}{\delta}}\right)^2 \int_{B_1^+(0))} |\nabla v|^2 \, dx.
\]
We also observe that for $x_1 \in [-1, -\frac{1}{R}) \cup (\frac{1}{R}, 1]$,
\[
1 - \tilde{\mu}^2 = \gamma \left(1 - \frac{\log \frac{1}{|x_1|}}{\log \frac{1}{\delta}}\right) \left(2 - \gamma + \frac{\gamma \log \frac{1}{|x_1|}}{\log \frac{1}{\delta}}\right) \ge \gamma (2 - \gamma) \left(1 - \frac{\log \frac{1}{|x_1|}}{\log \frac{1}{\delta}}\right).
\]
Hence
\[
\begin{split}
\int_{(-1,-1/R)\cup (1/R, 1)} \frac{(\tilde{\mu}')^2}{1 - \tilde{\mu}^2} \, dx_1 & \le \frac{2\gamma}{(2 - \gamma) \log \frac{1}{\delta}} \int_{1/R}^1 \frac{dx_1}{x_1^2 \left(\log \frac{1}{\delta} - \log \frac{1}{x_1}\right)} \\
& =  \frac{2\gamma}{(2 - \gamma) \log \frac{1}{\delta}} \int_1^R \frac{ds}{\log \frac{1}{\delta} - \log s}.
\end{split}
\]
Define
\[
g(R) = \frac{\gamma}{(2 - \gamma) \log \frac{1}{\delta}} \int_1^R \frac{ds}{\log \frac{1}{\delta} - \log s}.
\]
Finally, we compute
\[
\int_{B_1^+(0) \backslash B_{1/R}(0)} |\nabla \tilde{v}|^2 \, dx = \frac{\pi \gamma^2 \log R}{\left(\log \frac{1}{\delta}\right)^2}.
\]
Let
\[
\tilde{\epsilon} = \tilde{\epsilon}(R) = \frac{\epsilon}{R} \left(1 - \frac{\log R}{\log \frac{1}{\delta}}\right)<\eps.
\]
Then we have
\[
\begin{split}
E_{\tilde{\epsilon}}^\gamma(\tilde{\mu}) & \le \frac{\tilde{\epsilon}}{2} \int_{-1}^1 \frac{(\tilde{\mu}')^2}{1 - \tilde{\mu}^2} \, dx_1 + \frac{1}{2} \int_{B_1^+(0)} |\nabla \tilde{v}|^2 \, dx \\
& \le \left(1 - \frac{\log R}{\log \frac{1}{\delta}}\right)^2 E_\epsilon^\gamma(\mu) + 
\frac{\epsilon g(R)}{R} \left(1 - \frac{\log R}{\log \frac{1}{\delta}}\right) + \frac{\pi \gamma^2 \log R}{2\left(\log \frac{1}{\delta}\right)^2}.
\end{split}
\]
It follows that
\[
f(\tilde{\epsilon}) \le \left(\log \frac{1}{\tilde{\delta}}\right)^2 \left(\left(1 - \frac{\log R}{\log \frac{1}{\delta}}\right)^2 E_\epsilon^\gamma(\mu) + 
\frac{\epsilon g(R)}{R} \left(1 - \frac{\log R}{\log \frac{1}{\delta}}\right) + \frac{\pi \gamma^2 \log R}{2\left(\log \frac{1}{\delta}\right)^2}\right)
- \frac{\pi \gamma^2}{2} \log \frac{1}{\tilde{\delta}},
\]
where $\tilde{\delta} = \tilde{\epsilon} \log \frac{1}{\tilde{\epsilon}}$.
Since we have equality for $R = 1$, we can use this inequality
to estimate the left-hand superdifferential
\[
f_-'(\epsilon) = \liminf_{s \nearrow \epsilon} \frac{f(s) - f(\epsilon)}{s - \epsilon}.
\]
Indeed, note first that
\[
\left.\frac{d\tilde{\epsilon}}{dR}\right|_{R = 1} = - \epsilon \left(1 + \frac{1}{\log \frac{1}{\delta}}\right)
\]
and
\[
\left.\frac{d\tilde{\delta}}{dR}\right|_{R = 1} = \epsilon \left(1 - \log \frac{1}{\epsilon}\right) \left(1 + \frac{1}{\log \frac{1}{\delta}}\right).
\]
The above inequality therefore implies that
for all $\epsilon \in (0, e^{-2}]$,
\[
\begin{split}
- \epsilon \left(1 + \frac{1}{\log \frac{1}{\delta}}\right) f_-'(\epsilon) & \le 2\left(\frac{\log \frac{1}{\delta}}{\log \frac{1}{\epsilon}} \left(\log \frac{1}{\epsilon} - 1\right) \left(1 + \frac{1}{\log \frac{1}{\delta}}\right) - \log \frac{1}{\delta}\right) E_\epsilon^\gamma(\mu) \\
& \quad + \frac{\gamma \epsilon}{2 - \gamma} + \frac{\pi \gamma^2}{2} + \frac{\pi \gamma^2}{2 \log \frac{1}{\epsilon}} \left(1 - \log \frac{1}{\epsilon}\right) \left(1 + \frac{1}{\log \frac{1}{\delta}}\right) \\
& = \left(\log \log \frac{1}{\epsilon} - 1\right) \left(\frac{2E_\epsilon^\gamma(\mu)}{\log \frac{1}{\epsilon}} - \frac{\pi \gamma^2}{2 \log \frac{1}{\epsilon} \log \frac{1}{\delta}}\right) + \frac{\epsilon \gamma}{2 - \gamma} \\
& \le \frac{C \log \log \frac{1}{\epsilon}}{\left(\log \frac{1}{\epsilon}\right)^2}
\end{split}
\]
for a constant $C = C(\gamma)$ by Proposition \ref{prop:core_construction}.
Hence
\[
f_-'(\epsilon) \ge - \frac{C \log \log \frac{1}{\epsilon}}{\epsilon \left(\log \frac{1}{\epsilon}\right)^2}.
\]
Note that
\[
\int_0^{e^{-2}} \frac{\log \log \frac{1}{\epsilon}}{\epsilon \left(\log \frac{1}{\epsilon}\right)^2} \, d\epsilon < \infty.
\]
Thus if we denote
\[
e_\gamma = \liminf_{\epsilon \searrow 0} f(\epsilon),
\]
then for any $\eta > 0$ we can find a number $\epsilon_0 > 0$
such that
\[
f(\epsilon_0) \le e_\gamma + \frac{\eta}{2}
\]
and at the same time,
\[
\int_\epsilon^{\epsilon_0} f_-'(s) \, ds \ge - \frac{\eta}{2}
\]
for any $\epsilon \in (0, \epsilon_0]$. It then follows that
\[
f(\epsilon) \le f(\epsilon_0) - \int_\epsilon^{\epsilon_0} f_-'(s) \, ds \le e_\gamma + \eta.
\]
Hence we have in fact
\[
e_\gamma = \lim_{\epsilon \searrow 0} f(\epsilon),
\]
as required.
\end{proof}

\section{Several walls}

We now consider a given $a \in A_N$ and $d \in \{\pm 1\}^N$
and we study magnetisations $m \in M(a, d)$.
In particular, we want to estimate $\inf_{M(a, d)} E_\epsilon$
and derive some inequalities for the magnetisation and the
stray field in terms of the energy excess
\[
E_\epsilon(m) - \inf_{M(a, d)} E_\epsilon.
\]

\subsection{Upper bound for the minimal energy}

The purpose of this section is to prove the following
upper bound for the energy by a direct construction. It is a generalization of Proposition \ref{prop:core_construction} to configurations with several walls $a\in A_N$.

\begin{proposition} \label{prop:construction}
For every $R \in (0, 1]$ there exists a constant $C_0>0$
such that for all $\epsilon \in (0, \frac{1}{2}]$, $a \in A_N$ with $\rho(a) \ge R$, and $d \in \{\pm 1\}^N$, the inequality
\[
\inf_{M(a, d)} E_\epsilon \le \frac{\pi}{2\log \frac{1}{\delta}} \sum_{n = 1}^N (d_n - \cos \alpha)^2 + \frac{C_0}{\left(\log \frac{1}{\delta}\right)^2}
\]
holds true.
\end{proposition}

\begin{proof}
We will in fact prove a more explicit estimate, showing that
there exists an $m \in M(a, d)$ with
\begin{equation} \label{eqn:construction_estimate}
E_\epsilon(m) \le \frac{\sum_{n = 1}^N (d_n - \cos \alpha)^2}{2\log \sqrt{\frac{R^2}{\delta^2} + 1}} \left(\pi + \frac{2}{\sin^2 \alpha \log \frac{1}{\epsilon}} \int_{-\infty}^\infty \frac{t^2}{(t^2 + 1)^2 \log (t^2 + 1)} \, dt\right).
\end{equation}
This will clearly imply the statement of the proposition.
The proof is similar to the proof of Proposition
\ref{prop:core_construction}.
Let $\gamma_n = d_n - \cos \alpha$. Define
\[
f(x_1) = \frac{\log (x_1^2 + \delta^2) - \log (R^2 + \delta^2)}{\log \delta^2 - \log (R^2 + \delta^2)}
\]
and $m_1:(-1, 1)\to [-1, 1]$, given by
\[
m_1(x_1) = \begin{cases}
\cos \alpha + \gamma_n f(x_1 - a_n) & \text{if $x_1 \in (a_n - R, a_n + R)$ for $n = 1, \ldots, N$,} \\
\cos \alpha & \text{else}.
\end{cases}
\]
Then there exists a function $m_2 : (-1,1) \to [-1, 1]$ such that
$m = (m_1, m_2) \in M(a, d)$.
Suppose that $x_1 \in (a_n - R, a_n + R)$.
If $d_n = -1$, then $1 - m_1(x_1) \ge 1 - \cos \alpha$ and
\[
1 + m_1(x_1) = (1 + \cos \alpha)\left(1 - f(x_1 - a_n)\right).
\]
If $d_n = 1$, then $1 + m_1(x_1) \ge 1 + \cos \alpha$ and
\[
1 - m_1(x_1) = (1 - \cos \alpha)\left(1 - f(x_1 - a_n)\right).
\]
In both cases,
\[
1 - (m_1(x_1))^2 \ge \sin^2 \alpha \left(1 - f(x_1 - a_n)\right).
\]
Thus as in the proof of Proposition \ref{prop:core_construction},
we can estimate
\[
\begin{split}
\int_{a_n - R}^{a_n + R} |m'|^2 \, dx_1 & \le \frac{2\gamma_n^2}{\sin^2 \alpha \log \sqrt{\frac{R^2}{\delta^2} + 1}} \int_{-R}^R \frac{x_1^2}{(x_1^2 + \delta^2)^2 \log \left(\frac{x_1^2}{\delta^2} + 1\right)} \, dx_1 \\
& \le \frac{2\gamma_n^2}{\delta \sin^2 \alpha \log \sqrt{\frac{R^2}{\delta^2} + 1}} \int_{-\infty}^\infty  \frac{t^2}{(t^2 + 1)^2 \log (t^2 + 1)} \, dt.
\end{split}
\]
Summing over $n$, we find
\[
\int_{-1}^1 |m'|^2 \, dt \le \frac{2\sum_{n = 1}^N \gamma_n^2}{\delta \sin^2 \alpha \log \sqrt{\frac{R^2}{\delta^2} + 1}} \int_{-\infty}^\infty  \frac{t^2}{(t^2 + 1)^2 \log (t^2 + 1)} \, dt.
\]

Now consider the function $u = U(m)$ as defined on page \pageref{def:U}.
Since $\curl \nabla^\perp u = 0$,
there exists a function $v \in \dot{W}^{1, 2}(\R_+^2)$ such
that $\nabla v = \nabla^\perp u$ in $\R_+^2$. Since this means that
$v' = m_1'$ on $(-1, 1) \times \{0\}$, we can choose $v$
such that $v = m_1 - \cos \alpha$ on $(-1, 1) \times \{0\}$.
Then we also have $v = 0$ on $(-\infty, -1) \times \{0\}$ and on
$(1, \infty) \times \{0\}$, and of course $\Delta v = 0$ in $\R_+^2$.
Furthermore, the function has finite Dirichlet energy, and it
follows that it is the unique minimiser of the Dirichlet energy
under these boundary conditions.

Define $w : \R_+^2 \to \R$ by
\[
w(a_n + r \cos \theta, r \sin \theta) = m_1(a_n + r) - \cos \alpha \quad \textrm{for } 0 < r \le R \text{ and } 0 \le \theta \le \pi, \ n=1, \dots, N,
\]
while $w = 0$ in $\Omega_R(a)$. Then we compute, similarly to the
proof of Proposition \ref{prop:core_construction}, that
\[
\begin{split}
\int_{B_R^+(a_n, 0)} |\nabla w|^2 \, dx & = \frac{\pi \gamma_n^2}{\left(\log \sqrt{\frac{R^2}{\delta^2} + 1}\right)^2} \int_0^R \frac{r^3}{(r^2 + \delta^2)^2} \, dr \\
& = \frac{\pi \gamma_n^2}{2\left(\log \sqrt{\frac{R^2}{\delta^2} + 1}\right)^2} \int_{\delta^2}^{R^2 + \delta^2} \frac{r - \delta^2}{r^2} \, dr \le \frac{\pi \gamma_n^2}{\log \sqrt{\frac{R^2}{\delta^2} + 1}}.
\end{split}
\]
Hence
\[
\int_{\R_+^2} |\nabla w|^2 \, dx \le \frac{\pi \sum_{n = 1}^N \gamma_n^2}{\log \sqrt{\frac{R^2}{\delta^2} + 1}}.
\]
In particular
\[
\int_{\R_+^2} |\nabla u|^2 \, dx = \int_{\R_+^2} |\nabla v|^2 \, dx \le \int_{\R_+^2} |\nabla w|^2 \, dx \le \frac{\pi \sum_{n = 1}^N \gamma_n^2}{\log \sqrt{\frac{R^2}{\delta^2} + 1}}.
\]
If we combine these inequalities, then we obtain
\eqref{eqn:construction_estimate}.
\end{proof}

\subsection{Stray field estimates}

The following is the main result of this section and
one of the key ingredients for the proof of Theorem \ref{thm1}.
We recall the function $u_{a, d}^*$ from
Section \ref{sect:reduced}, solving $\Delta u_{a, d}^*=0$ in
$\R_+^2$ and $\dd{u_{a, d}^*}{x_2} = 0$ on $(-\infty, -1) \times \{0\}$
and $(1, \infty) \times \{0\}$, and with a piecewise constant
trace on $(-1, 1) \times \{0\}$ given by the values
\[
\sigma_n = \frac{\pi}{2} \left(\sum_{k = n + 1}^N (d_k - \cos \alpha) - \sum_{k = 1}^n (d_k - \cos \alpha)\right).
\]
It has the property that for $n = 1, \ldots, N$, the function
\[
x \mapsto u_{a, d}^*(x) - (d_n - \cos \alpha) \left(\arctan \left(\frac{x_2}{x_1 - a_n}\right) - \frac{\pi(x_1 - a_n)}{2|x_1 - a_n|}\right)
\]
is harmonic in $B_{\rho(a)}^+(a_n, 0)$ and constant on
$(a_n - \rho(a), a_n + \rho(a)) \times \{0\}$.
Standard elliptic estimates then imply that this function is smooth
near $(a_n, 0)$. In view of the energy estimates
\eqref{eqn:reduced_energy1} and \eqref{eqn:reduced_energy2},
we can make more quantitative statements as well: if $\rho(a) \ge R > 0$, then
\begin{equation} \label{eqn:pointwise_u*}
\left|\dd{u_{a, d}^*}{x_1}(x) + \frac{(d_n - \cos \alpha) x_2}{(x_1 - a_n)^2 + x_2^2}\right| + \left|\dd{u_{a, d}^*}{x_2}(x) - \frac{(d_n - \cos \alpha) (x_1 - a_n)}{(x_1 - a_n)^2 + x_2^2}\right| \le C
\end{equation}
for $x \in B_{\rho(a)}^+(a_n, 0)$, where $C = C(\alpha, N, R)$.

\begin{theorem} \label{thm:stray_field}
For any $R \in (0, \frac{1}{2}]$ and $C_0 > 0$, there exists a constant $C_1>0$ such
the following holds true. Let $a \in A_N$ with $\rho(a) \ge R$
and $d \in \{\pm 1\}^N$. Set
\[
\Gamma = \sum_{n = 1}^N (d_n - \cos \alpha)^2.
\]
Suppose that $\epsilon \in (0, \frac{1}{2}]$ with $\delta \le R$ and
$m \in M(a, d)$ with
\begin{equation}
\label{aprior_est}
E_\epsilon(m) \le \frac{\pi \Gamma}{2\log \frac{1}{\delta}} + \frac{C_0}{\left(\log \frac{1}{\delta}\right)^2}.
\end{equation}
Let $u = U(m)$ be the function defined on page \pageref{def:U}. Then
\begin{equation}
\label{estim11}
\int_{\Omega_\delta(a)} \left|\nabla u - \frac{\nabla u_{a, d}^*}{\log \frac{1}{\delta}}\right|^2 \, dx \le \frac{C_1}{\left(\log \frac{1}{\delta}\right)^2}
\end{equation}
and
\begin{equation}
\label{estim12}
\int_{\Omega_\delta(a)} |\nabla u|^2 \, dx \ge \frac{\pi \Gamma}{\log \frac{1}{\delta}} - \frac{C_1}{\left(\log \frac{1}{\delta}\right)^2}.
\end{equation}
\end{theorem}

This statement is somewhat similar to Theorem \ref{thm:core_energy_estimate}.
The main difference, apart from the fact that we consider several N\'eel walls
here, is that we only assume a suitable bound for the energy, whereas in
Theorem \ref{thm:core_energy_estimate}, we study minimisers.
Before we prove the theorem, we establish the following
auxiliary result.

\begin{lemma} \label{lemma:W11_estimate}
Let $s > 0$ and $\mu \in W^{1, 2}(-s, s)$.
If $\mu(0) = 1$ and $|\mu| \le 1$, then
\[
\int_{-s}^s |\mu'| \, dx_1 \le 2s \int_{-s}^s \frac{(\mu')^2}{1 - \mu^2} \, dx_1.
\]
\end{lemma}

\begin{proof}
By the Cauchy-Schwarz inequality, we have
\begin{equation}
\label{Young_m1}
\int_{-s}^s |\mu'|\, dx_1 \le \left(\int_{-s}^s \frac{(\mu')^2}{1 - \mu^2}\, dx_1\right)^{1/2} \left(\int_{-s}^s (1 - \mu^2) \, dx_1\right)^{1/2}.
\end{equation}
Since $\mu(0) = 1$, any $x_1 \in (-s, s)$ will satisfy
\[
\begin{split}
1 - (\mu(x_1))^2 & = -2 \int_0^{x_1} \mu(t) \mu'(t) \, dt \\
& \le 2 \left|\int_{0}^{x_1} \frac{(\mu')^2}{1 - \mu^2} \, dt\right|^{1/2} \left|\int_{0}^{x_1} \mu^2 (1 - \mu^2)\, dt\right|^{1/2}.
\end{split}
\]
Integrating over $(-s, s)$, recalling that $|\mu| \le 1$, and using the Cauchy-Schwarz inequality, we obtain
\[
\begin{split}
\int_{-s}^s (1- \mu^2) \, dx_1 & \le 2 \left(\int_{-s}^s \left|\int_{0}^{x_1} \frac{(\mu')^2}{1 - \mu^2} \, dt\right| \, dx_1\right)^{1/2} \left(\int_{-s}^s \left|\int_{0}^{x_1} (1 - \mu^2) \, dt\right| \, dx_1\right)^{1/2} \\
& \le 2s \left(\int_{-s}^s \frac{(\mu')^2}{1 - \mu^2} \, dx_1\right)^{1/2} \left(\int_{-s}^s (1 - \mu^2) \, dx_1\right)^{1/2},
\end{split}
\]
which leads to
\[
\int_{-s}^s (1 - \mu^2) \, dx_1 \le 4s^2 \int_{-s}^s \frac{(\mu')^2}{1 - \mu^2} \, dx_1.
\]
Combining this with \eqref{Young_m1}, we obtain the desired inequality.
\end{proof}

\begin{proof}[Proof of Theorem \ref{thm:stray_field}]
It is clear that it suffices to prove the inequalities for
small values of $\epsilon$.
We modify the functions $u_{a, d}^*$
as follows: for a fixed $s \in (0, R]$, let $\xi_s \in \dot{W}^{1, 2}(\R_+^2)$
be such that
\[
\xi_s(x) = \frac{u_{a, d}^*(x)}{\log \frac{1}{\delta}}
\]
for $x \in \Omega_s(a)$ and
\[
\xi_s(a_n + r\cos \theta, r\sin \theta) = \frac{r u_{a, d}^*(a_n + s\cos \theta, s\sin \theta)}{s \log \frac{1}{\delta}} + \left(1 - \frac{r}{s}\right) \frac{\sigma_{n - 1} + \sigma_n}{2 \log \frac{1}{\delta}}
\]
for $0 < r < s$, $0 < \theta < \pi$, and $n = 1, \ldots, N$.
Then we have
\begin{equation} \label{eqn:xi_energy}
\int_{\R_+^2} |\nabla \xi_s|^2 \, dx \le \frac{\pi \Gamma \log \frac{1}{s} + C_1}{\left(\log \frac{1}{\delta}\right)^2}
\end{equation}
for a constant $C_1 = C_1(\alpha, N, R) > 0$.
This follows from \eqref{eqn:reduced_energy1}, \eqref{eqn:reduced_energy2},
and \eqref{eqn:pointwise_u*}.

We observe that
\[
\begin{split}
\frac{\pi \Gamma}{\log \frac{1}{\delta}} & = \int_{-1}^1 \frac{u_{a, d}^*(x_1, 0)}{\log \frac{1}{\delta}} m_1'(x_1) \, dx_1 \\
& = \int_{-1}^1 \xi_s(x_1, 0) m_1'(x_1) \, dx_1 - \int_{-1}^1 \left(\xi_s(x_1, 0) - \frac{u_{a, d}^*(x_1, 0)}{\log \frac{1}{\delta}}\right) m_1'(x_1) \, dx_1 \\
& = \int_{\R_+^2} \nabla \xi_s \cdot \nabla u \, dx - \int_{-1}^1 \left(\xi_s(x_1, 0) - \frac{u_{a, d}^*(x_1, 0)}{\log \frac{1}{\delta}}\right) m_1'(x_1) \, dx_1.
\end{split}
\]
We have a constant $C_2 = C_2(\alpha, N, R) > 0$ such that
\[
\left|\int_{-1}^1 \left(\xi_s(x_1, 0) - \frac{u_{a, d}^*(x_1, 0)}{\log \frac{1}{\delta}}\right) m_1'(x_1) \, dx_1\right| \le \frac{C_2}{\log \frac{1}{\delta}} \sum_{n = 1}^N \int_{a_n - s}^{a_n + s} |m_1'| \, dx_1.
\]
Thus by Lemma \ref{lemma:W11_estimate},
\[
\left|\int_{-1}^1 \left(\xi_s(x_1, 0) - \frac{u_{a, d}^*(x_1, 0)}{\log \frac{1}{\delta}}\right) m_1'(x_1) \, dx_1\right| \le \frac{2C_2 s}{\epsilon \log \frac{1}{\delta}} \left(2E_\epsilon(m) - \|\nabla u\|_{L^2(\R_+^2)}^2\right).
\]
We conclude that
\begin{equation} \label{eqn:key_estimate}
\frac{\pi \Gamma}{\log \frac{1}{\delta}} \le \frac{2C_2 s}{\epsilon \log \frac{1}{\delta}} \left(2E_\epsilon(m) - \|\nabla u\|_{L^2(\R_+^2)}^2\right) + \int_{\R_+^2} \nabla \xi_s \cdot \nabla u \, dx.
\end{equation}
Using the Cauchy-Schwarz inequality and \eqref{eqn:xi_energy}, we obtain
\begin{equation} \label{eqn:key_estimate2}
\frac{\pi \Gamma}{\log \frac{1}{\delta}} \le \frac{2C_2 s}{\epsilon \log \frac{1}{\delta}} \left(2E_\epsilon(m) - \|\nabla u\|_{L^2(\R_+^2)}^2\right) + \frac{\sqrt{\pi \Gamma \log \frac{1}{s} + C_1}}{\log \frac{1}{\delta}} \|\nabla u\|_{L^2(\R_+^2)}.
\end{equation}

We want to use this inequality to prove \eqref{estim12} first.
For this purpose, we choose $s \in (0, R]$ such that
\begin{equation} \label{eqn:definition_s}
\|\nabla u\|_{L^2(\R_+^2)}^2 = \frac{\pi \Gamma}{\log \frac{1}{s}} - \frac{2C_1}{\left(\log \frac{1}{s}\right)^2}.
\end{equation}
This is possible
whenever $\epsilon$ is sufficiently small because of \eqref{aprior_est}. Then \eqref{eqn:key_estimate2} and \eqref{aprior_est} imply
\[
\begin{split}
\frac{\pi \Gamma}{\log \frac{1}{\delta}} & \le \frac{2C_2 s}{\epsilon \log \frac{1}{\delta}} \left(\frac{\pi \Gamma}{\log \frac{1}{\delta}} - \frac{\pi \Gamma}{\log \frac{1}{s}} + \frac{2C_0}{\left(\log \frac{1}{\delta}\right)^2} + \frac{2C_1}{\left(\log \frac{1}{s}\right)^2}\right) \\
& \quad + \frac{\sqrt{\pi^2 \Gamma^2 \left(\log \frac{1}{s}\right)^2 - C_1\pi \Gamma \log \frac{1}{s} - 2C_1^2}}{\log \frac{1}{\delta} \log \frac{1}{s}}\\
& \le \frac{2C_2 s}{\epsilon \log \frac{1}{\delta}} \left(\frac{\pi \Gamma}{\log \frac{1}{\delta}} - \frac{\pi \Gamma}{\log \frac{1}{s}} + \frac{2C_0}{\left(\log \frac{1}{\delta}\right)^2} + \frac{2C_1}{\left(\log \frac{1}{s}\right)^2}\right) \\
& \quad + \frac{\pi \Gamma}{\log \frac{1}{\delta}} \left(1 - \frac{C_1}{2 \pi \Gamma \log \frac{1}{s}} - \frac{C_1^2}{\pi^2 \Gamma^2\left(\log \frac{1}{s}\right)^2}\right),
\end{split}
\]
because $\sqrt{1-\beta}\leq 1-\frac \beta 2$ for $\beta\in (0,1)$. That is,
\[
\frac{C_1}{2\pi \Gamma} + \frac{C_1^2}{\pi^2 \Gamma^2 \log \frac{1}{s}} \le \frac{2C_2 s}{\epsilon \log \frac{1}{\delta}} \left(\log \frac{1}{s} - \log \frac{1}{\delta} + \frac{2C_0 \log \frac{1}{s}}{\pi \Gamma \log \frac{1}{\delta}} + \frac{2C_1 \log \frac{1}{\delta}}{\pi \Gamma \log \frac{1}{s}}\right).
\]
In particular, there exist certain constants $C_3, C_4, C_5, C_6$,
all of them positive and depending only on $\alpha$, $N$, $C_0$, and $R$, such that
\begin{equation} \label{eqn:inequality_for_s}
C_3 + \frac{C_4}{\log \frac{1}{s}} \le \frac{s}{\delta} \left(\log \frac{\delta}{s} + \frac{C_5 \log \frac{1}{s}}{\log \frac{1}{\delta}} + \frac{C_6 \log \frac{1}{\delta}}{\log \frac{1}{s}}\right).
\end{equation}
We want to use this inequality to show that there exists a
constant $C_7 = C_7(\alpha, R, N, C_0) > 0$ such that $C_7 s \ge \delta$.
Then \eqref{estim12}
follows immediately from \eqref{eqn:definition_s}, because the right hand side
is increasing in $s$ when $\log \frac 1 s\geq 4C_1/(\pi \Gamma)$, which is the case for $\eps$ small due to \eqref{aprior_est}.

To this end, let $c = \min\{1, \frac{C_3}{4C_6}\}$. We distinguish three cases.
\begin{description}
\item[Case 1.] If $s \ge c\delta$, then the claim is obvious.
\item[Case 2.] If $s < c\delta$ and
\[
s \log \frac{1}{s} \ge \frac{C_3}{2C_5} \delta  \log \frac{1}{\delta},
\]
then it follows that $s \ge \delta^2$, provided
that $\epsilon$ is small enough. (Otherwise we would have an
immediate contradiction to the assumptions for this case.)
Hence $\log \frac{1}{s} \le 2\log \frac{1}{\delta}$
and
\[
s \ge \frac{C_3 \delta}{4C_5}.
\]
\item[Case 3.] If $s < c\delta$ and
\[
s \log \frac{1}{s} < \frac{C_3}{2C_5} \delta \log \frac{1}{\delta},
\]
then we obtain
\[
\frac{C_3}{2} + \frac{C_4}{\log \frac{1}{s}} \le \frac{s}{\delta} \left(\log \frac{\delta}{s}  + \frac{C_6 \log \frac{1}{\delta}}{\log \frac{1}{s}}\right)
\]
from \eqref{eqn:inequality_for_s}.
We also have $\log \frac{1}{s} > \log \frac{1}{\delta}$
(since $c \le 1$). Hence
\[
\frac{s \log \frac{1}{\delta}}{\delta \log \frac{1}{s}} \le c \le \frac{C_3}{4C_6}.
\]
Hence
\[
\frac{C_3}{4} \le \frac{s}{\delta} \log \frac{\delta}{s},
\]
which implies the claim.
\end{description}
This concludes the proof of \eqref{estim12}.

Now we go back to inequality \eqref{eqn:key_estimate} and use it for $s = \delta$
in order to prove
\eqref{estim11}. Since we now have \eqref{estim12}, the inequality
implies that
\begin{equation} \label{eqn:xi_estimate2}
\int_{\R_+^2} \nabla \xi_\delta \cdot \nabla u \, dx \ge \frac{\pi \Gamma}{\log \frac{1}{\delta}} - \frac{C_8}{\left(\log \frac{1}{\delta}\right)^2}
\end{equation}
for a constant $C_8 = C_8(\alpha, N, C_0, R)$. Hence
\[
\begin{split}
\int_{\R_+^2} |\nabla u - \nabla \xi_\delta|^2 \, dx & \le 2E_\epsilon(m) - 2\int_{\R_+^2} \nabla \xi_\delta \cdot \nabla u \, dx + \int_{\R_+^2} |\nabla \xi_\delta|^2 \, dx \\
& \le \frac{2C_0 + 2C_8 + C_1}{\left(\log \frac{1}{\delta}\right)^2}
\end{split}
\]
by \eqref{aprior_est}, \eqref{eqn:xi_energy} and \eqref{eqn:xi_estimate2}.
Since $\xi_\delta$ coincides with $u_{a, d}^*/\log \frac{1}{\delta}$
in $\Omega_\delta(a)$, this finally implies \eqref{estim11}.
\end{proof}

\begin{remark}
\label{rem:non_min}
Once we have \eqref{aprior_est} and \eqref{estim12} in Theorem \ref{thm:stray_field},
we can also derive the inequalities
$$\eps \int_{-1}^1 (\varphi')^2\, dx_1\leq \frac{C}{\left(\log \frac{1}{\delta}\right)^2} \quad \textrm{and}\quad \int_{a_n-\delta}^{a_n+\delta} \sin^2 \varphi\, dx_1\leq 
\frac{C\delta}{\log \frac{1}{\delta}}$$
for a lifting $\varphi$ of $m$ and for $n=1, \dots, N$, where $C = C(\alpha, N, R, C_0)$.
The first inequality is an immediate consequence of \eqref{aprior_est} and \eqref{estim12},
and the second follows with the same arguments as in the proof of
Lemma \ref{lemma:core_estimate}.
These estimates are similar to \eqref{eqn:core_estimate1} and \eqref{eqn:core_estimate2}, but now we know that they hold true for non-minimizing configurations (under the energy control \eqref{aprior_est}) and the Pohozaev identity previously used for the proof of \eqref{eqn:core_estimate1} and \eqref{eqn:core_estimate2} is no longer needed.
\end{remark}

\begin{remark} \label{rem:magnetostatic_core_energy}
Inequalities \eqref{aprior_est} and \eqref{estim12} further imply
that
\[
\int_{B_\delta^*(a)} |\nabla u|^2 \, dx \le \frac{2C_0 + C_1}{\left(\log \frac{1}{\delta}\right)^2}.
\]
If we combine this estimate with \eqref{estim11}, then we also obtain
\[
\int_{B_r^*(a)} |\nabla u|^2 \, dx \le \frac{1}{\left(\log \frac{1}{\delta}\right)^2} \left(2C_0 + 3C_1 + 2\int_{B_r^*(a) \backslash B_\delta^*(a)} |\nabla u_{a, d}^*|^2 \, dx\right)
\]
for any $r > 0$. Furthermore, since $u_{a, d}^*$ is known quite
explicitly, the last integral is typically not too difficult to estimate.
\end{remark}

\section{Proof of the main result}

We now prove Theorem \ref{thm1}. To this end, fix
$a \in A_N$ and $d \in \{\pm 1\}^N$. Set $\gamma_n = d_n - \cos \alpha$
for $n = 1, \ldots, N$ and
\[
\Gamma = \sum_{n = 1}^N \gamma_n^2.
\]
Furthermore, set $\gamma_\pm = 1 \mp \cos \alpha$ and
recall Definition \ref{def:e}, which introduces the function $e : \{\pm 1\} \to \R$ with
\[
e(\pm 1) = \lim_{\epsilon \searrow 0} \left(\left(\log \frac{1}{\delta}\right)^2 \inf_{M_{\gamma_\pm}} E_\epsilon^{\gamma_\pm} - \frac{\pi \gamma_\pm^2}{2} \log \frac{1}{\delta}\right).
\]

We divide the identity from Theorem \ref{thm1} into
two inequalities, which are proved in Section \ref{sect:lower} and
Section \ref{sect:upper}, respectively, after some preparation.
Throughout the proof, we indiscriminately write $C$ for various positive
constants that depend only on $\alpha$, $N$, $a$, $d$, and occasionally
on the exponents of $L^p$-spaces appearing in the context
(always denoted by $p$ or $q$).

\subsection{Preparation}

Define
\[
w_0(x) = \arctan \left(\frac{x_2}{x_1}\right) - \frac{\pi x_1}{2|x_1|}
\]
for $x \in \R_+^2$. Recall the functions
\[
u_{a, d}^* = \sum_{n = 1}^N \gamma_n u_{a_n} \quad \text{and} \quad
\mu_{a, d}^* = \sum_{n = 1}^N \gamma_n \mu_{a_n}
\]
from Section \ref{sect:reduced}.
Consider a number $r \in (0, \rho(a)]$. For $n = 1, \ldots, N$, let
\[
\lambda_n = \gamma_n \log(2 - 2a_n^2) + \sum_{k \not= n} \gamma_k \mu_{a_k}(a_n)
\]
again and recall estimate \eqref{eqn:mu_estimate}, which implies that
\begin{equation} \label{eqn:mu_estimate_repeated}
\left|\mu_{a, d}^*(x_1) - \lambda_n - \gamma_n \log \frac{1}{|x_1 - a_n|}\right| \le Cr
\end{equation}
for $x_1 \in [a_n - r, a_n + r]$.
Also define
\[
\omega_n = \sum_{k \not= n} \gamma_k u_{a_k}(a_n).
\]
Then
\begin{equation} \label{eqn:pointwise1}
|u_{a, d}^*(x) - \omega_n - \gamma_n w_0(x_1 - a_n, x_2)| \le Cr \quad \textrm{in }\, B_r(a_n, 0)
\end{equation}
and
\begin{equation} \label{eqn:pointwise2}
|\nabla u_{a, d}^*(x) - \gamma_n \nabla w_0(x_1 - a_n, x_2)| \le C \quad \textrm{in }\, B_r(a_n, 0)
\end{equation}
for $1\leq n\leq N$, because $u_{a, d}^*(x) - \omega_n - \gamma_n w_0(x_1 - a_n, x_2)$
is a smooth function that vanishes at $(a_n, 0)$.

We now study how $\delta$ changes when we replace
$\epsilon$ by $\epsilon/r$ for a number $r \in (\epsilon, 1]$,
since we will have to rescale the magnetisation about the
centres of the N\'eel walls.
We have
\[
\log \left(\frac{\epsilon}{r} \log \frac{r}{\epsilon}\right) = \log \delta - \log r + \log\left(1 - \frac{\log r}{\log \epsilon}\right).
\]
Since $\log(1-\xi)\leq -\xi$ for $\xi\in (0,1)$, we obtain
\begin{equation} \label{eqn:rescaled_delta1}
\log \left(\frac{\epsilon}{r} \log \frac{r}{\epsilon}\right) \le \log \delta - \log r - \frac{\log r}{\log \epsilon}.
\end{equation}
Similarly, if $\epsilon$ is sufficiently small (for a fixed $r$),
then
\begin{equation} \label{eqn:rescaled_delta2}
\log \left(\frac{\epsilon}{r} \log \frac{r}{\epsilon}\right) \ge \log \delta - \log r - \frac{2\log r}{\log \epsilon}.
\end{equation}

\subsection{A lower bound for the interaction energy} \label{sect:lower}

The purpose of this section is to prove the inequality
\begin{equation} \label{eqn:lower_bound}
\liminf_{\epsilon \searrow 0} \left(\left(\log \frac{1}{\delta}\right)^2 \inf_{M(a, d)} E_\epsilon  - \frac{\pi \Gamma}{2} \log \frac{1}{\delta}\right) \\
\ge \sum_{n = 1}^N e(d_n) + W(a, d)
\end{equation}
for the function $W$ defined at the end of Section \ref{sect:reduced},
which amounts to half of the statement of Theorem \ref{thm1}.

\paragraph{First step: use minimisers}
Clearly it is sufficient to consider functions $m_\epsilon \in W^{1, 2}((-1, 1); \Ss^1)$
that minimise $E_\epsilon$ in $M(a, d)$.
Then
\begin{equation} \label{eqn:minimisers}
\limsup_{\epsilon \searrow 0} \left(\left(\log \frac{1}{\delta}\right)^2 E_\epsilon(m_\epsilon) - \frac{\pi \Gamma}{2} \log \frac{1}{\delta}\right) < \infty
\end{equation}
by Proposition \ref{prop:construction}.

\paragraph{Second step: prove convergence away from the walls}
This part of the proof is similar to the proof of Theorem \ref{thm:core_convergence}.
Let $\varphi_\epsilon \in W^{1, 2}(-1, 1)$ such that
$m_\epsilon = (\cos \varphi_\epsilon, \sin \varphi_\epsilon)$.
Define $u_\epsilon = U(m_\epsilon)$ and $v_\epsilon = u_\epsilon \log \frac{1}{\delta}$.
Then by Theorem \ref{thm:stray_field}, we have a sequence
$\epsilon_k \searrow 0$ such that
$v_{\epsilon_k} \rightharpoonup v$ weakly in $\dot{W}^{1, 2}(\Omega_r(a))$ for every $r > 0$
for some function
\[
v \in u_{a, d}^* + \dot{W}^{1, 2}(\R_+^2).
\]
In fact, for any fixed $r > 0$, owing to Lemma \ref{lemma:m1_estimate} (with $\Omega=\R^2_+$), Theorem \ref{thm:stray_field}, and Remark \ref{rem:non_min}, there exists
a number $\beta>0$ (depending on $r$) such that $|\sin \varphi_\epsilon|\geq \beta$
at distance at least $r$ away from $-1, a_1, \dots, a_N, 1$ for $\eps$ small enough.
For $r, R > 0$, define $\Sigma_{r, R}(a) = (\Omega_r(a) \cap B_R^+(0)) \backslash (B_r(-1, 0) \cup B_r(1, 0))$. Then we can use Lemma \ref{lemma:higher_estimates}
and standard elliptic estimates to obtain uniform estimates in
$W^{2, 2}(\Sigma_{r, R}(a))$ for any
$r, R > 0$ and any sufficiently small $\epsilon$. Therefore, we even have
$v_{\epsilon_k} \rightharpoonup v$ weakly in $W^{2, 2}(\Sigma_{r, R}(a))$
for all $r, R > 0$. Furthermore, we have $v_{\epsilon_k} \to v$ uniformly
in $\R_+^2 \backslash B_2(0)$ by standard estimates for the Laplace equation.
It follows that $\lim_{|x| \to \infty} v(x) = 0$.

Obviously $\Delta v = 0$ in $\R_+^2$ and $\dd{v}{x_2} = 0$
on $(-1, -\infty) \times \{0\}$ and on $(1, \infty) \times \{0\}$. By Lemma \ref{lemma:higher_estimates},
we also have
\[
\limsup_{\epsilon \searrow 0} \left(\left(\log \frac{1}{\delta}\right)^2 \epsilon \int_{a_n + r}^{a_{n + 1} - r} (\varphi_\epsilon'')^2  \, dx_1\right) < \infty
\]
for $n = 1, \ldots, N - 1$ and any $r > 0$,
and we have similar inequalities in $(r - 1, a_1 - r)$ and in $(a_N + r, 1 - r)$.
Since
\[
u_\epsilon' = \epsilon \varphi_\epsilon''/ \sin \varphi_\epsilon
\]
by \eqref{eqn:Euler-Lagrange},
we conclude that $v(\blank, 0)$ is locally constant in
$(-1, 1) \backslash \{a_1, \ldots, a_N\}$. But there is only
one function in the space $u_{a, d}^* + \dot{W}^{1, 2}(\R_+^2)$
with these properties (which can be seen with the
arguments from the proof of Theorem \ref{thm:core_convergence}),
and thus we have $$v = u_{a, d}^* \quad \textrm{in }\, \R^2_+.$$

Since $\dd{v_\epsilon}{x_2}(\blank, 0) = -m_{1\epsilon}' \log \frac{1}{\delta}$
on $(-1, 1) \backslash \{a_1, \ldots, a_N\}$, it also follows
that there exists a sequence $\epsilon_k \searrow 0$ such that
\[
(m_{1 {\epsilon_k}} - \cos \alpha) \log \frac{1}{\delta_k} \to \nu
\]
locally uniformly in $(-1, 1) \backslash \{a_1, \ldots, a_N\}$
for a function $\nu : (-1, 1) \to [-\infty, \infty]$
such that $\mu_{a, d}^* - \nu$ is locally constant in
$(-1, 1) \backslash \{a_1, \ldots, a_N\}$, where
$\delta_k = \epsilon_k \log \frac{1}{\epsilon_k}$.
With the same arguments as in the proof of Theorem \ref{thm:core_convergence},
we show that $\nu = \mu_{a, d}^*$ and
\[
(m_{1\epsilon} - \cos \alpha) \log \frac{1}{\delta} \to \mu_{a, d}^* \quad \textrm{locally uniformly in }\, (-1, 1) \backslash \{a_1, \ldots, a_N\}.
\]

Now for any $r \in (0, \rho(a)]$, we have
\begin{equation} \label{eqn:energy_away_from_vortices}
\begin{split}
\int_{\Omega_r(a)} |\nabla u_{a, d}^*|^2 \, dx \le \liminf_{\epsilon \searrow 0} \int_{\Omega_r(a)} |\nabla v_\epsilon|^2 \, dx 
 = \liminf_{\epsilon \searrow 0} \left(\left(\log \frac{1}{\delta}\right)^2 \int_{\Omega_r(a)} |\nabla u_\epsilon|^2 \, dx\right).
\end{split}
\end{equation}
Furthermore, by \eqref{eqn:mu_estimate_repeated}, we have
\begin{equation} \label{eqn:m1_boundary}
\left|m_{1\epsilon} (a_n \pm r) - \cos \alpha - \frac{\gamma_n \log \frac{1}{r} + \lambda_n}{\log \frac{1}{\delta}}\right| \le \frac{Cr}{\log \frac{1}{\delta}} + \frac{o(1)}{\log \frac{1}{\delta}}.
\end{equation}
Here and subsequently, we use the notation $o(1)$ for any quantity that converges
to $0$ as $\epsilon \searrow 0$, with a rate of convergence possibly
depending on $r$.

\paragraph{Third step: rescale the cores}
Fix $n\in \{1, \dots, N\}$ and $r\in (0, \rho(a)]$,
and define the functions $\tilde{m}_\epsilon:(-1,1)\to \Ss^1$
and $\tilde{u}_\epsilon:\R^2_+\to \R$ and the number $\tilde{\epsilon}$ by
\begin{align*}
\tilde{m}_\epsilon(x_1) & = m_\epsilon(rx_1 + a_n), \\
\tilde{u}_\epsilon(x) & = u_\epsilon(r x_1 + a_n, r x_2), \\
\tilde{\epsilon} & = \frac{\epsilon}{r}.
\end{align*}
Then we have
\begin{equation} \label{eqn:m_identity}
\tilde{\epsilon} \int_{-1}^1 |\tilde{m}_\epsilon'|^2 \, dx_1 = \epsilon \int_{a_n - r}^{a_n + r} |m_\epsilon'|^2 \, dx_1
\end{equation}
and
\begin{equation} \label{eqn:u_identity}
\int_{B_1^+(0)} |\nabla \tilde{u}_\epsilon|^2 \, dx = \int_{B_r^+(a_n, 0)} |\nabla u_\epsilon|^2 \, dx.
\end{equation}
Moreover, if we denote $\tilde{u}^*(x) = u_{a, d}^*(r x_1 + a_n, r x_2)$
and $\tilde{v}_\epsilon = \tilde{u}_\epsilon \log \frac{1}{\delta}$, then by
the observations in the second step, we have
$\tilde{v}_\epsilon - \tilde{u}^* \rightharpoonup 0$
weakly in $W^{2, 2}(B_1^+(0) \backslash B_{1/2}(0))$. In particular,
if we fix a number $q > 2$, then we have
strong convergence of the boundary data in $W^{1, q}(\partial^+ B_1^+(0))$.
Because of \eqref{eqn:pointwise2}, we have
\[
\|x \cdot \nabla \tilde{u}^*(x)\|_{L^\infty(\partial^+ B_1(0))} \le Cr.
\]
Hence
\[
\|x \cdot \nabla \tilde{u}_\epsilon\|_{L^q(\partial^+ B_1(0))} \le \frac{Cr + o(1)}{\log \frac{1}{\delta}}.
\]
From Theorem \ref{thm:stray_field}, Remark \ref{rem:magnetostatic_core_energy},
and inequality \eqref{eqn:pointwise2}, we also obtain the inequality
\[
\left\|\nabla \tilde{u}_\epsilon - \frac{\gamma_n x^\perp}{|x|^2 \log \frac{1}{\delta}}\right\|_{L^2(B_1^+(0) \backslash B_{\tilde{\delta}}(0))} + \|\nabla \tilde{u}_\epsilon\|_{L^2(B_{\tilde{\delta}}^+(0))} \le \frac{C}{\log \frac{1}{\delta}},
\]
where $\tilde{\delta} = \tilde{\epsilon} \log \frac{1}{\tilde{\epsilon}}$.
We then also obtain
\[
\|x \cdot \nabla \tilde{u}_\epsilon\|_{L^q(\partial^+ B_1(0))} \le \frac{Cr + o(1)}{\log \frac{1}{\tilde{\delta}}}
\]
and
\[
\left\|\nabla \tilde{u}_\epsilon - \frac{\gamma_n x^\perp}{|x|^2 \log \frac{1}{\tilde{\delta}}}\right\|_{L^2(B_1^+(0) \backslash B_{\tilde{\delta}}(0))} + \|\nabla \tilde{u}_\epsilon\|_{L^2(B_{\tilde{\delta}}^+(0))} \le \frac{C}{\log \frac{1}{\tilde{\delta}}}.
\]
Because of this and \eqref{eqn:m1_boundary}, we may apply
Corollary~\ref{cor:core_estimate} to $d_n \tilde{m}_{1\epsilon}$ with
\[
\gamma = d_n \gamma_n, \quad \eta=Cr+o(1), \quad  \text{and} \quad \zeta = d_n\left(\lambda_n + \gamma_n \log \frac{1}{r}\right) + Cr + o(1).
\]
In view of Definition \ref{def:e}, we conclude that
\begin{multline} \label{eqn:energy_lower_estimate}
\left(\log \frac{1}{\tilde{\delta}}\right)^2 \left(\tilde{\epsilon} \int_{-1}^1 |\tilde{m}_\epsilon'|^2 \, dx_1 + \int_{B_1^+(0)} |\nabla \tilde{u}_\epsilon|^2 \, dx\right) \\
\ge \pi \gamma_n^2 \log \frac{1}{\tilde{\delta}} + 2e(d_n) - 2\pi \gamma_n^2 \log \frac{1}{r} - 2\pi \gamma_n \lambda_n - Cr - o(1).
\end{multline}
We recall \eqref{eqn:rescaled_delta1} and \eqref{eqn:rescaled_delta2}.
Now \eqref{eqn:m_identity}, \eqref{eqn:u_identity}, and
\eqref{eqn:energy_lower_estimate} imply that
\begin{multline*}
\left(\log \frac{1}{\delta} - \log \frac{1}{r}\right)^2 \left(\epsilon \int_{a_n - r}^{a_n + r} |m_\epsilon'|^2 \, dx_1 + \int_{B_r^+(a_n, 0)} |\nabla u_\epsilon|^2 \, dx\right) - \pi \gamma_n^2 \left(\log \frac{1}{\delta} - \log \frac{1}{r}\right) \\
\ge 2e(d_n) - 2\pi \gamma_n^2 \log \frac{1}{r} - 2\pi \gamma_n \lambda_n - Cr - o(1).
\end{multline*}
Since this means in particular that
\[
\log \frac{1}{\delta} \left(\epsilon \int_{a_n - r}^{a_n + r} |m_\epsilon'|^2 \, dt + \int_{B_r^+(a_n, 0)} |\nabla u_\epsilon|^2 \, dx\right) \ge \pi \gamma_n^2 + o(1),
\]
and since we have \eqref{eqn:minimisers},
the inequality also yields
\begin{multline*}
\left(\log \frac{1}{\delta}\right)^2 \left(\epsilon \int_{a_n - r}^{a_n + r} |m_\epsilon'|^2 \, dt + \int_{B_r^+(a_n, 0)} |\nabla u_\epsilon|^2 \, dx\right) - \pi \gamma_n^2 \left(\log \frac{1}{\delta} - \log \frac{1}{r}\right) \\
\ge 2e(d_n) - 2\pi \gamma_n \lambda_n - Cr - o(1).
\end{multline*}

\paragraph{Fourth step: combine the estimates}
If we sum over $n$ and use \eqref{eqn:energy_away_from_vortices},
we obtain
\[
\left(\log \frac{1}{\delta}\right)^2 E_\epsilon(m_\epsilon) - \frac{\pi \Gamma}{2} \left(\log \frac{1}{\delta} - \log \frac{1}{r}\right) \ge \sum_{n = 1}^N (e(d_n) - \pi \gamma_n \lambda_n) + \frac{1}{2} \int_{\Omega_r(a)} |\nabla u_{a, d}^*|^2 \, dx - Cr - o(1).
\]
Thus
\[
\liminf_{\epsilon \searrow 0} \left(\left(\log \frac{1}{\delta}\right)^2 E_\epsilon(m_\epsilon) - \frac{\pi \Gamma}{2} \log \frac{1}{\delta}\right) 
\ge \sum_{n = 1}^N (e(d_n) - \pi \gamma_n \lambda_n) + \frac{1}{2} \int_{\Omega_r(a)} |\nabla u_{a, d}^*|^2 \, dx - \frac{\pi \Gamma}{2} \log \frac{1}{r} - Cr.
\]
Finally we let $r \searrow 0$. Recall that
\[
\frac{1}{2} \lim_{r \searrow 0} \left(\int_{\Omega_r(a)} |\nabla u_{a, d}^*|^2 \, dx - \pi \Gamma \log \frac{1}{r}\right) =E_{a, d}^*(u_{a, d}^*)= W_1(a, d)
\]
and
\[
- \pi \sum_{n = 1}^N \gamma_n \lambda_n = W_2(a, d)
\]
for the functions $W_1$ and $W_2$ defined in Section \ref{sec:rescale_tail}.
Hence we conclude that
\[
\liminf_{\epsilon \searrow 0} \left(\left(\log \frac{1}{\delta}\right)^2 E_\epsilon(m_\epsilon) - \frac{\pi \Gamma}{2} \log \frac{1}{\delta}\right) \ge \sum_{n = 1}^N e(d_n) + W(a, d).
\]
That is, inequality \eqref{eqn:lower_bound} is indeed satisfied.

\subsection{An upper bound for the interaction energy} \label{sect:upper}

We now want to prove the inequality
\begin{equation} \label{eqn:upper_bound}
\limsup_{\epsilon \searrow 0} \left(\left(\log \frac{1}{\delta}\right)^2 \inf_{M(a, d)} E_\epsilon  - \frac{\pi\Gamma}{2} \log \frac{1}{\delta}\right) \le \sum_{n = 1}^N e(d_n) + W(a, d),
\end{equation}
which complements \eqref{eqn:lower_bound}.

\paragraph{First step: glue energy minimising cores into the tail profile} 
Define
\[
\kappa_n^r = \frac{\lambda_n + \gamma_n {\log \frac 1 r}}{\gamma_n \log \frac{1}{\delta}}.
\]
Fix $r \in (0, \rho(a)]$
small enough so that $\kappa_n^r \in (0, 1)$ for sufficiently
small values of $\epsilon$.
Choose minimisers $\hat{\mu}_\epsilon^n \in M_{|\gamma_n|}$ of the
functionals $E_\epsilon^{|\gamma_n|}$ and let $\hat{u}_\epsilon^n$
be the solutions of the corresponding boundary value
problem \eqref{eqn:core_harmonic}--\eqref{eqn:core_boundary2} with
\[
\int_{B_1^+(0)\setminus B_{1/2}(0)} \hat{u}_\epsilon^n \, dx = 0.
\]
Define $\mu_\epsilon^n = d_n \hat{\mu}_\epsilon^n$ and
$u_\epsilon^n = d_n \hat{u}_\epsilon^n$.

Let $\eta \in C^\infty(\R)$ with $\eta \equiv 1$ in
$(- \infty, \frac{1}{2}]$ and $\eta \equiv 0$ in
$[\frac{3}{4}, \infty)$. Set
\[
\varphi_n(x_1) = \eta\left(\frac{|x_1 - a_n|}{r}\right) \quad \text{and} \quad \psi_n(x) = \eta\left(\frac{\sqrt{(x_1 - a_n)^2 + x_2^2}}{r}\right).
\]
Note that $\frac{\partial \psi_n}{\partial x_2}=0$ on $\R\times \{0\}$.
Now define
\[
\begin{split}
m_{1\epsilon}(x_1) & = \cos \alpha + \frac{\mu_{a, d}^*(x_1)}{\log \frac{1}{\delta}} \\
& \quad + \sum_{n = 1}^N \varphi_n(x_1) \left((1 - \kappa_n^r) \mu_{\epsilon/r}^n\left(\frac{x_1 - a_n}{r}\right) + \kappa_n^r d_n - \cos \alpha - \frac{\mu_{a, d}^*(x_1)}{\log \frac{1}{\delta}}\right) 
\end{split}
\]
and
\[
\tilde{u}_\epsilon(x) = \frac{u_{a, d}^*(x)}{\log \frac{1}{\delta}} + \sum_{n = 1}^N \psi_n(x) \left((1 - \kappa_n^r) u_{\epsilon/r}^n\left(\frac{(x_1 - a_n, x_2)}{r}\right) - \frac{u_{a, d}^*(x) - \omega_n}{\log \frac{1}{\delta}}\right).
\]
Then $m_{1\epsilon}(a_n) = d_n$ for $n = 1, \ldots, N$.
If $\epsilon$ is sufficiently small, we have $-1 \le m_{1\epsilon} \le 1$.
Hence there exists a function $m_{2\epsilon}$ such that
$m_\epsilon = (m_{1\epsilon}, m_{2\epsilon}) \in M(a, d)$.
Let $u_\epsilon = U(m_\epsilon)$.

\paragraph{Second step: estimate the magnetostatic energy in terms of $\tilde{u}_\epsilon$}
Since $u_{a, d}^*$ is harmonic on $\R^2_+$, we compute
\[
\begin{split}
\Delta \tilde{u}_\epsilon(x) & = \sum_{n = 1}^N \Delta \psi_n(x) \left((1 - \kappa_n^r) u_{\epsilon/r}^n \left(\frac{(x_1 - a_n, x_2)}{r}\right) - \frac{u_{a, d}^*(x) - \omega_n}{\log \frac{1}{\delta}}\right) \\
& \quad + 2\sum_{n = 1}^N \nabla \psi_n(x) \left(r^{-1} (1 - \kappa_n^r) \nabla u_{\epsilon/r}^n \left(\frac{(x_1 - a_n, x_2)}{r}\right) - \frac{\nabla u_{a, d}^*(x)}{\log \frac{1}{\delta}}\right).
\end{split}
\]
Let $\Sigma_n^r = B_{3r/4}^+(a_n, 0) \backslash B_{r/2}(a_n, 0)$.
Using Theorem \ref{thm:core_convergence} and the inequalities
\eqref{eqn:pointwise1} and \eqref{eqn:pointwise2}, we infer
\begin{equation} \label{eqn:v_estimate}
\left\|u_{\epsilon/r}^n\left(\frac{(x_1 - a_n, x_2)}{r}\right) - \frac{u_{a, d}^*(x) - \omega_n}{\log \frac{1}{\delta}}\right\|_{L^\infty(\Sigma_n^r)} \le \frac{Cr + o(1)}{\log \frac{1}{\delta}}
\end{equation}
and
\begin{equation} \label{eqn:v_derivatives}
\left\|\nabla u_{\epsilon/r}^n \left(\frac{(x_1 - a_n, x_2)}{r}\right) - \frac{r \nabla u_{a, d}^*(x)}{\log \frac{1}{\delta}}\right\|_{L^p(\Sigma_n^r)} \le \frac{Cr^{1 + 2/p} + o(1)}{\log \frac{1}{\delta}}
\end{equation}
for any $p < \infty$ and all $n = 1, \ldots, N$.
(Here $o(1)$ again stands for any quantity that converges to $0$
as $\epsilon \searrow 0$, with a rate of convergence that may
possibly depend on $r$.)
Therefore, we have
\[
\|\Delta \tilde{u}_\epsilon\|_{L^p(\R_+^2)} \le \frac{C r^{2/p - 1} + o(1)}{\log \frac{1}{\delta}}.
\]

We also compute
\[
\begin{split}
\dd{\tilde{u}_\epsilon}{x_2}(x_1, 0) & = - \frac{\frac{d}{dx_1} \mu_{a, d}^*(x_1)}{\log \frac{1}{\delta}} - \sum_{n = 1}^N \varphi_n(x_1)
 \left(\frac{1 - \kappa_n^r}{r} \frac{d\mu_{\epsilon/r}^n}{dx_1}\left(\frac{x_1 - a_n}{r}\right) - \frac{\frac{d}{dx_1} \mu_{a, d}^*(x_1)}{\log \frac{1}{\delta}}\right) \\
& = - m_{1\epsilon}'(x_1) + \sum_{n = 1}^N \varphi_n'(x_1) \left((1 - \kappa_n^r) \mu_{\epsilon/r}^n\left(\frac{x_1 - a_n}{r}\right) + d_n \kappa_n^r - \cos \alpha - \frac{\mu_{a, d}^*(x_1)}{\log \frac{1}{\delta}}\right).
\end{split}
\]
Using Theorem \ref{thm:core_convergence}, we then see that
\[
\left|\mu_\epsilon^n(x_1) - \cos \alpha + \frac{\gamma_n \log |x_1|}{\log \frac{1}{\delta}}\right| \le \frac{o(1)}{\log \frac{1}{\delta}}
\]
for $x_1 \in [-\frac{3}{4}, -\frac{1}{2}] \cup [\frac{1}{2}, \frac{3}{4}]$.
This, together with \eqref{eqn:mu_estimate_repeated}, implies that
\begin{equation} \label{eqn:mu_core}
\left|(1 - \kappa_n^r) \mu_{\epsilon/r}^n\left(\frac{x_1 - a_n}{r}\right) + d_n \kappa_n^r - \cos \alpha - \frac{\mu_{a, d}^*(x_1)}{\log \frac{1}{\delta}}\right| \le \frac{Cr + o(1)}{\log \frac{1}{\delta}}
\end{equation}
for any $x_1 \in [a_n - \frac{3r}{4}, a_n - \frac{r}{2}]
\cup [a_n + \frac{r}{2}, a_n + \frac{3r}{4}]$.

Recall that $u_\epsilon$ is the solution of
\begin{alignat*}{2}
\Delta u_\epsilon & = 0 & \quad & \text{in $\R_+^2$}, \\
\dd{u_\epsilon}{x_2} & = - m_{1\epsilon}' && \text{on $(-1, 1) \times \{0\}$}, \\
\dd{u_\epsilon}{x_2} & = 0 && \text{on $(-\infty, 1) \times \{0\}$ and $(1, \infty) \times \{0\}$}.
\end{alignat*}
Thus we have
\begin{equation} \label{eqn:Laplacian}
\|\Delta (u_\epsilon - \tilde{u}_\epsilon)\|_{L^p(\R_+^2)} \le \frac{Cr^{2/p - 1} + o(1)}{\log \frac{1}{\delta}}
\end{equation}
for an arbitrary (but fixed) $p \in (1, 2)$ and
\begin{equation} \label{eqn:boundary_derivative}
\left\|\dd{}{x_2} (u_\epsilon - \tilde{u}_\epsilon)\right\|_{L^\infty(\R)} \le \frac{C + o(1)}{\log \frac{1}{\delta}}.
\end{equation}
Also note that the support of $\Delta (u_\epsilon - \tilde{u}_\epsilon)$
is contained in $B_r^*(a)$ and the support of
$\dd{}{x_2} (u_\epsilon - \tilde{u}_\epsilon)$ is contained in
$\bigcup_{n = 1}^N (a_n - r, a_n + r)$.
Thus if $\mathcal{M}(\R^2)$ denotes the space of Radon measures
on $\R^2$, then after extending to $\R^2$ by an even reflection
on $\R \times \{0\}$,
we have
\[
\|\Delta (u_\epsilon - \tilde{u}_\epsilon)\|_{\mathcal{M}(\R^2)} \le \frac{Cr + o(1)}{\log \frac{1}{\delta}},
\]
whence
\[
\|\nabla (u_\epsilon - \tilde{u}_\epsilon)\|_{L^q(B_2^+(0))} \le \frac{Cr + o(1)}{\log \frac{1}{\delta}}
\]
for any fixed $q \in [1, 2)$. Theorem \ref{thm:core_energy_estimate}
and \eqref{eqn:v_estimate} imply that
\[
\|\nabla \tilde{u}_\epsilon\|_{L^q(B_2^+(0))} \le \frac{C + o(1)}{\log \frac{1}{\delta}}.
\]
It then follows that
\begin{equation} \label{eqn:Lq}
\|\nabla u_\epsilon\|_{L^q(B_2^+(0))} \le \frac{C + o(1)}{\log \frac{1}{\delta}}
\end{equation}
as well. We will use this inequality for
$q = \frac{2p}{3p - 2}$ in conjunction with
\eqref{eqn:Laplacian}.

Let
\[
\bar{u}_\epsilon = \fint_{B_1^+(0)} u_\epsilon \, dx.
\]
Then it follows that
\[
\begin{split}
\int_{\R_+^2} |\nabla u_\epsilon|^2 \, dx & = - \int_{(-1, 1) \times \{0\}} (u_\epsilon - \bar{u}_\epsilon) \dd{u_\epsilon}{x_2} \, dx_1 \\
& = \int_{\R_+^2} \nabla \tilde{u}_\epsilon \cdot \nabla u_\epsilon \, dx - \int_{(-1, 1) \times \{0\}} (u_\epsilon - \bar{u}_\epsilon) \dd{}{x_2} (u_\epsilon - \tilde{u}_\epsilon) \, dx_1 + \int_{\R_+^2} (u_\epsilon - \bar{u}_\epsilon) \Delta \tilde{u}_\epsilon \, dx.
\end{split}
\]
Now with the help of \eqref{eqn:boundary_derivative}, \eqref{eqn:Lq}, and the
continuous embedding $W^{1, q}(B_2^+(0)) \to L^{q/(2 - q)}(-1, 1)$,
we derive the estimate
\[
- \int_{(-1, 1) \times \{0\}} (u_\epsilon - \bar{u}_\epsilon) \dd{}{x_2} (u_\epsilon - \tilde{u}_\epsilon) \, dx_1 \le \frac{C (r^{2 - 2/q} + o(1))}{\left(\log \frac{1}{\delta}\right)^2}.
\]
(Note that for $q = \frac{2p}{3p - 2}$, we have
$2 - \frac{2}{q} = \frac{2}{p} - 1$.)
Furthermore, by \eqref{eqn:Laplacian}, \eqref{eqn:Lq}, and the Sobolev inequality,
\[
\int_{\R_+^2} (u_\epsilon - \bar{u}_\epsilon) \Delta \tilde{u}_\epsilon \, dx \le \frac{C (r^{2/p - 1} + o(1))}{\left(\log \frac{1}{\delta}\right)^2}.
\]
If we choose $p = \frac{4}{3}$, then we obtain
\[
\int_{\R_+^2} |\nabla u_\epsilon|^2 \, dx \le \int_{\R_+^2} |\nabla \tilde{u}_\epsilon|^2 \, dx + \frac{C\sqrt{r} + o(1)}{\left(\log \frac{1}{\delta}\right)^2}.
\]
Hence
\begin{equation} \label{eqn:energy1}
E_\epsilon(m_\epsilon) \le \frac{\epsilon}{2} \int_{-1}^1 |m_\epsilon'|^2 \, dt + \frac{1}{2} \int_{\R_+^2} |\nabla \tilde{u}_\epsilon|^2 \, dx + \frac{C\sqrt{r} + o(1)}{\left(\log \frac{1}{\delta}\right)^2}.
\end{equation}

\paragraph{Third step: estimate $\|\nabla \tilde{u}_\epsilon\|_{L^2(\R_+^2)}$}
We clearly have
\begin{equation} \label{eqn:energy2}
\int_{\Omega_r(a)} |\nabla \tilde{u}_\epsilon|^2 \, dx = \frac{1}{\left(\log \frac{1}{\delta}\right)^2} \int_{\Omega_r(a)} |\nabla u_{a, d}^*|^2 \, dx.
\end{equation}
In $B_r^+(a_n, 0)$ for $n = 1, \ldots, N$, we have
\[
\begin{split}
\nabla \tilde{u}_\epsilon(x) & = r^{-1} (1 - \kappa_n^r) \nabla u_{\epsilon/r}^n\left(\frac{(x_1 - a_n, x_2)}{r}\right) \\
& \quad + (\psi_n(x) - 1) \left(r^{-1} (1 - \kappa_n^r) \nabla u_{\epsilon/r}^n \left(\frac{(x_1 - a_n, x_2)}{r}\right) - \frac{\nabla u_{a, d}^*(x)}{\log \frac{1}{\delta}}\right) \\
& \quad + \nabla \psi_n(x) \left((1 - \kappa_n^r) u_{\epsilon/r}^n\left(\frac{(x_1 - a_n, x_2)}{r}\right) - \frac{u_{a, d}^*(x) - \omega_n}{\log \frac{1}{\delta}}\right).
\end{split}
\]
Thus using \eqref{eqn:v_estimate}, \eqref{eqn:v_derivatives} and Theorem \ref{thm:core_convergence}, and
observing that
\[
\|\nabla u_{\epsilon/r}^n\|_{L^2(B_1^+(0) \backslash B_{1/2}(0))}^2 \le \frac{C}{\left(\log \frac{1}{\delta}\right)^2}
\]
by Theorem \ref{thm:core_energy_estimate}, we obtain
\begin{equation} \label{eqn:energy3}
\|\nabla \tilde{u}_\epsilon\|_{L^2(B_r^+(a_n, 0))}^2 \le (1 - \kappa_n^r)^2 \|\nabla u_{\epsilon/r}^n\|_{L^2(B_1^+(0))}^2 + \frac{Cr + o(1)}{\left(\log \frac{1}{\delta}\right)^2}.
\end{equation}
Combining \eqref{eqn:energy2} and \eqref{eqn:energy3}, we now find that
\[
\int_{\R_+^2} |\nabla \tilde{u}_\epsilon|^2 \, dx \le \frac{1}{\left(\log \frac{1}{\delta}\right)^2} \int_{\Omega_r(a)} |\nabla u_{a, d}^*|^2 \, dx + \sum_{n = 1}^N (1 - 2\kappa_n^r) \int_{B_1^+(0)} |\nabla u_{\epsilon/r}^n|^2 \, dx + \frac{Cr + o(1)}{\left(\log \frac{1}{\delta}\right)^2}.
\]
Since
\[
\int_{B_1^+(0)} |\nabla u_{\epsilon/r}^n|^2 \, dx \ge \frac{\pi \gamma_n^2}{\log \frac{1}{\delta}} - \frac{C}{\left(\log \frac{1}{\delta}\right)^2}
\]
by Theorem \ref{thm:core_energy_estimate},
it follows that
\begin{multline} \label{eqn:energy4}
\int_{\R_+^2} |\nabla \tilde{u}_\epsilon|^2 \, dx \le \frac{1}{\left(\log \frac{1}{\delta}\right)^2} \int_{\Omega_r(a)} |\nabla u_{a, d}^*|^2 \, dx + \sum_{n = 1}^N \int_{B_1^+(0)} |\nabla u_{\epsilon/r}^n|^2 \, dx \\
- \sum_{n = 1}^N \frac{2\pi \gamma_n(\lambda_n + \gamma_n \log \frac{1}{r})}{\left(\log \frac{1}{\delta}\right)^2} + \frac{Cr + o(1)}{\left(\log \frac{1}{\delta}\right)^2}.
\end{multline}

\paragraph{Fourth step: estimate the exchange energy}
Note that we have $|m_{1\epsilon} - \cos \alpha| \le \frac{C}{\log \frac{1}{\delta}}$
in $(-1, a_1 - \frac{r}{2}]$, in $[a_n + \frac{r}{2}, a_{n + 1} - \frac{r}{2}]$
for $n = 1, \ldots, N$, and in $[a_N - \frac{r}{2}, 1)$ by
Theorem \ref{thm:core_convergence}. Moreover,
by \eqref{eqn:mu_core}, we have
\[
\begin{split}
d_n - m_{1\epsilon}(x_1) & = (1 - \kappa_n^r) \left(d_n - \mu_{\epsilon/r}^n\left(\frac{x_1 - a_n}{r}\right)\right) \\
&+ (1 - \varphi_n(x_1)) \left((1 - \kappa_n^r) \mu_{\epsilon/r}^n\left(\frac{x_1 - a_n}{r}\right) - \cos \alpha - \frac{\mu_{a, d}^*(x_1)}{\log \frac{1}{\delta}} + d_n \kappa_n^r\right) \\
& = \left(d_n - \mu_{\epsilon/r}^n\left(\frac{x_1 - a_n}{r}\right)\right) \left(1 - \frac{O(1)}{\log \frac{1}{\delta}}\right)
\end{split}
\]
and
\[
\begin{split}
d_n + m_{1\epsilon}(x_1) & = (1 - \kappa_n^r)\left(d_n + \mu_{\epsilon/r}^n\left(\frac{x_1 - a_n}{r}\right)\right) + 2 d_n \kappa_n^r \\
& \quad + (1 - \varphi_n(x_1)) \left((\kappa_n^r - 1) \mu_{\epsilon/r}^n\left(\frac{x_1 - a_n}{r}\right) -d_n \kappa_n^r+ \cos \alpha + \frac{\mu_{a, d}^*(x_1)}{\log \frac{1}{\delta}}\right) \\
& =\left(d_n + \mu_{\epsilon/r}^n\left(\frac{x_1 - a_n}{r}\right)\right) \left(1 - \frac{O(1)}{\log \frac{1}{\delta}}\right)
\end{split}
\]
in $(a_n - r, a_n + r)$. 
We also have
\[
\begin{split}
m_{1\epsilon}'(x_1) & = \frac{1 - \kappa_n^r}{r} \frac{d\mu_{\epsilon/r}^n}{dx_1}\left(\frac{x_1 - a_n}{r}\right) \\
& \quad + \varphi_n'(x_1)\left((1 - \kappa_n^r) \mu_{\epsilon/r}^n\left(\frac{x_1 - a_n}{r}\right) + d_n \kappa_n^r - \cos \alpha - \frac{\mu_{a, d}^*(x_1)}{\log \frac{1}{\delta}}\right) \\
& \quad - (1 - \varphi_n(x_1)) \left(\frac{1 - \kappa_n^r}{r} \frac{d\mu_{\epsilon/r}^n}{dx_1}\left(\frac{x_1 - a_n}{r}\right) - \frac{\frac{d}{dx_1} \mu_{a, d}^*(x_1)}{\log \frac{1}{\delta}}\right)
\end{split}
\]
near $a_n$.
We have
\[
\frac{d\mu_{a, d}^*}{dx_1}(x_1) = - \dd{u_{a, d}^*}{x_2}(x_1, 0).
\]
Hence defining
$T_n^r = (a_n - r, a_n - \frac{r}{2}) \cup (a_n + \frac{r}{2}, a_n + r)$, we have
\[
\left\|\frac{1 - \kappa_n^r}{r} \frac{d\mu_{\epsilon/r}^n}{dx_1}\left(\frac{x_1 - a_n}{r}\right) - \frac{\frac{d}{dx_1} \mu_{a, d}^*(x_1)}{\log \frac{1}{\delta}}\right\|_{L^2(T_n^r)} \le \frac{C\sqrt{r} + o(1)}{\log \frac{1}{\delta}}
\]
by Theorem \ref{thm:core_convergence} and \eqref{eqn:pointwise2}.
Recalling \eqref{eqn:mu_core}, we then compute
\begin{equation} \label{eqn:energy_core}
\frac{\epsilon}{2} \int_{-1}^1 |m_\epsilon'|^2 \, dt = \sum_{n = 1}^N \frac{\epsilon}{2r} \int_{-1}^1 \frac{(\frac{d}{dx_1}\mu_{\epsilon/r}^n)^2}{1 - (\mu_{\epsilon/r}^n)^2} \, dt + \frac{o(1)}{\left(\log \frac{1}{\delta}\right)^2}.
\end{equation}
Combining \eqref{eqn:energy1}, \eqref{eqn:energy4},
and \eqref{eqn:energy_core}, we now find
\[
E_\epsilon(m_\epsilon) \le \frac{1}{2 \left(\log \frac{1}{\delta}\right)^2}  \int_{\Omega_r(a)} |\nabla u_{a, d}^*|^2 \, dx + \sum_{n = 1}^N \inf_{M_{|\gamma_n|}} E_{\epsilon/r}^{|\gamma_n|} - \sum_{n = 1}^N \frac{\pi \gamma_n(\lambda_n + \gamma_n \log \frac{1}{r})}{\left(\log \frac{1}{\delta}\right)^2} + \frac{C \sqrt{r} + o(1)}{\left(\log \frac{1}{\delta}\right)^2}.
\]

\paragraph{Fifth step: estimate the core energy}
Recalling Definition \ref{def:e}, we see that
\[
\inf_{M_{|\gamma_n|}} E_{\epsilon/r}^{|\gamma_n|} \le \frac{\pi \gamma_n^2}{2\log \frac{1}{\tilde{\delta}}} + \frac{e(d_n) + o(1)}{\left(\log \frac{1}{\tilde{\delta}}\right)^2},
\]
where
\[
\tilde{\delta} = \frac{\epsilon}{r} \log \frac{r}{\epsilon}.
\]
Using the estimates \eqref{eqn:rescaled_delta1} and \eqref{eqn:rescaled_delta2},
we obtain
\[
\inf_{M_{|\gamma_n|}} E_{\epsilon/r}^{|\gamma_n|} \le \frac{\pi \gamma_n^2}{2\log \frac{1}{\delta}} + \frac{\pi \gamma_n^2 \log \frac{1}{r}}{2\left(\log \frac{1}{\delta}\right)^2} + \frac{e(d_n) + o(1)}{\left(\log \frac{1}{\delta}\right)^2}.
\]

\paragraph{Sixth step: combine the estimates}
It follows that
\begin{multline*}
\left(\log \frac{1}{\delta}\right)^2 E_\epsilon(m_\epsilon) \le \frac{1}{2} \int_{\Omega_r(a)} |\nabla u_{a, d}^*|^2 \, dx + \frac{\pi \Gamma}{2} \log \frac{1}{\delta} - \frac{\pi \Gamma}{2} \log \frac{1}{r} + \sum_{n = 1}^N e(d_n) + W_2(a, d) \\
+ C \sqrt{r} + o(1).
\end{multline*}
That is,
\begin{multline*}
\limsup_{\epsilon \searrow 0} \left(\left(\log \frac{1}{\delta}\right)^2 E_\epsilon(m_\epsilon) - \frac{\pi \Gamma}{2} \log \frac{1}{\delta}\right) \le \frac{1}{2} \int_{\Omega_r(a)} |\nabla u_{a, d}^*|^2 \, dx - \frac{\pi\Gamma}{2} \log \frac{1}{r} + \sum_{n = 1}^N e(d_n) + W_2(a, d) \\
+ C\sqrt{r}.
\end{multline*}
Letting $r \searrow 0$, we finally obtain
\[
\limsup_{\epsilon \searrow 0} \left(\left(\log \frac{1}{\delta}\right)^2 E_\epsilon(m_\epsilon) - \frac{d\pi\Gamma}{2} \log \frac{1}{\delta}\right) \le \sum_{n = 1}^N e(d_n) + W(a, d),
\]
which amounts to inequality \eqref{eqn:upper_bound}.
This completes the proof of Theorem \ref{thm1}.

\paragraph{Acknowledgement} RI acknowledges partial support from the ANR project ANR-14-CE25-0009-01.

\bibliographystyle{siam}
\bibliography{bib}

\end{document}